\documentclass[a4paper]{amsart}
\usepackage{amsmath,amsthm,amssymb}
\usepackage[T1]{fontenc}
\usepackage{url}
\usepackage{amsmath,amsthm,amssymb}
\usepackage{verbatim}
\usepackage{mathrsfs}
\usepackage{graphics}
\usepackage{graphicx}
\usepackage{subfigure}

\newtheorem{theorem}{Theorem}[section]
\newtheorem{proposition}[theorem]{Proposition}
\newtheorem{lemma}[theorem]{Lemma}
\newtheorem{corollary}[theorem]{Corollary}

\theoremstyle{remark}
\newtheorem{remark}[theorem]{Remark}
\newtheorem{definition}{Definition}
\newtheorem*{convention}{Convention}

\newtheorem*{notation}{Notation}

\numberwithin{equation}{section}

\newcommand{\Var}[2]{{\operatorname{Var}}\!\left.\!\left(#1\right)\right|_{#2}}
\newcommand{\ud}{\mathrm{d}}
\newcommand{\J}{I}
\newcommand{\bez}{\nopagebreak\hspace*{\fill}\nolinebreak$\Box$}
\newcommand{\vep}{\varepsilon}
\newcommand{\R}{{\mathbb{R}}}

\newcommand{\Z}{{\mathbb{Z}}}
\newcommand{\N}{{\mathbb{N}}}
\newcommand{\T}{{\mathbb{T}}}

\newcommand{\xbm}{(X,\mathcal{B},\mu)}

\newcommand{\ls}{\operatorname{LSSG}}
\newcommand{\sgn}{\operatorname{sgn}}
\newcommand{\var}{\operatorname{Var}}
\newcommand{\ac}{\operatorname{AC}}
\newcommand{\bv}{\operatorname{BV}}
\newcommand{\pl}{\operatorname{PL}}
\newcommand{\ol}{\operatorname{LG}}
\newcommand{\logs}{\operatorname{LSG}}
\newcommand{\mR}{Leb}
\newcommand{\Int}{\operatorname{Int}}

\newcommand{\SL}{\operatorname{SL}}

\newcommand{\lv}{\mathscr{LV}}
\newcommand{\bl}{\mathscr{L}}

\newcommand{\esssup}{\operatorname{ess\, sup}}
\newcommand{\Surf}{S}
\newenvironment{proofof}[2]{\begin{proof}[Proof of #1 \ref{#2}.]}{\end{proof}}

\begin{document}
\title[Ergodic properties of infinite extensions of area-preserving flows]
{Ergodic properties of infinite extensions of area-preserving
flows}
\author[K. Fr\k{a}czek \and C. Ulcigrai]{Krzysztof Fr\k{a}czek \and Corinna Ulcigrai}

\address{Faculty of Mathematics and Computer Science, Nicolaus
Copernicus University, ul. Chopina 12/18, 87-100 Toru\'n, Poland}
 \email{fraczek@mat.umk.pl}
\address{Department of Mathematics\\
University Walk, Clifton\\
Bristol BS8 1TW, United Kingdom}
\email{corinna.ulcigrai@bristol.ac.uk}
\date{}

\subjclass[2000]{ 37A40, 37C40}  \keywords{}
\thanks{Research partially supported by MNiSzW grant N N201
384834 and Marie Curie "Transfer of Knowledge" program, project
MTKD-CT-2005-030042 (TODEQ)}

\begin{abstract}
We consider volume-preserving flows $(\Phi^f_t)_{t\in\mathbb{R}}$
on $\Surf\times \mathbb{R}$, where $\Surf$ is a closed connected
surface of genus $g\geq 2$ and $(\Phi^f_t)_{t\in\mathbb{R}}$ has
the form $\Phi^f_t(x,y)=\left(\phi_tx,\
y+\int_0^{t}f(\phi_sx)\,ds\right)$ where
$(\phi_t)_{t\in\mathbb{R}}$ is a locally Hamiltonian flow of
hyperbolic periodic type  on $\Surf$  and $f$ is a  smooth real
valued function on $\Surf$.   We investigate ergodic properties of
these infinite measure-preserving flows and prove that if $f$
belongs to a space of finite codimension in
$\mathscr{C}^{2+\epsilon}(\Surf)$, then the following dynamical
dichotomy holds: if there is a fixed point of
$(\phi_t)_{t\in\mathbb{R}}$ on which $f$ does not vanish, then
$(\Phi^f_t)_{t\in\mathbb{R}}$ is ergodic, otherwise, if $f$
vanishes on all fixed points, it is reducible, i.e.~isomorphic to
the trivial extension $(\Phi^0_t)_{t\in\mathbb{R}}$.
% is reducible then if $f$ vanishes on all fixed points
% of  $(\phi)_{t\in\mathbb{R}}$, then $(\Phi^f_t)_{t\in\mathbb{R}}$
%is reducible, otherwise  $(\Phi^f_t)_{t\in\mathbb{R}}$ .
The proof of this result exploits the reduction of
$(\Phi^f_t)_{t\in\mathbb{R}}$ to a skew product automorphism over
an interval exchange transformation of periodic type.  If there is
a fixed point of $(\phi_t)_{t\in\mathbb{R}}$ on which $f$ does not
vanish, the reduction  yields cocycles with symmetric logarithmic
singularities, for which we prove  ergodicity.
%, for which we prove a corresponding dichotomy.
\end{abstract}

\maketitle
\section{Introduction}\label{introduction:sec}
In this paper we investigate ergodic properties  for a class of
infinite measure preserving extensions of area-preserving flows on
compact surfaces of higher genus. Let $(\Surf,\omega)$ be a
compact connected oriented symplectic smooth surface of genus
$g\geq 2$ and consider a symplectic flow
$(\phi_t)_{t\in\mathbb{R}}$ on $\Surf$ given by the vector field
$X$.
 Let $f:\Surf\to\R$ be a
$\mathscr{C}^{2+\epsilon}$-function. Following \cite{FL:ont} we
will consider a system of  coupled differential equations on
$\Surf\times\R$ of the form
\[\left\{\begin{array}{ccc}
\frac{dx}{dt}&=&X(x),\\
\frac{dy}{dt}&=&f(x),
\end{array}\right.\]
for $(x,y)\in \Surf\times\R$. The flow given by these equations is a
skew-product extension of $(\phi_t)_{t\in\mathbb{R}}$ which we
will denote by $(\Phi^f_t)_{t\in\mathbb{R}}$.

We  consider locally
Hamiltonian  flows $(\phi_t)_{t\in\mathbb{R}}$, which are
%(also known as flows given by a multi-valued Hamiltonian),
a natural class of symplectic flows (in dimension $2$ locally
Hamiltonian and  symplectic are both equivalent to area
preserving) introduced and studied by S.P. Novikov and his school
(see for example \cite{No, Zo:how} and also \cite{Arn} for the
toral case) and are also known as flows given by a multivalued
Hamiltonian. We now recall their definition.

Let $\eta$ be a  closed $1$-form on $\Surf$. Denote by
$\pi:\widehat{\Surf}\to \Surf$ the universal cover of $\Surf$ and by
$\widehat{\eta}$ the pullback of $\eta$ by $\pi:\widehat{\Surf}\to \Surf$.
Since $\widehat{\Surf}$ is simply connected and $\widehat{\eta}$ is
also a closed form, there exists a smooth function
$\widehat{H}:\widehat{\Surf}\to \R$, called a \emph{multivalued
Hamiltonian}, such that $d\widehat{H}=\widehat{\eta}$. We will
assume that  $\widehat{H}$ is a \emph{Morse function}. Denote by
$X:\Surf\to T\Surf$ the smooth vector field determined by
\[\eta=i_X\omega=\omega(X,\,\cdot\,).\]
Let $(\phi_t)_{t\in\R}$ stand for the smooth flow on $\Surf$
associated to the vector field $X$. Since $d\eta=0$, the flow
$(\phi_t)_{t\in\R}$ preserves the symplectic form $\omega$ and
hence it preserves the associated measure $\nu$ obtained by
integrating the form $\omega$. Moreover, it is by construction
locally Hamiltonian and it has finitely many fixed points,
which coincide with the image of the  critical points set of the
multivalued Hamiltonian $\widehat{H}$ by the map $\pi$. Denote by
$\Sigma$ the set of fixed points. Since we assume that $\widehat{H}$ is a Morse
function, the points in $\Sigma$ are either centers or
non-degenerate saddles. We will assume throughout that  the flow
has \emph{no saddle connections}, i.e. that there are no saddles which
belong to the closure of the same separatrix of the flow. This
assumption implies that the flow on $\Surf\backslash \Sigma$ is
\emph{minimal} (see \cite{Maier}) and that all points in $\Sigma$
are saddles.

Given a $\mathscr{C}^{2+\epsilon}$-function $f: \Surf\to\R$, the
extension $(\Phi^f_t)_{t\in\mathbb{R}}$ of the locally Hamiltonian
flow $(\phi_t)_{t\in\mathbb{R}}$ has the following form
\[\Phi^f_t(x,y)=\left(\phi_tx, \ y+\int_0^{t}f(\phi_sx)\,ds\right),\]
i.e. $(\Phi^f_t)_{t\in\R}$ is a \emph{skew product flow} over the
base flow $(\phi_t)_{t\in\R}$ on $\Surf$. In particular, it follows
that $(\Phi^f_t)_{t\in\R}$ preserves the \emph{infinite} product
measure $ \nu\times Leb$, where $\nu$ is the invariant measure
for $(\phi_t)_{t\in\R}$ and $Leb$ here is the Lebesgue measure on
${\R}$.
%The deviation of the cocycle $F$ was studied by Forni in \cite{Fo1}, \cite{Fo2} for typical $(\phi_t)_{t\in\R}$.

 A basic question in  ergodic theory is the description
of ergodic components.  Let us recall that a flow
$(\Phi_t)_{t\in\R}$ preserving a
 invariant measure $\mu$ (finite or infinite) is \emph{ergodic} if
for any measurable set $A$ which is invariant, i.e.~such that
$\mu(A)=\mu (\Phi_t A) $ for all $t\in \mathbb{R}$, either
$\mu(A)=0$ or $\mu(A^c)=0$ where $A^c$ denotes the complement. The
problem of ergodicity for locally Hamiltonian flows on compact
surfaces is well understood. A \emph{typical} locally Hamiltonian
flow $(\phi_t)_{t\in\R}$ on $\Surf$ with no saddle connection is
(uniquely) ergodic, by a celebrated theorem by  Masur and Veech
\cite{Masur, Ve1}. Moreover, mixing properties of locally
Hamiltonian flows have been investigated in \cite{Ko, Ko:mix,
Scheg, Ul:mix, Ul:wea, Ul:abs}. On the other hand, very little is
understood in the  case of non-compact extensions with the
exception of the special case of $g=1$ (see
\cite{FL:ont,Fr-Le:wm}) and the case where $f$ vanish on the set
of fixed points of the flow $(\phi_t)_{t\in\R}$ (see
\cite{Co-Fr,Fo1,Ma-Mo-Yo}).

In the setting of extensions, a property completely opposite to
ergodicity is \emph{reducibility}. Let us note that if $f=0$, the
phase space $\Surf\times \R$ for the corresponding trivial extension
given by $\Phi^0_t(x,y)=\left(\phi_t x,y\right)$ is foliated in
invariant sets of the form $\Surf\times \{y\}$, $y\in \R$. In this
sense, the dynamics is reduced to the dynamics of the surface flow
$(\phi_t)_{t\in\R}$. We say that $(\Phi^f_t)_{t\in\R}$ is
\emph{(topologically) reducible} if it is isomorphic to
$(\Phi^0_t)_{t\in\R}$ and the isomorphism $\mathbf{G}: \Surf\times \R
\rightarrow  \Surf\times \R$ is of the form $\mathbf{G}(x,y)= (x, y+
G(x))$, where $G: \Surf \rightarrow \R$  is continuous (and automatically its inverse $\mathbf{G}^{-1}(x,y)= (x, y-
G(x))$ is also continuous). In this case,
the phase space is again foliated into invariant sets for
$(\Phi^f_t)_{t\in\R}$ of the form $\{(x, y+ G(x)),\ x\in \Surf \}$, $\
y \in \R $. On each leaf the action of $(\Phi^f_t)_{t\in\R}$ is
conjugated to the one of $(\phi_t)_{t\in\R}$ on $\Surf$.

We will consider extensions of a special class of {ergodic} flows
$(\phi_t)_{t\in\R}$ on surfaces of genus $g\geq 2$.  For these
extensions, we will completely describe the ergodic behavior and
prove a dichotomy between ergodicity and reducibility.

Let us define the special class of locally Hamiltonian flows
$(\phi_t)_{t\in\R}$. Consider the foliation $\mathscr{F}$
determined by orbits of the locally Hamiltonian flow
$(\phi_t)_{t\in\R}$ on $\Surf$. The foliation $\mathscr{F}$ is a
singular foliation with simple saddles at the set $\Sigma$. It
comes equipped with a transverse measure $\nu_{\mathscr{F}}$, i.e.
a measure on arcs $\gamma$ transverse to the flow, given by
$\nu_{\mathscr{F}}(\gamma)= \int_\gamma \eta$. The pair
$(\mathscr{F}, \nu_{\mathscr{F}})$ is a \emph{measured foliation}
in the sense of Thurston (see \cite{Thurston, FLP}). We say that
$(\phi_t)_{t\in\R}$ is \emph{of periodic type} if there exists a
diffeomorphism $\Psi: \Surf \rightarrow \Surf$ which fixes the foliation
$\mathscr{F}$ and rescales the transverse measure, i.e.~there
exists $\rho<1$ such that
$\Psi(\nu_{\mathscr{F}})=\rho\,\nu_{\mathscr{F}}$ ($
\nu_{\mathscr{F}}( \Psi \circ \gamma) = \rho
\nu_{\mathscr{F}}(\gamma)$ for all transverse arcs $\gamma$).  For
example, $\Psi$ could be a pseudo-Anosov diffeomorphism such that
the stable foliation for $\Psi$ is the measured foliation
$(\mathscr{F}, \nu_{\mathscr{F}})$. Remark that flows of periodic
type have no saddle connections. The diffeomorphism $\Psi$ induces
a linear action $\Psi_*$ on the homology $H_1(\Surf, \mathbb{R})$. We
say that a locally  Hamiltonian flow $(\phi_t)_{t\in\R}$ is of
\emph{hyperbolic} periodic type if it is of periodic type and
additionally $\Psi_* : H_1(\Surf, \mathbb{R})\rightarrow  H_1(\Surf,
\mathbb{R})$ is hyperbolic, i.e.~all eigenvalues have  absolute
value different than one.

We can now state our main result. %Given a real-valued function $f$
%on $M$, let us denote by $\V(f)=\sum_{z\in \Sigma} |f(z)|$. Let
%$\mathscr{C}^{2+\epsilon}(M, \Sigma) = \{ f\in \mathscr{C}^{2+\epsilon}(M), \V(f)=0 \}$
%be the space of functions which vanish on $\Sigma$.
\begin{theorem}\label{mainHamiltonian}
Let $(\phi_t)_{t\in\R}$ be a locally Hamiltionian flow of
hyperbolic periodic type on a compact surface $\Surf$ of genus $g\geq
2$. There exists a closed $(\phi_t)_{t\in\R}$-invariant subspace
$K \subset \mathscr{C}^{2+\epsilon}(\Surf)$ with codimension $g$ in
$\mathscr{C}^{2+\epsilon}(\Surf)$, where $g$ is the genus of $\Surf$, such
that if $f \in K $ we have the following dichotomy:
\begin{itemize}
\item If $\sum_{z\in \Sigma} |f(z)|\neq 0$ then the extension $(\Phi^f_t)_{t\in\R}$ is
ergodic;
\item If $\sum_{z\in \Sigma} |f(z)| = 0$ then the extension $(\Phi^f_t)_{t\in\R}$ is
reducible.
\end{itemize}
Moreover, for every $f \in \mathscr{C}^{2+\epsilon}(\Surf)$ we can
write $f = f_K + f_\Sigma $ where $f_K \in K$ and $f_\Sigma$
vanishes on $\Sigma $ and belongs to a $g$ dimensional subspace of
$\mathscr{C}^{2+\epsilon}(\Surf, \Sigma)= \{ f\in
\mathscr{C}^{2+\epsilon}(\Surf), \sum_{z\in \Sigma} |f(z)|=0 \}$.
%there exists a projection operator $\Pf
%:\mathscr{C}^{2+\epsilon}(\Surf)\to K$ which has the form $\Pf (f)
%= f - \Delta\Pf(f)$ where the correction $\Delta\Pf(f) \in
%\mathscr{C}^{2+\epsilon}(\Surf,\Sigma)$.
\end{theorem}
Thus, in the setting of flows of periodic type there is an
infinite dimensional subspace of functions $f\in
\mathscr{C}^{2+\epsilon}(\Surf)$ on which we have a full
understanding of ergodic behavior of $(\Phi^f_t)_{t\in\R}$ and no
behavior other than ergodicity or reducibility can arise. We do
not have any results about ergodicity  when $f\notin K$. The space
$K$ will be defined as the kernel of finitely many invariant
$\mathscr{C}^{2+\epsilon}(\Surf)$-distributions. A similar space
arise also in the works by G. Forni \cite{Fo1, Fo2}, where it is
shown that in the context of area-preserving flows on surfaces
there are finitely many distributional obstructions to solve the
cohomological equation.

\subsection{Skew products over interval exchange transformations.}
A standard technique to study a flow on a surface is to choose a
transversal arc $\gamma$ on the surface and consider the
Poincar{\'e} first return map on the transversal. When the flow is
area-preserving, this map, in suitably chosen coordinates, is an
interval exchange transformation. The original flow
$(\phi_t)_{t\in\R}$ can be represented as a special flow over the
interval exchange transformation (see Definition
\ref{specialflowdef} below) and the study of the ergodic
properties of the surface flow are then reduced to the study of
the ergodic properties of the special flow. Similarly, choosing a
transversal surface of the form $\gamma \times \R$ one gets a two
dimensional section of  $\Surf\times \mathbb{R}$. In this case the
Poincar{\'e} map of the extension $(\Phi^f_t)_{t\in \mathbb{R}}$,
in suitable coordinates, is a  a skew product automorphism over an
interval exchange transformation. The main
Theorem~\ref{mainHamiltonian} will follow from a result  about
ergodicity for skew products  with logarithmic singularities over
interval exchange transformations (Theorem~\ref{mainIETs}). In
this section we recall basic definitions and  formulate the main
result in the setting of skew products. The relation with  the
main Theorem~\ref{mainHamiltonian} is explained in
\S\ref{outlinesec} (see Theorem~\ref{reductionthm}).

Interval exchange transformations (IETs) are a generalization of
rotations, well studied both as simple examples of dynamical
systems and in connection with flows on surfaces and
Teichm{\"u}ller dynamics (e.g.~see for an overview
\cite{ViB,YoLN,ZoFlat}).  To define an IET we adopt the notation
from \cite{ViB} introduced in \cite{Ma-Mo-Yo}. Let $\mathcal{A}$
be a $d$-element alphabet and let $\pi=(\pi_0,\pi_1)$ be a pair of
bijections $\pi_\vep:\mathcal{A}\to\{1,\ldots,d\}$ for $\vep=0,1$.
Let us consider
$\lambda=(\lambda_\alpha)_{\alpha\in\mathcal{A}}\in
\R_+^{\mathcal{A}}$, where $\R_+=(0,+\infty)$. Set
$|\lambda|=\sum_{\alpha\in\mathcal{A}}\lambda_\alpha$ and
$I=\left[0,|\lambda|\right)$ and
\begin{eqnarray}
&& I_{\alpha}=[l_\alpha,r_\alpha),\text{ where
}l_\alpha=\sum_{\pi_0(\beta)<\pi_0(\alpha)}\lambda_\beta,\;\;\;r_\alpha
=\sum_{\pi_0(\beta)\leq\pi_0(\alpha)}\lambda_\beta.\nonumber  \\
\nonumber && I'_{\alpha}=[l'_\alpha,r'_\alpha),\text{ where
}l'_\alpha=\sum_{\pi_1(\beta)<\pi_1(\alpha)}\lambda_\beta,\;\;\;r'_\alpha
=\sum_{\pi_1(\beta)\leq\pi_1(\alpha)}\lambda_\beta.\
\end{eqnarray}
The \emph{interval exchange transformation} $T =
T_{(\pi,\lambda)}$ given by the data $(\pi, \lambda )$ is the
orientation preserving piecewise isometry
$T_{(\pi,\lambda)}:[0,|\lambda|)\rightarrow[0,|\lambda|)$ which,
for each $\alpha \in \mathcal{A}$,  maps the interval $I_{\alpha}$
isometrically onto the interval $I'_{\alpha}$. Clearly $T$
preserves the Lebesgue measure on $I$. If $d=2$, the IET is a
rotation.

Each measurable function $\varphi:I \rightarrow \R$ determines a
{\em cocycle} $\varphi^{(\,\cdot\,)}$ %:\Z\times I \to \mathbb{R}$
for $T$ by the formula
\begin{equation}\label{cocycledef}
\varphi^{(n)}(x)=\left\{
\begin{array}{ccl}
\varphi(x)+\varphi(Tx)+\ldots+\varphi(T^{n-1}x) & \mbox{if} & n>0 \\
0 & \mbox{if} & n=0\\
-(\varphi(T^nx)+\varphi(T^{n+1}x)+\ldots+\varphi(T^{-1}x)) &
\mbox{if} & n<0,
\end{array}
\right.
\end{equation}
the function $\varphi$ will  be called a cocycle, as well.  We
also call $\varphi^{(n)}$ the $n^{th}$ \emph{Birkhoff sum} of
$\varphi$ over $T$.
 The {\it skew product} associated to the cocycle is
the map $T_\varphi: I\times \mathbb{R} \rightarrow  I\times
\mathbb{R} $
\[ T_{\varphi}(x,y)=(Tx,y+{\varphi}(x)). \]
Clearly $T_\varphi$ preserves the Lebesgue measure on $I\times
\mathbb{R}$. We will denote by $Leb$ the Lebesgue measure on $I$.

While there is large literature about cocycles for rotations (see
\cite{Aa-Le-Ma-Na,Conze,Fraczek,Le-Pa-Vo,Oren,Pask1,Pask2,Sch}),
very little is known in general about cocycles for IETs. Another
motivation to study skew products over IETs, in addition to
extensions of locally Hamiltonian flows, comes also from rational
billiards on non-compact spaces (for example the Ehrenfest
wind-tree model) and $\mathbb{Z}^d$-covers of translation surfaces
(see \cite{Gutkin}). The cocycles that arise in this setting are
piecewise constant functions with values in $\mathbb{Z}^d$. First
results in these geometric settings were only recently proved by
\cite{ConzeGutkin,HooperWeiss,HubertWeiss1,HubertWeiss2,
HubertLelievreTroubetzkoy}.

The class of skew products over IETs which we consider in this
paper are the ones that appear as Poincar{\'e} maps of extensions
of locally Hamiltonian flows on surfaces of genus $g\geq1$, which
typically  yield cocycles which have \emph{logarithmic
singularities}. Ergodicity in a particular case of extensions of
locally Hamiltonian flows which yield cocycles \emph{without}
logarithmic singularities was recently considered by the first
author and Conze in \cite{Co-Fr}. Cocycles with logarithmic
singularities have been previously investigated only over
rotations of the circle (see \cite{FL:ont,Fr-Le:wm}), which
correspond to surfaces of $g=1$.

Let $\{ \cdot \}$ denotes the fractional part, that is the
periodic function of period $1$ on $\mathbb{R}$ defined by $\{ x
\}=x$ if $0\leq x <1$.
\begin{definition}\label{LogSing}
We say that a cocycle $\varphi:I\to\R$ for an IET $T_{(\pi,
\lambda)}$ has \emph{logarithmic singularities} if there exists
constants $C_\alpha^+,C_\alpha^-\in \mathbb{R}$, $\alpha \in
\mathcal{A}$,  and $g_\varphi:I\to\R$ absolutely continuous on
each $I_\alpha$ with derivative of bounded variation, such that
\begin{equation}\label{fform}
\varphi(x)=-\sum_{\alpha\in\mathcal{A}}C^+_\alpha\log(|I|\{(x-l_\alpha)/|I|\}
- \sum_{\alpha\in\mathcal{A}} C^-_\alpha\log(|I|\{(r_\alpha-x)/|I|\}))+g_\varphi(x).
\end{equation}
We say that the logarithmic singularities are \emph{of geometric
type} if at least one among $ C_{\pi_0^{-1}(d)}^{-}$ and $
C_{\pi_1^{-1}(d)}^{-}$ is zero and  at least one among $
C_{\pi_0^{-1}(1)}^{+}$ or $C_{\pi_1^{-1}(1)}^{+}$ is zero. We
denote
%$$\bl(\varphi)=\sum_{\alpha\in\mathcal{A}}(|C_{\alpha}^+|+|C_{\alpha}^-|)$$ and
 by $\ol(\sqcup_{\alpha\in \mathcal{A}} I_{\alpha})$ the
space of  functions with logarithmic singularities of geometric
type.
\end{definition}
Cocycles in $\ol(\sqcup_{\alpha\in \mathcal{A}} I_{\alpha})$
appear naturally from extensions of locally Hamiltonian
flows\footnote{The condition on constants which are zero, which
seems rather technical, is automatically satisfied by functions
which have this geometric origin. This condition is used in the
proof of ergodicity (see Lemma~\ref{comparingconstants} and
Lemma~\ref{largederivatives}).}, see \S\ref{reduction:sec}. Notice
that the coefficients $C_\alpha^\pm$ can have different signs
(while if $\varphi\geq 0$ is the roof function of a special flow,
all constants $C_\alpha^\pm$  are non negative).

%The logarithmic singularities are called \emph{symmetric} if
%\begin{equation}\label{zerosym}
%\sum_{\alpha\in\mathcal{A},\pi_0(\alpha)\in\mathcal{O}}C^-_{\alpha}-
%\sum_{\alpha\in\mathcal{A},\pi_0(\alpha)-1\in\mathcal{O}}C^+_{\alpha}=0
%\end{equation}
%or every $\mathcal{O}\in\Sigma(\pi)$.
If $f \in\ol(\sqcup_{\alpha\in \mathcal{A}} I_{\alpha})$ has the
form (\ref{fform}) we say that the logarithmic singularities are
\emph{symmetric} if in addition the constants satisfy
\begin{equation}
\label{zerosymweak} \sum_{\alpha\in\mathcal{A}} C^-_{\alpha}-
\sum_{\alpha\in\mathcal{A}}C^+_{\alpha}=0.
\end{equation}
We will denote by $\logs(\sqcup_{\alpha\in \mathcal{A}}
I_{\alpha})$ the subspace of elements of $\ol(\sqcup_{\alpha\in
\mathcal{A}} I_{\alpha})$ which  have logarithmic symmetric
singularities.  The definition (\ref{zerosymweak}) of symmetry
appears often in the literature, for example in \cite{Ko, Scheg,
Ul:abs}. In this paper we need a more restrictive notion of
symmetry: we give in \S\ref{cocycles:sec}  the definition of
\emph{strong symmetric} logarithmic singularities (see Definition
\ref{strongsymmetry}) and we denote by $\ls(\sqcup_{\alpha\in
\mathcal{A}} I_{\alpha}) \subset \logs(\sqcup_{\alpha\in
\mathcal{A}} I_{\alpha})$ the corresponding space of functions
with strong symmetric logarithmic singularities of geometric type.
Even if the notion of  strong symmetric singularities is more
restrictive than (\ref{zerosymweak}), it   is  automatically
satisfied for functions which arise from extensions
of locally Hamiltonian flows (see \S\ref{extensionsspecialflows}).

We will restrict our attention to interval exchange transformation
\emph{of periodic type} (see \cite{Si-Ul}), which are analogous to
rotation whose rotation number is a quadratic irrational   (or
equivalently, has periodic continued fraction expansion). The
precise definition (also of hyperbolic periodic type) will be given
in \S\ref{IETs:sec} (Definitions~\ref{periodicIETdef}
and~\ref{defhypiet}). The class of hyperbolic periodic type IETs
arise as Poincar{\'e} maps of area-preserving flows
$(\phi_t)_{t\in\R}$ of hyperbolic periodic type.

Our main result in the context of skew products over IETs is the
following.
%Given $\varphi \in \ls(\sqcup_{\alpha\in \mathcal{A}} I_{\alpha})$,
%let $$\bl(\varphi)=\sum_{\alpha\in\mathcal{A}}(|C_{\alpha}^+|+|C_{\alpha}^-|)$$
\begin{theorem}\label{mainIETs}
Let $T$ be an interval exchange transformation of hyperbolic
periodic type. For every cocycle $\varphi$ for $T$ with $\varphi
\in \ls(\sqcup_{\alpha\in \mathcal{A}} I_{\alpha})$ such that $\bl(\varphi)\neq 0$ (i.e.~with at least
one logarithmic singularity)
%$\sum_{\alpha\in\mathcal{A}}(|C_{\alpha}^+|+|C_{\alpha}^-|)\neq 0$
there exists a correction function $\chi$, piecewise constant on
each $I_\alpha$, such that  the skew product $T_{\varphi-\chi}$ is
ergodic.
%the corrected cocycle $\varphi - \chi$ satisfy the following dichotomy:
%\begin{itemize}b%\item
%If $\bl(\varphi) \neq 0$ then
%\item If $\bl(\varphi)  = 0$ then ${\varphi-\chi}$ is
%a coboundary, i.e.~there exist a function $g:I\to \mathbb{R}$ such
%that $\varphi -\chi = g\circ T -g$ and $g $ is bounded.
%\end{itemize}
\end{theorem}
Let us remark that the correction $\chi$ belongs to a finite
dimensional space  and cocycles for which $\chi=0$ are the natural
counterpart, at the level of IETs, of the subspace $K$ in Theorem
\ref{mainHamiltonian}. A similar correction procedure was
introduced in \cite{Ma-Mo-Yo} to solve the cohomological equation
for IETs.

\subsection{Methods and outline}\label{outlinesec}
Let us first recall that definition of special flow and explain
how Theorem~\ref{mainHamiltonian} is related to Theorem
\ref{mainIETs}.
\begin{definition}\label{specialflowdef}
The special flow $T^\tau$ build over the base transformation
$T:(X,\mu)\to (X,\mu)$ and under the roof $\tau :X\to
\mathbb{R}_+$ is the quotient of the unit speed  flow
$v_t(x,y)=(x,y+t)$ on $X\times \mathbb{R}$ by the equivalence
relation $(x,y+ \tau^{(n)}(x))\sim (T^n(x), y) $, $n\in\Z$.
\end{definition}
%It is known that $(\phi_t)_{t\in\mathbb{R}}$  can be
%expressed as a special flow over a skew product over an IET. More precisely, we have
\begin{theorem}\label{reductionthm}
Let $f: \Surf\to\R$ be a $\mathscr{C}^{2+\epsilon}$-function and
$(\phi_t)_{t\in\mathbb{R}}$ be a locally Hamiltonian flow with no
saddle connections. The extension $(\Phi^f_t)_{t\in\mathbb{R}}$ is
measure-theoretically isomorphic to a special flow built over a
skew product $T_{\varphi_f}$ for an IET $T$ where $\varphi_f=
\varphi_f^1 + \varphi_f^2 $ and $\varphi_f^1 \in
\ls(\sqcup_{\alpha\in \mathcal{A}} I_{\alpha})$ and $\varphi_f^2$
is absolutely continuous on each $I_\alpha$ with $(\varphi_f^2)'
\in \ls(\sqcup_{\alpha\in \mathcal{A}} I_{\alpha})$.

If additionally we assume that  $(\phi_t)_{t\in\mathbb{R}}$  is a
locally Hamiltonian flow of hyperbolic periodic type, then we can choose $T$
to be an IET of hyperbolic periodic type  and $\varphi_f \in
\ls(\sqcup_{\alpha\in \mathcal{A}} I_{\alpha})$.
\end{theorem}
Theorem~\ref{reductionthm} allows to
reduce Theorem~\ref{mainHamiltonian} to Theorem~\ref{mainIETs}.
%The proof of Theorem~\ref{reductionthm} in given in \S\ref{reduction:sec}.
While the fact that  $(\Phi^f_t)_{t\in\mathbb{R}}$ can be reduced
to a skew product $T_{\varphi_f}$ where $\varphi_f$ has
logarithmic singularities is rather known, we need to show that
$\varphi_f$ has the precise form given in Theorem
\ref{reductionthm}\footnote{The reduction to  $\varphi_f \in
\ls(\sqcup_{\alpha\in \mathcal{A}} I_{\alpha})$ when
$(\phi_t)_{t\in\mathbb{R}}$  is of periodic type requires the
proof that when  the IET is of periodic type,  a cocycle as
$\varphi_f^2$  in Theorem \ref{reductionthm}, i.e.~absolutely
continuous on each $I_\alpha$ and with derivative $(\varphi_f^2)'
\in \ls(\sqcup_{\alpha\in \mathcal{A}} I_{\alpha})$,   is
cohomologous to a piecewise linear function (see Proposition
\ref{theorem:coblog}).}.

In order to prove ergodicity of the skew product in Theorem
\ref{mainIETs}, we use the technique of \emph{essential values},
which was developed by K.~Schmidt and J.-P.~Conze (see for example
\cite{Sch,Conze}). We recall all the definitions that we use in
\S\ref{essentialvalues:sec}. To control essential values, we
investigate the behavior of Birkhoff sums $\varphi^{(n)}$ (defined
in (\ref{cocycledef})) of a function $\varphi\in
\ol(\sqcup_{\alpha\in \mathcal{A}} I_{\alpha})$. As a standard
tool to study Birkhoff sums over IETs, we use  Rauzy-Veech
induction, a renormalization operator on the space of IETs first
developed by Rauzy and Veech in \cite{Ra, Ve1} (see
\S\ref{IETs:sec}). In order to prove ergodicity, we need to show
that the Birkhoff sums are
\emph{tight} %(Proposition~\ref{tightness})
and at the same time have enough \emph{oscillation}
%(Proposition~\ref{oscillation} based on Corollary~\ref{largederivative})
(in a sense which will made precise in \S\ref{ergodicity:sec}) on
a subsequence of partial rigidity times $(n_k)_{k \in \N}$ for the
IET (defined in \S\ref{oscillationsec}).

It is in order to achieve \emph{tightness} (see
Proposition~\ref{tightness}) that we need to \emph{correct} the
function $\varphi$ by a piecewise constant function $\chi$ (see
the statement of Theorem~\ref{mainIETs}). The idea of
\emph{correction} was introduced by Marmi, Moussa and Yoccoz in
order to solve the cohomological equation for IETs in the
breakthrough paper \cite{Ma-Mo-Yo}. The correction operator that
we use is closely related to the correction operator used by the
first author and Conze in \cite{Co-Fr}. The additional difficulty
that we have to face to achieve tightness is the presence of
logarithmic singularities. Here the assumption that the
singularities are \emph{symmetric} is crucial to exploit the
cancellation mechanism introduced by the second author in
\cite{Ul:abs} in order to show that locally Hamiltonian flows are
typically not mixing.

On the other hand the presence of logarithmic singularities helps
in order to prove that Birkhoff sums display enough
\emph{oscillation} (see Corollary~\ref{largederivative} and
Proposition~\ref{oscillations}). Our mechanism to achieve
oscillations is similar to the one used by the second author in
\cite{Ul:wea} to prove that locally Hamiltonian flows are
typically weakly mixing, with the novelty that in this context we
cannot exploit, as in \cite{Ul:wea}, that all constants
$C_\alpha^{\pm}$ are non-negative.

\subsubsection*{Structure of the paper.}
Let us outline the structure of the paper. In
\S\ref{essentialvalues:sec} we summarize the tools from the theory
of essential values that we will use to prove ergodicity. In
\S\ref{IETs:sec} we recall the definition of Rauzy-Veech induction
and  give the definition of IETs of periodic type. The
 definition of cocycles with strong symmetric logarithmic
singularities appears in \S\ref{cocycles:sec}, where  we also
prove basic properties of these cocycles.  In
\S\ref{renormalization:sec} we exploit Rauzy-Veech induction to
define a renormalization operator on cocycles in $ \ls$. In
\S\ref{symmetric:sec}   we formulate results on the growth of
Birkhoff sums based on the work of the first author in
\cite{Ul:abs}. The correction operator, which is crucial to define
the correction $\chi$ in Theorem \ref{mainIETs},  is constructed
in \S\ref{correction:sec}. In \S\ref{ergodicity:sec} we formulate
and prove the tightness and oscillation properties needed for
ergodicity and prove Theorem \ref{mainIETs}. The proof of
Theorem~\ref{mainHamiltonian} is given in \S\ref{reduction:sec}
and, as already mentioned, exploits the reduction via
Theorem~\ref{reductionthm}, which is also proved in
\S\ref{reduction:sec} (see also Appendix~\ref{singularities}).

\section{Preliminary material}\label{background:sec}
\subsection{Ergodicity of cocycles}\label{essentialvalues:sec}
We give here a brief overview of the tools needed to prove
ergodicity. For further background material concerning skew
products and infinite measure-preserving dynamical systems we
refer the reader to \cite{Aa} and \cite{Sch}.

%on $\varphi:I \rightarrow \R$ determines a
%{\em cocycle} $\varphi^{(\,\cdot\,)}:\Z\times I \to \mathbb{R}$
%for $T$ by the formula \begin{equation}\la
 Two cocycles $\varphi,\psi: X\to \R$ for $T:(X,\mu)\to(X,\mu)$ are called {\em
cohomologous} if there exists a measurable function $g:X\to \R$
(called the {\em transfer function}) such that
$\varphi=\psi+g-g\circ T$. If $\varphi$ and $\psi$ are
cohomologous then the corresponding skew products $T_\varphi$ and
$T_{\psi}$ are measure-theoretically isomorphic via the maps
$(x,y)\mapsto(x,y+g(x))$, where $g$ is a transfer function. A
cocycle $\varphi:X\to \R$ is a {\em coboundary} if it is
cohomologous to the zero cocycle.

 Denote by $\overline{\R}$ the one point compactification
of the group $\R$. An element $r\in \overline{\R}$
 is said to be an {\em essential value} of $\varphi$, if for each open
neighborhood $V_r$ of $r$ in $\overline{\R}$ and an arbitrary set
$B\in\mathcal{B}$, $\mu(B)>0$, there exists  $n\in\Z$ such that
\begin{eqnarray}
\mu(B\cap T^{-n}B\cap\{x\in X:\varphi^{(n)}(x)\in V_r\})>0.
\label{val-ess}
\end{eqnarray}
The set of essential values of $\varphi$ will be denoted by
$\overline{E}(\varphi)$. Let
${E}(\varphi)=\R\cap\overline{E}(\varphi)$. Then ${E}(\varphi)$ is
a closed subgroup of $\R$. We recall below some properties of
$\overline{E}(\varphi)$ (see \cite{Sch}).

\begin{proposition}[see \cite{Sch}]\label{proposition:schmidt}
Suppose that $T:(X,\mu)\to(X,\mu)$ is an ergodic automorphism. The
skew product $T_{\varphi}$ is ergodic if and only if
$E(\varphi)=\R$. The cocycle $\varphi$ is a coboundary if and only
if $\overline{E}(\varphi)=\{0\}$.
\end{proposition}

\smallskip Let $(X,d)$ be a compact metric space. Let
$\mathcal{B}$ stand for the $\sigma$--algebra of all Borel sets
and let $\mu$ be a probability Borel measure on $X$. For every
$B\in\mathcal{B}$ with $\mu(B)>0$ denote by $\mu_B$ the
conditional probability measure, i.e.\ $\mu_B(A)=\mu(A\cap
B)/\mu(B)$. Suppose that $T:\xbm\to\xbm$ is an ergodic
measure--preserving automorphism and there exist an increasing
sequence of natural numbers $(q_n)$ and a sequence of Borel sets
$(\Xi_n)$ such that \begin{equation}\label{rigsets}
\mu(\Xi_n)\to\delta>0,\;\;\mu(\Xi_n\triangle T^{-1}\Xi_n)\to
0\quad\mbox{ and }\quad \sup_{x\in \Xi_n}d(x,T^{q_n}x)\to
0.\end{equation}
 Let $\varphi:X\to \R$ be a
Borel integrable cocycle for $T$. Its mean value
$\int_X\varphi\,d\mu$ we will denote by $\mu(\varphi)$. Suppose
that $\mu(\varphi)=0$ and the sequence
$\left(\int_{\Xi_n}|\varphi^{(q_n)}(x)|d\mu(x)\right)_{n\in\N}$ is
bounded. As the the family of distributions
$\left\{(\varphi^{(q_n)})_*(\mu_{\Xi_n}):n\in\N\right\}$ is
uniformly tight, by passing to a further subsequence if necessary
we can assume that there exists a probability Borel measure $\nu$ on $\R$  such that
\[(\varphi^{(q_n)})_*(\mu_{\Xi_n})\to\nu\]
weakly in  the set of probability Borel measures on $\R$.

\begin{proposition}[see \cite{Co-Fr}]\label{limitmeasure}
The topological support of the measure $\nu$ is included in the
group $E(\varphi)$ of essential values  of the cocycle $\varphi$.
\end{proposition}

The following result is a general version of Proposition~12 in
\cite{Le-Pa-Vo}.

\begin{proposition}\label{ergodicity:criterium}
Let $\varphi:X\to\R$ be a cocycle such that
$\left(\int_{\Xi_n}|\varphi^{(q_n)}(x)|d\mu(x)\right)_{n\in\N}$ is
bounded, where $(\Xi_n), (q_n )$ and $\delta>0$ are as in
(\ref{rigsets}).  If there exists $0< c< \delta$ such that for all
$k$ large enough
\[\limsup_{n\to\infty}\left|\int_{\Xi_n}e^{2\pi i k \varphi^{(q_n)}(x)}\,d\mu(x)\right|
\leq c \]
then the skew product $T_\varphi$ is ergodic.
\end{proposition}

\begin{proof}
Let $e:\R\to\T$ stand for the character $e(x)=e^{2\pi ix}$.
Suppose that $\varphi$ is not ergodic, so by Proposition
\ref{proposition:schmidt}, $ E(\varphi)\neq \R$. Thus, since
$E(\varphi)$ is a closed subgroup, $E(\varphi)=r\Z$ for some
$r\in\R$. By Proposition~\ref{limitmeasure}, the limit measure
$\nu$ of the sequence
$\left((\varphi^{(q_n)})_*(\mu_{\Xi_n})\right)$ is concentrated on
$r\Z$, and hence $\nu$ is a discrete measure. It follows that the
measure $e_*\nu$ on $\T$ is as well a discrete measure and hence
it is a Dirichlet measure (see \cite{Ho-Me-Pa}). Therefore  one
has
\begin{equation}\label{dirichlet}
\limsup_{k\to\infty}\left|\int_{\R}e^{2\pi i
kt}\,d\nu(t)\right|=
\limsup_{k\to\infty}\left|\int_{\T}z^k\,d(e_*\nu)(z)\right|=\limsup_{k\to\infty}\left|\widehat{e_*\nu}(k)\right|=1.
\end{equation}
By assumption, there exists $k_0$ such that
\[
\limsup_{n\to\infty}\left|\int_{\Xi_n}e^{2\pi i k \varphi^{(q_n)}(x)}\,d\mu(x)\right|
\leq c\text{ for }k\geq k_0.
\]
It follows that for all $k\geq
k_0$, since $c< \delta$ and $\mu(\Xi_n) \to \delta$, we have
\begin{align*}
\left|\int_{\R}e^{2\pi i
kt}\,d\nu(t)\right|&=\lim_{n\to\infty}\left|\int_{\Xi_n}e^{2\pi i
k \varphi^{(q_n)}(x)}\,d\mu_{\Xi_n}(x)\right|\\
&=\lim_{n\to\infty}\frac{1}{\mu(\Xi_n)}\left|\int_{\Xi_n}e^{2\pi i
k \varphi^{(q_n)}(x)}\,d\mu(x)\right|\leq\frac{c}{\delta}<1,
\end{align*}
contrary to (\ref{dirichlet}).
\end{proof}

%\begin{remark}\label{remark:przecoh}
%A straightforward application of Lusin's theorem gives that for
%every measurable function $h:X\to\R$ we have
%\[\lim_{n\to\infty}\left|\left|\int_{\Xi_n}e^{2\pi i k (\varphi^{(q_n)}(x)+h(T^{q_n}x)-h(x))}\,d\mu(x)\right|-
%\left|\int_{\Xi_n}e^{2\pi i k \varphi^{(q_n)}(x)}\,d\mu(x)\right|\right|=0.\]
%\end{remark}

\subsection{IET of periodic type}\label{IETs:sec}
In this section we briefly summarize the Rauzy-Veech algorithm and
the properties that we need later and we give the definition of
IETs of hyperbolic periodic type. For further background material concerning
interval exchange transformations and Rauzy-Veech induction we
refer the reader to the excellent lecture notes \cite{ViB, Yo,
YoLN}.

 Let $T$ be the IET given by $(\pi,\lambda)$.
 Denote by
$\mathcal{S}^0_{\mathcal{A}}$ the subset of irreducible pairs,
i.e.\ such that
$\pi_1\circ\pi_0^{-1}\{1,\ldots,k\}\neq\{1,\ldots,k\}$ for $1\leq
k<d$.  We will always assume that $\pi \in
\mathcal{S}^0_{\mathcal{A}}$. The IET $T_{(\pi,\lambda)}$ is
explicitly given by $T(x)=x+w_\alpha$ for $x\in I_\alpha$, where
$w=\Omega_\pi\lambda$ and $\Omega_\pi$ is  the matrix
$[\Omega_{\alpha\,\beta}]_{\alpha,\beta\in\mathcal{A}}$ given by
\[\Omega_{\alpha\,\beta}=
\left\{\begin{array}{cl} +1 & \text{ if
}\pi_1(\alpha)>\pi_1(\beta)\text{ and
}\pi_0(\alpha)<\pi_0(\beta),\\
-1 & \text{ if }\pi_1(\alpha)<\pi_1(\beta)\text{ and
}\pi_0(\alpha)>\pi_0(\beta),\\
0& \text{ in all other cases.}
\end{array}\right.\]
Note that for every $\alpha\in\mathcal{A}$ with $\pi_0(\alpha)\neq
1$ there exists $\beta\in\mathcal{A}$ such that $\pi_0(\beta)\neq
d$ and $l_\alpha=r_\beta$. It follows that
\begin{equation}\label{zbzero}
\{l_\alpha:\alpha\in\mathcal{A},\;\pi_0(\alpha)\neq 1\}=
\{r_\alpha:\alpha\in\mathcal{A},\;\pi_0(\alpha)\neq d\}.
\end{equation}
Let  $\hat{I}=(0,|I|]$ and by $\widehat{T}_{(\pi,\lambda)}:\hat{I} \to \hat{I}$ denote the
exchange of the intervals
$\widehat{I}_\alpha:=(l_\alpha,r_\alpha]$, $\alpha\in\mathcal{A}$,
i.e.\ $T_{(\pi,\lambda)}x=x+w_\alpha$ for $x\in
(l_\alpha,r_\alpha]$.
%%%DO WE NEED IT???
%Note that for every $\alpha\in\mathcal{A}$
%with $\pi_1(\alpha)\neq 1$ there exists $\beta\in\mathcal{A}$ such
%that $\pi_1(\beta)\neq d$ and
%$T_{(\pi,\lambda)}l_\alpha=\widehat{T}_{(\pi,\lambda)}r_\beta$. It
%follows that
%\begin{equation}\label{zbjeden}
%\{T_{(\pi,\lambda)}l_\alpha:\alpha\in\mathcal{A},\;\pi_1(\alpha)\neq
%1\}=
%\{\widehat{T}_{(\pi,\lambda)}r_\alpha:\alpha\in\mathcal{A},\;\pi_1(\alpha)\neq
%d\}.
%\end{equation}
Let $End(T) = \{ l_\alpha, r_\alpha, \alpha \in \mathcal{A}\}$
stand for the set of end points of the intervals
$I_\alpha:\alpha\in\mathcal{A}$.

 A pair ${(\pi,\lambda)}$
satisfies the {\em Keane condition} (see \cite{Keane}) if
$T_{(\pi,\lambda)}^m l_{\alpha}\neq l_{\beta}$ for all $m\geq 1$
and for all $\alpha,\beta\in\mathcal{A}$ with $\pi_0(\beta)\neq
1$.

\subsubsection*{Rauzy-Veech induction} Let $T=T_{(\pi,\lambda)}$,
$(\pi,\lambda)\in\mathcal{S}^0_{\mathcal{A}}\times\R_+^{\mathcal{A}}$
be an IET satisfying the Keane condition. Then
$\lambda_{\pi_0^{-1}(d)}\neq\lambda_{\pi_1^{-1}(d)}$. Let
\[\widetilde{I}=\left[0,\max\left({l}_{\pi_0^{-1}(d)},{l}_{\pi_1^{-1}(d)}\right)\right)\]
and denote by
$\mathcal{R}(T)=\widetilde{T}:\widetilde{I}\to\widetilde{I}$ the
first return map of $T$ to the interval $\widetilde{I}$. Set
\begin{equation}\label{epsilondef} \vep(\pi,\lambda)=\left\{
\begin{array}{ccl}
0&\text{ if }&\lambda_{\pi_0^{-1}(d)}>\lambda_{\pi_1^{-1}(d)},\\
1&\text{ if }&\lambda_{\pi_0^{-1}(d)}<\lambda_{\pi_1^{-1}(d)}.
\end{array}
\right.
\end{equation}
Let us consider a pair
$\widetilde{\pi}=(\widetilde{\pi}_0,\widetilde{\pi}_1)\in\mathcal{S}^0_{\mathcal{A}}$,
where
\begin{eqnarray*}\widetilde{\pi}_\vep(\alpha)&=&\pi_\vep(\alpha)
\text{ for all }\alpha\in\mathcal{A}\text{ and }\\
\widetilde{\pi}_{1-\vep}(\alpha)&=&\left\{
\begin{array}{cll}
\pi_{1-\vep}(\alpha)& \text{ if
}&\pi_{1-\vep}(\alpha)\leq\pi_{1-\vep}\circ\pi^{-1}_\vep(d),\\
\pi_{1-\vep}(\alpha)+1& \text{ if
}&\pi_{1-\vep}\circ\pi^{-1}_\vep(d)<\pi_{1-\vep}(\alpha)<d,\\
\pi_{1-\vep}\pi^{-1}_\vep(d)+1& \text{ if
}&\pi_{1-\vep}(\alpha)=d.\end{array} \right.
\end{eqnarray*}
As it was shown by Rauzy in \cite{Ra}, $\widetilde{T}$ is also an
IET on $d$-intervals
\[\widetilde{T}=T_{(\widetilde{\pi},\widetilde{\lambda})}\text{ with }
\widetilde{\lambda}=\Theta^{-1}(\pi,\lambda)\lambda,\] where
\[\Theta(T)=\Theta(\pi,\lambda)=I+E_{\pi_{\vep}^{-1}(d)\,\pi_{1-\vep}^{-1}(d)}\in\SL(\Z^{\mathcal{A}}).\]
Moreover,
\begin{equation}\label{omega}
\Theta^t(\pi,\lambda)\cdot\Omega_{\pi}\cdot\Theta(\pi,\lambda)=\Omega_{\widetilde{\pi}}.
\end{equation}
It follows that
$\ker\Omega_{{\pi}}=\Theta(\pi,\lambda)\ker\Omega_{\widetilde{\pi}}$.
Thus taking
$H_{\pi}=\Omega_{{\pi}}(\R^{\mathcal{A}})=\ker\Omega_{{\pi}}^\perp$
we get $H_{\widetilde{\pi}}=\Theta^t(\pi,\lambda)H_{{\pi}}$.
Moreover, $\dim H_{{\pi}}=2g$ and $\dim
\ker\Omega_{{\pi}}=\kappa-1$, where $g$ is the genus of the
translation surface  associated to $\pi$ and  $\kappa$ the number
of singularities (for more details we refer the reader to
\cite{ViB}).
\begin{figure}
\centering
\subfigure[Case $\lambda_{\alpha_0}> \lambda_{\alpha_1}$ or $\epsilon(\lambda, \pi) = 0$.]{\label{Rauzytop}
\includegraphics[width=0.46\textwidth]{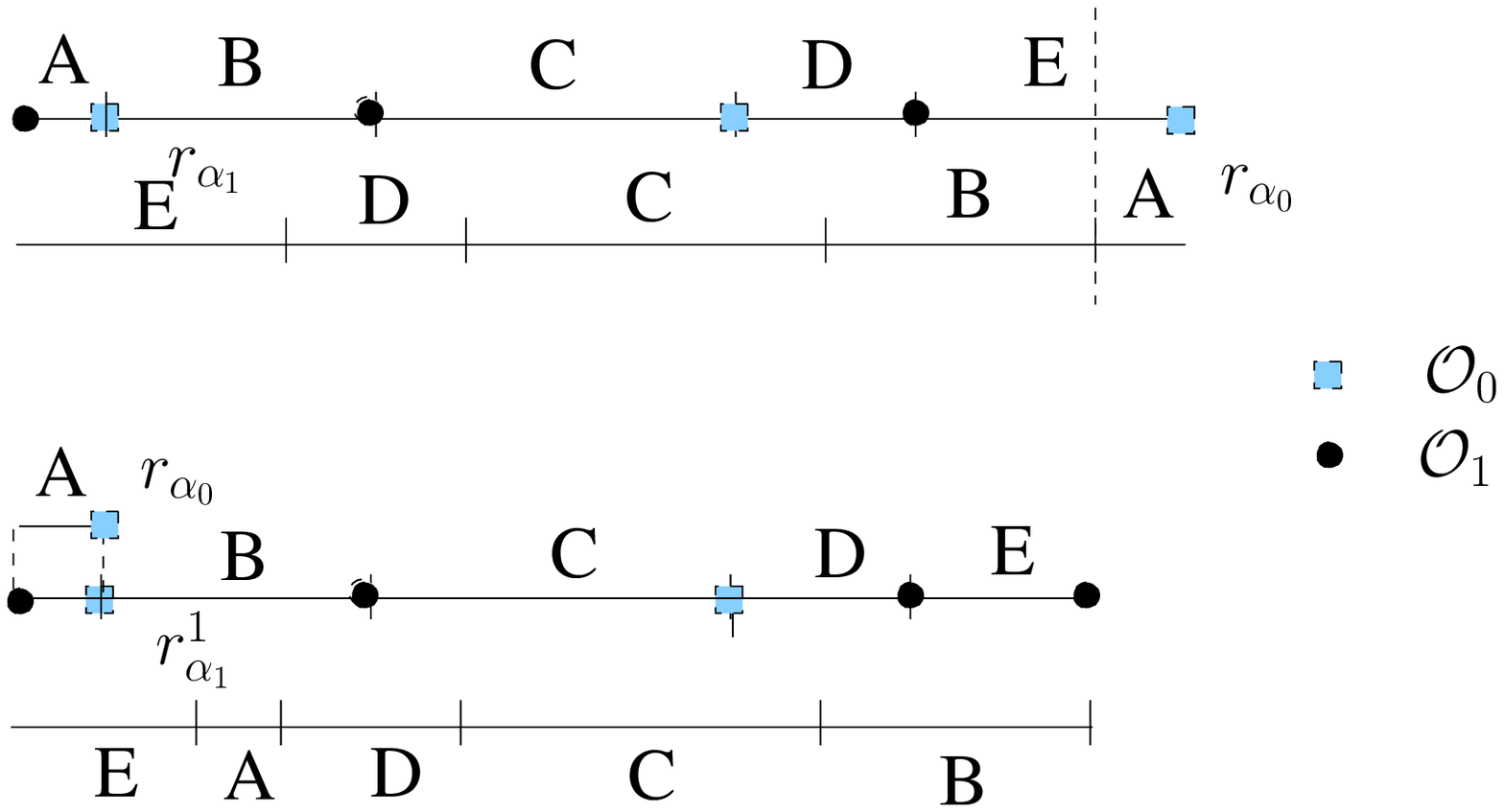}}
\hspace{2mm}
\subfigure[Case $\lambda_{\alpha_0}< \lambda_{\alpha_1}$ or $\epsilon(\lambda, \pi) = 1$.]{\label{Rauzybottom}
\includegraphics[width=0.48\textwidth]{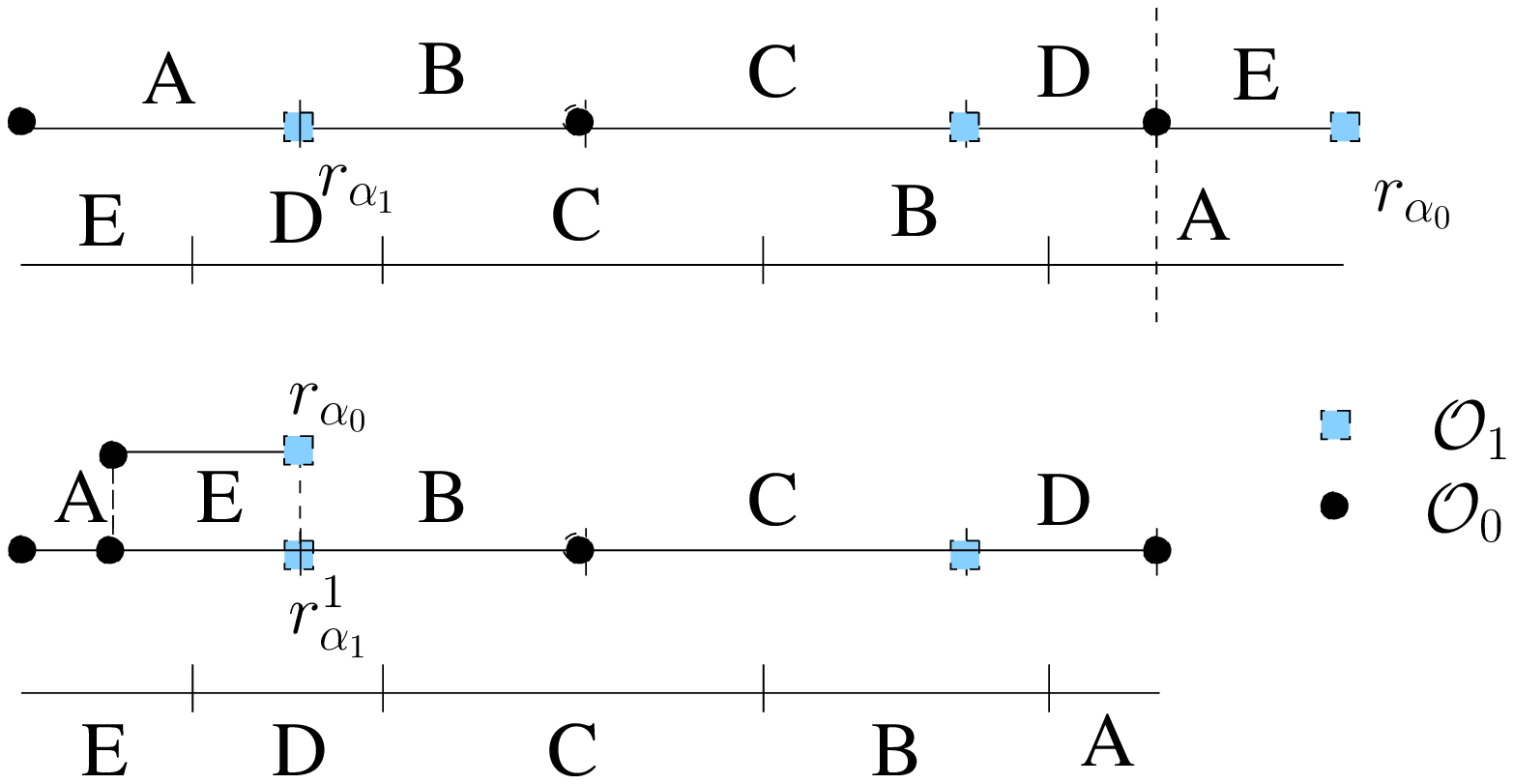}}
\caption{Rauzy Veech induction}\label{Rauzy}
\end{figure}

 The IET $\widetilde{T}$ fulfills the Keane condition as
well. Therefore we can iterate the renormalization procedure and
generate a sequence of IETs $(\mathcal{R}^n(T))_{n\geq 0}$. % , where $T^{(n)}=$ for $n\geq 0$.
Denote by
%$\pi^{(n)}=(\pi^{(n)}_0,\pi^{(n)}_1)\in\mathcal{S}^0_{\mathcal{A}}$
$\pi^{n}=(\pi^{n}_0,\pi^{n}_1)\in\mathcal{S}^0_{\mathcal{A}}$ and $\lambda^{n}=(\lambda^{n}_\alpha)_{\alpha\in\mathcal{A}}$
respectively the pair and
%$\lambda^{(n)}=(\lambda^{(n)}_\alpha)_{\alpha\in\mathcal{A}}$ the
 the
vector which determine $\mathcal{R}^n(T)$. Then $\mathcal{R}^n(T)$ is the first
return map of $T$ to the interval $I^{n}=[0,|\lambda^{n}|)$
and
\[\lambda=\Theta^{(n)}(T)\lambda^{n}\text{ with }\Theta^{(n)}(T)=
\Theta(T)\cdot\Theta(\mathcal{R}(T))\cdot\ldots\cdot\Theta(\mathcal{R}^{n-1}(T)).\]
We denote by $I^{n}_{\alpha}=[l^n_\alpha,r^n_{\alpha})$ the
intervals exchanged by $\mathcal{R}^n(T)$.

Let $T:I\to I$ be an arbitrary  IET satisfying the Keane
condition. Suppose that $(n_k)_{k\geq 0}$ is an increasing
sequence of natural numbers such $n_0=0$ and set
\begin{equation}\label{Z}
Z(k+1):=\Theta(\mathcal{R}^{n_k}(T))\cdot\Theta(\mathcal{R}^{n_k+1}(T))
\cdot\ldots\cdot\Theta(\mathcal{R}^{n_{k+1}-1}(T))
\end{equation}
Since $\lambda^{n_k}=Z(k+1)\lambda^{n_{k+1}}$, if for each $k<k'$
we let
\begin{equation} \label{Q}
Q(k,k')=Z(k+1)\cdot Z(k+2)\cdot\ldots\cdot Z(k')
\end{equation}
then we have $\lambda^{n_k}=Q(k,k')\lambda^{n_{k'}}$. We will
write $Q(k)$ for $Q(0,k)$. By definition,
$\mathcal{R}^{n_{k'}}(T):I^{n_{k'}}\to I^{n_{k'}}$ is the first
return map of $\mathcal{R}^{n_k}(T):I^{n_k}\to I^{n_k}$ to the
interval $I^{n_k}\subset I^{n_{k'}}$. Moreover,
$Q_{\alpha\beta}(k,k')$ is the time spent by any point of
$I^{n_{k'}}_{\beta}$ in $I^{n_k}_{\alpha}$ until it returns to
$I^{n_{k'}}$. It follows that
\[Q_{\beta}(k,k')=\sum_{\alpha\in\mathcal{A}}Q_{\alpha\beta}(k,k')\]
is the first return time of points of $I^{n_{k'}}_{\beta}$ to
$I^{n_{k'}}$.

In what follows, the norm of a vector is defined as the largest
absolute value of the coefficients and for any matrix
$B=[B_{\alpha\beta}]_{\alpha,\beta\in\mathcal{A}}$ we set
$\|B\|=\max_{\beta\in\mathcal{A}}\sum_{\alpha\in\mathcal{A}}|B_{\alpha\beta}|$.

\subsubsection*{IETs of periodic type}
\begin{definition}[see \cite{Si-Ul}]\label{periodicIETdef}
An IET $T$ is of {\em periodic type} if there exists
$p>0$ (called a {\em period of $T$}) such that
$\Theta(\mathcal{R}^{n+p}(T))=\Theta(\mathcal{R}^n(T))$ for every $n\geq 0$ and
$A=A(T):=\Theta^{(p)}(T)$ (called a {\em period matrix of $T$}) has
strictly positive entries.
\end{definition}
Since the set $\mathcal{S}^0_{\mathcal{A}}$ is finite, up to
taking a multiple of the period $p$ if necessary, we can assume
that $\pi^{p}=\pi$. We will always assume that the period $p$ is
chosen so that $\pi^{p}=\pi$. Explicit examples of IETs of
periodic type appear in \cite{Si-Ul}. The procedure  to construct
them is based on choosing closed paths on Rauzy class and using
the following Remark.

\begin{remark}
Suppose that $T=T_{(\pi,\lambda)}$ is of periodic  type with
period matrix $A=\Theta^{(p)}(T)$. It follows that
$\lambda=A^n\lambda^{pn}\in A^n \R^{\mathcal{A}}_+$ and hence
$\lambda$ belongs to $\bigcap_{n\geq 0} A^n\R^{\mathcal{A}}_+$
which is a one-dimensional convex cone (see \cite{Ve1}). Therefore
$\lambda$ is a positive right Perron-Frobenius eigenvector of the
matrix $\Theta^{(p)}(T)$. It follows that
$(\pi^{p},\lambda^{p}/|\lambda^{p}|)=(\pi,\lambda/|\lambda|)$ and
$|\lambda|/|\lambda^{p}|$ is the Perron-Frobenius eigenvector of
the matrix $A$.
\end{remark}

\begin{remark}\label{periodic are ue and minimal}
IETs of periodic type automatically satisfy the Keane condition.
Indeed, $T$ satisfies the Keane condition if and only if the orbit
of $T$ under $\mathcal{R}$ is infinite (see \cite{Ma-Mo-Yo}) and
IETs of periodic type by definition have an infinite (periodic)
orbit under $\mathcal{R}$. Moreover, using the methods in
\cite{Ve3} (see also \cite{ViB}) one can show that every IET of
periodic type is uniquely ergodic.
\end{remark}

 Suppose that $T=T_{(\pi,\lambda)}$ is of periodic type
and let $A=\Theta^{(p)}(T)$.  By (\ref{omega}),
\[A^t\Omega_{\pi}A=\Omega_{{\pi}}\text{ and hence }
\ker\Omega_{{\pi}}=A\ker\Omega_{{\pi}}\text{ and
}H_{{\pi}}=A^tH_{{\pi}}.\] Moreover, multiplying the period $p$ if
necessary, we can assume that $A|_{\ker\Omega_{\pi}}=Id$ (see
Remark~\ref{assumexiid} for details). Denote by $Sp(A)$ the set of
complex eigenvalues of $A$, including multiplicities. Let us
consider the set of Lyapunov exponents $\{\log|\rho|:\rho\in
Sp(A)\}$. It consists of the numbers
\[\theta_1>\theta_2\geq\theta_3\geq\ldots\geq\theta_g\geq0
=\ldots=0\geq-\theta_g\geq\ldots\geq-\theta_3\geq-\theta_2>-\theta_1,\]
where $2g=\dim H_{\pi}$ and $0$ occurs with the multiplicity
$\kappa-1=\dim\ker\Omega_{{\pi}}$ (see e.g.\ \cite{Zo}). Moreover,
$\rho_1:=\exp\theta_1$ is the Perron-Frobenius eigenvalue of $A$.
%We will use sometimes the symbol $\theta_i(T)$ instead of
%$\theta_i$ to emphasize that it depends on $T$.

\begin{definition}\label{defhypiet}
An IET $T_{(\pi,\lambda)}$ is {\em of hyperbolic periodic type} if
it is of periodic type and $A^t:H_{\pi}\to H_{\pi}$ is a
hyperbolic linear map, or equivalently  $\theta_g>0$.
\end{definition}

\begin{convention}
When $T$ is of periodic type, we will always consider  iterates of
$\mathcal{R}$ corresponding to the sequence $(p k)_{k\geq 0}$,
where $p$ is a period of $T$ and $A$ the associated periodic
matrix,  chosen so that $\pi^p=\pi$ and
$A|_{\ker\Omega_{\pi}}=Id$.
\end{convention}
\begin{definition}
Suppose that $T=T_{(\pi,\lambda)}$ is of periodic type with period $p$ and period
matrix $A=\Theta^{(p)}(T)$ as above. In this case we will denote by
$T^{(k)}=(\pi^{(k)},\lambda^{(k)})$  the IET $\mathcal{R}^{p
k}(T)$, by $I^{(k)}=[0, |\lambda^{kp}|)$ the interval on which
$T^{(k)}$ is defined and by $I_{\alpha}^{(k)}=[l_\alpha^{(k)},
r_\alpha^{(k)})$ the intervals exchanged by $T^{(k)}$.
\end{definition}
\begin{convention}
In the rest of the paper,  when $T$ is of periodic type, the
matrices $Z(k)$ and $Q(k)$ will denote be the matrices  associated
to the sequence $(p k)_{k\geq 0}$ by (\ref{Z}) and (\ref{Q})
respectively. Clearly $Z(k)=A$ and $Q(k,k')=A^{k'-k} = Q(k'-k)$
for all $0\leq k\leq k'$.
\end{convention}

In the spirit of \cite{Ve2}, we set
$\nu_1(A)=\max\{{A_{\alpha\gamma}}/{A_{\beta\gamma}}:\alpha,\beta,\gamma\in\mathcal{A}\}$,
$\nu_2(A)=\nu_1(A^T)= \max\{{A_{\gamma \alpha}}/{A_{\gamma \beta}}:\alpha,\beta,\gamma\in\mathcal{A}\}$
and let  $\nu(A) = \max \{ \nu_1(A), \nu_2(A)\}$.  Since
$\lambda^{(k)}= A\lambda^{(k+1)}$ and for any $k\geq 1$ we have $Q(k)= Q(k-1)A$,  we have
\begin{equation} \label{balancedintervals}
\frac{|I^{(k)}_\beta|}{ \nu(A)} \leq |I^{(k)}_\alpha |  \leq
\nu(A) |I^{(k)}_\beta| , \quad  \frac{Q_\beta(k)}{ \nu(A)} \leq Q_\alpha(k)   \leq
\nu(A) Q_\beta(k) \quad \forall \alpha, \beta \in
\mathcal{A}.
\end{equation}
From the above relation, it also follows that Rohlin towers have
comparable areas, that is, since by Rohlin's Lemma and Pigeon
Hole principle there exists $\beta$ such that $Q_\beta (k )
|I^{(k)}_\beta | \geq |I|/d$,  one has
\begin{equation}\label{areas}
\frac{1}{d \nu(A)^2|I^{(0)}|} \leq Q_\alpha (k ) |I^{(k)}_\alpha |  \leq |I^{(0)}|,
\quad \text{ for all }\alpha \in \mathcal{A}.
\end{equation}
\subsubsection*{A bases for the kernel}
Let $p:\{0,1,\ldots,d,d+1\}\to\{0,1,\ldots,d,d+1\}$ stand for  the
permutation
\[p(j)=\left\{
\begin{matrix}
\pi_1\circ\pi^{-1}_0(j)&\text{ if }&1\leq j\leq d\\
j&\text{ if }&j=0,d+1.
\end{matrix}
\right.
\]
Following \cite{Ve1,Ve2}, denote by $\sigma=\sigma_\pi$ the
corresponding permutation on $\{0,1,\ldots,d\}$,
\[\sigma(j)=p^{-1}(p(j)+1)-1\text{ for }0\leq j\leq d.\]
Then
$\widehat{T}_{(\pi,\lambda)}r_{\pi_0^{-1}(j)}={T}_{(\pi,\lambda)}r_{\pi_0^{-1}(\sigma
j)}$ for all $j\neq 0,p^{-1}(d)$. Denote by $\Sigma(\pi)$ the set
of orbits for the permutation $\sigma$. Let $\Sigma_0(\pi)$ stand
for the subset of orbits that do not contain zero.
\begin{remark}
If $T$ is obtained from a minimal flow $(\phi_t)_{t\in \R}$ on a
surface $\Surf$ as Poincar{\'e} first return map to a transversal,
then the orbits $\mathcal{O}\in \Sigma(\pi)$ are in one to one
correspondence with saddle points of $(\phi_t)_{t\in \R}$. Hence
$\#\Sigma(\pi)=\kappa$, where $\kappa$ is the number of saddle
points of $(\phi_t)_{t\in \R}$.
\end{remark}
For every $\mathcal{O}\in\Sigma(\pi)$ denote by
$b(\mathcal{O})\in\R^{\mathcal{A}}$ the vector given by
\begin{equation} \label{bdef}
b(\mathcal{O})_{\alpha}=\chi_{\mathcal{O}}(\pi_0(\alpha))-\chi_{\mathcal{O}}(\pi_0(\alpha)-1)
\text{ for }\alpha\in\mathcal{A},
\end{equation}
where $\chi_{\mathcal{O}} (j)=1$ iff $j \in \mathscr{O}$ and $0$ otherwise.
Moreover, for every $\mathcal{O}\in \Sigma(\pi)$, we denote by
\begin{equation} \label{defAO}
\mathcal{A}^-_{\mathcal{O}}=\{
\alpha\in\mathcal{A}, \ \pi_0(\alpha)\in \mathcal{O}\}, \qquad
\mathcal{A}^+_{\mathcal{O}}=\{ \alpha\in\mathcal{A}, \
\pi_0(\alpha)-1\in \mathcal{O}\} .
\end{equation} If $\alpha\in
\mathcal{A}_{\mathcal{O}}^+$ (respectively $\alpha\in
\mathcal{A}_{\mathcal{O}}^-$ ) then the left (respectively right)
endpoint of $I_\alpha$ belongs to a separatrix of the saddle
represented by $\mathcal{O}$.
\begin{lemma}[see \cite{Ve2}]\label{bcharh}
For every irreducible pair $\pi$ we have
$\sum_{\mathcal{O}\in\Sigma(\pi)}b(\mathcal{O})=0$, the vectors
$b(\mathcal{O})$, $\mathcal{O}\in\Sigma_0(\pi)$ are linearly
independent and the linear subspace generated by them is equal to
$\ker\Omega_\pi$. Moreover, $h\in H_{\pi}$ if and only if $\langle
h,b(\mathcal{O)}\rangle=0$ for every $\mathcal{O}\in\Sigma(\pi)$.
  \bez
\end{lemma}

\begin{remark}\label{remiso}
Let $\Lambda^\pi:\R^{\mathcal{A}}\to\R^{\Sigma_0(\pi)}$ stand for the
linear transformation given by $(\Lambda^\pi
h)_\mathcal{O}=\langle h,b(\mathcal{O)}\rangle$ for
$\mathcal{O}\in\Sigma_0(\pi)$. By Lemma~\ref{bcharh},
$H_{\pi}=\ker\Lambda^\pi$ and if $\R^{\mathcal{A}}=F\oplus H_\pi$
is a direct sum decomposition then
$\Lambda^\pi:F\to\R^{\Sigma_0(\pi)}$ establishes an isomorphism of
linear spaces. It follows that there exists $K_F>0$ such that
\[\|h\|\leq K_F\|\Lambda^\pi h\| \;\;\text{ for all }\;\;h\in F.\]
\end{remark}

\begin{lemma}[see \cite{Ve2}]\label{invario}
Suppose that
$T_{(\widetilde{\pi},\widetilde{\lambda})}=\mathcal{R}(T_{(\pi,\lambda)})$.
Then there exists a bijection
$\xi:\Sigma(\pi)\to\Sigma(\widetilde{\pi})$ that depends only on
$(\pi, \lambda)$ such that
$\Theta(\pi,\lambda)^{-1}b(\mathcal{O})=b(\xi\mathcal{O})$ for
$\mathcal{O}\in\Sigma(\pi)$.\bez
\end{lemma}

Moreover, analyzing the explicit correspondence given by $\xi$ (we
refer the reader for example to the formulas in \cite{ViB}, \S
2.4) one can check that we have the following. For $\upsilon=0,1$,
let $\alpha_\upsilon\in \mathcal{A}$ be such that
$\pi_\upsilon(\alpha_\upsilon)=d$. Define the orbits
$\mathcal{O}_0, \mathcal{O}_1 \in \Sigma(\pi)$ (where possibly
$\mathcal{O}_0= \mathcal{O}_1$)   as follows. Let $\vep=\vep(\pi,
\lambda)$ is as in (\ref{epsilondef}) and Let $\mathcal{O}_\vep
\in \Sigma(\pi)$ such that $d \in \mathcal{O}_\vep$. Remark that
$\alpha_0,\alpha_1\in \mathcal{A}_{\mathcal{O}_{\vep}}^-$ since
$\pi_0(\alpha_0) = \pi_1(\alpha_1)=d\in \mathcal{O}_\vep$. Let
$\mathcal{O}_{1-\vep}$ be such that $\alpha_{1-\vep}\in
\mathcal{A}_{\mathcal{O}_{1-\vep}}^+$. Denote by
$\widetilde{\mathcal{A}}^{\pm}_{\mathcal{O}}$,
$\mathcal{O}\in\Sigma(\widetilde{\pi})$ the corresponding sets for
the pair $\widetilde{\pi}$.
\begin{lemma}\label{orbitcorrespondence}
For each  $\mathcal{O}\in\Sigma(\pi)$,
$\widetilde{\mathcal{A}}^+_{\xi\mathcal{O}_\vep} =
\mathcal{A}^+_{\mathcal{O}_\vep}$. For each  $\mathcal{O}\notin \{
\mathcal{O}_0,  \mathcal{O}_1\}$ or if  $\mathcal{O} =
\mathcal{O}_0= \mathcal{O}_1 $, then $\widetilde{\mathcal{A}}^-
_{\xi\mathcal{O}} = \mathcal{A}^-_{\mathcal{O}}$. If
$\mathcal{O}_0\neq \mathcal{O}_1 $, then
$\widetilde{\mathcal{A}}^-_{\xi\mathcal{O}_\vep} =
\mathcal{A}^-_{\mathcal{O}_\vep}\backslash \{\alpha_\vep\}$ and $
\widetilde{\mathcal{A}}^-_{\xi\mathcal{O}_{1-\vep}} =
\mathcal{A}^-_{\mathcal{O}_{1-\vep}}\cup \{\alpha_{\vep}\}$.
%\begin{equation}\label{orbitrelation} \mathcal{A}^-_{\xi\mathcal{O}_\vep} =
%\mathcal{A}^-_{\mathcal{O}_\vep}\backslash \{\alpha_\vep\}, \qquad
%\mathcal{A}^-_{\xi\mathcal{O}_{1-\vep}} =
%\mathcal{A}^-_{\mathcal{O}_{1-\vep}}\cup \{\alpha_{\vep}\}.
% \end{equation}
\end{lemma}
An example of these correspondence of orbits is illustrated in Figure \ref{Rauzy}.
\begin{remark}\label{assumexiid}
If $T$ is of periodic type, let us remark that $\Sigma(\pi^{(k)})=
\Sigma(\pi^{(k')})= \Sigma(\pi)$ for every $k'\geq k\geq 0$.  Up
to replacing the period $p$ by a multiple, we can assume that $Q(k,
k' )b(\mathcal{O})=b(\mathcal{O})$ for each $\mathcal{O}\in
\Sigma(\pi^{(k)})$ and $0\leq k\leq k'$.
\end{remark}

\subsection{Cocycles with logarithmic singularities}\label{cocycles:sec}
  Denote by $\bv(\sqcup_{\alpha\in \mathcal{A}}
I^{(k)}_{\alpha})$ the space of functions $\varphi:I^{(k)}\to
\mathbb{R}$ such that the restriction $\varphi:I^{(k)}_{\alpha}\to
\mathbb{R}$ is of bounded variation for every $\alpha\in
\mathcal{A}$. Let us denote by $\Var{{f}}{J}$ the total variation
of ${f}$ on the interval $J\subset I$. Then set
\begin{equation}\label{variation}
\var \varphi=\sum_{\alpha\in
\mathcal{A}}\Var{\varphi}{I^{(k)}_{\alpha}}.
\end{equation}
The space $\bv(\sqcup_{\alpha\in \mathcal{A}} I^{(k)}_{\alpha})$
is equipped with the norm
$\|\varphi\|_{\bv}=\|\varphi\|_{\sup}+\var\varphi$. Denote by
$\bv_0(\sqcup_{\alpha\in \mathcal{A}} I^{(k)}_{\alpha})$ the
subspace of all functions in $\bv(\sqcup_{\alpha\in \mathcal{A}}
I^{(k)}_{\alpha})$ with zero mean.

For every function $\varphi\in\bv(\sqcup_{\alpha\in \mathcal{A}}
I_{\alpha})$ and $x\in I$ we will denote by $\varphi_+(x)$ and
$\varphi_-(x)$ the right-handed  and left-handed limit of
$\varphi$ at $x$ respectively.
 Denote by $\ac(\sqcup_{\alpha\in
\mathcal{A}} I_{\alpha})$ the space of functions $\varphi:I\to
\mathbb{R}$ which are absolutely continuous on the interior of
each $I_{\alpha}$, $\alpha\in \mathcal{A}$ and by
$\ac_0(\sqcup_{\alpha\in \mathcal{A}} I_{\alpha})$ its subspace of
zero mean functions. For any $\varphi\in\ac(\sqcup_{\alpha\in
\mathcal{A}} I_{\alpha})$ let
\[s(\varphi)=\int_I\varphi'(x)\,dx=\sum_{\alpha\in\mathcal{A}}(\varphi_{-}(r_\alpha)-\varphi_{+}(l_\alpha)).\]
Denote by $\bv^1(\sqcup_{\alpha\in \mathcal{A}} I_{\alpha})$ the
space of functions $\varphi\in\ac(\sqcup_{\alpha\in \mathcal{A}}
I_{\alpha})$ such that $\varphi'\in \bv(\sqcup_{\alpha\in
\mathcal{A}} I_{\alpha})$ and by $\bv_*^1(\sqcup_{\alpha\in
\mathcal{A}} I_{\alpha})$ its subspace of functions $\varphi$ for
which $s(\varphi)=0$.

\begin{theorem}[see \cite{Ma-Mo-Yo} and \cite{MMYlinearization}]
\label{cohcon} If $T:I\to I$ satisfies a Roth type condition then
each cocycle $\varphi\in \bv_*^1(\sqcup_{\alpha\in \mathcal{A}}
I_{\alpha})$ for $T$ is cohomologous (via a continuous transfer
function) to a cocycle which is constant on each interval
$I_{\alpha}$, $\alpha\in \mathcal{A}$. Moreover, the set of IETs
satisfying this Roth type condition has full measure and contains
all IETs of periodic type.
\end{theorem}
The prove of the above result is based on the following conclusion
from the Gottschalk-Hedlund theorem (see \S3.4 in
\cite{MMYlinearization}).

\begin{proposition}\label{ghmmy}
If $T$  satisfies  the Keane condition and $\varphi\in
\ac_0(\sqcup_{\alpha\in \mathcal{A}} I_{\alpha})$ is a function
such that the sequence $(\varphi^{(n)})_{n\geq 0}$ is uniformly
bounded then $\varphi$ is a coboundary with a continuous  transfer
function.
\end{proposition}

Denote by $\pl(\sqcup_{\alpha\in \mathcal{A}} I_{\alpha})$ the set
of functions $\varphi:I\to\R$ such that $\varphi(x)=sx+c_{\alpha}$
for $x\in I_{\alpha}$. As a consequence of Theorem~\ref{cohcon} we
have the following.

\begin{corollary}\label{lempl}
If the IET $T:I\to I$ is of periodic type then each cocycle
$\varphi\in\bv^1(\sqcup_{\alpha\in \mathcal{A}} I_{\alpha})$ is
cohomologous (via a continuous transfer function) to a cocycle
$\varphi_{pl}\in\pl(\sqcup_{\alpha\in
\mathcal{A}} I_{\alpha})$ with $s({\varphi}_{pl})=s(\varphi)$.
\end{corollary}

In the Introduction \S\ref{introduction:sec} we defined the space
$\ol (\sqcup_{\alpha\in \mathcal{A}} I_{\alpha})$ of functions
with logarithmic singularities of geometric type (see
Definition~\ref{LogSing}) and the subspace
$\logs(\sqcup_{\alpha\in \mathcal{A}} I_{\alpha})$ of symmetric
logarithmic singularities of geometric type, which satisfy the
symmetry condition (\ref{zerosymweak}). We denote by
$\ol_0(\sqcup_{\alpha\in \mathcal{A}} I_{\alpha})$ and
$\logs_0(\sqcup_{\alpha\in \mathcal{A}} I_{\alpha})$   the
corresponding spaces of functions with zero mean.

\begin{definition}\label{strongsymmetry}
A function $\varphi\in  \ol (\sqcup_{\alpha\in \mathcal{A}}
I_{\alpha})$ of the form (\ref{fform}) has \emph{strong symmetric}
logarithmic singularities if for every $\mathcal{O}\in\Sigma(\pi)$
we have
\begin{equation}\label{zerosym}
\sum_{\alpha\in\mathcal{A}_\mathcal{O}^-}C^-_{\alpha}-
\sum_{\alpha\in\mathcal{A}_\mathcal{O}^+}C^+_{\alpha} =0,
%= \sum_{\alpha\in\mathcal{A},\pi_0(\alpha)\in\mathcal{O}}C^-_{\alpha}-
%\sum_{\alpha\in\mathcal{A},\pi_0(\alpha)-1\in\mathcal{O}}C^+_{\alpha}.
\end{equation}
where $\mathcal{A}_\mathcal{O}^-, \mathcal{A}_\mathcal{O}^+$ are
the sets defined in (\ref{defAO}).
\end{definition}
Denote by $\ls(\sqcup_{\alpha\in \mathcal{A}} I_{\alpha})$ the
space of functions with  strong symmetric logarithmic
singularities of geometric type and let $\ls_0:=\ls\cap\ol_0$.
Clearly $\ls(\sqcup_{\alpha\in \mathcal{A}} I_{\alpha}) \subset
\logs(\sqcup_{\alpha\in \mathcal{A}} I_{\alpha})$ since the
condition (\ref{zerosym}) implies the weaker symmetry condition
(\ref{zerosymweak}) by summing over $\mathcal{O}\in \Sigma$.
Strong symmetric singularities of geometric type appear naturally
from extensions of locally Hamiltonian flows, see
\S\ref{reduction:sec}. This stronger condition of symmetry is
important in the proof of ergodicity.
%\ref{reduction:sec}\footnote{The condition on constants which are zero,
%which seems rather technical, is automatically satisfied by
%functions which have this geometric origin.
%Let us remark that the definition
%of symmetry here given is stronger than the definition of
%symmetric logarithmic singularities usually in the literature, for
%example in \cite{Ul:abs, Scheg}, but again it is automatically
%satisfied for functions which are geometrically originated from
%extensions of locally Hamiltonian flows, as proved in \S
%\ref{motivation:sec}.

We will also use the space $\overline{\ol}(\sqcup_{\alpha\in
\mathcal{A}} I_{\alpha})={\ol}(\sqcup_{\alpha\in \mathcal{A}}
I_{\alpha})+{\bv}(\sqcup_{\alpha\in \mathcal{A}} I_{\alpha})$
(respectively $\overline{\ls}(\sqcup_{\alpha\in \mathcal{A}}
I_{\alpha})={\ls}(\sqcup_{\alpha\in \mathcal{A}}
I_{\alpha})+{\bv}(\sqcup_{\alpha\in \mathcal{A}} I_{\alpha})$),
i.e.\ the space of all functions with logarithmic singularities
(respectively strong symmetric logarithmic singularities) of
geometric type and zero mean of the form (\ref{fform}) for which
we  require only that $g_\varphi\in\bv( \sqcup_{\alpha\in
\mathcal{A}} I_{\alpha})$. We will denote by $\overline{\ol}_0$
and $\overline{\ls}_0$ their subspaces of zero mean functions.

 Note that the space  $\bv$ $(\bv^1)$ coincides with the subspace of functions
$\varphi\in \overline{\ol}$ ($\ol$) as in (\ref{fform}) such that
$C^{\pm}_\alpha=0$ for all $\alpha\in\mathcal{A}$.

\begin{definition}\label{LVdef}
For every $\varphi\in\overline{\ol}(\sqcup_{\alpha\in \mathcal{A}}
I_{\alpha})$ of the form (\ref{fform}) set
\[ \bl(\varphi) = \sum_{\alpha\in\mathcal{A}}(|C^+_\alpha|+|C^-_\alpha|)\quad\text{
and }\quad\lv(\varphi):=\bl(\varphi)+\var g_\varphi.\]
\end{definition}
\noindent The quantity $\lv(\varphi)$ will play throughout the
paper an essential role to bound functions $\overline{\ol}$, since
it controls simultaneously the logarithmic singularities, through
the logarithmic constants $\bl(\varphi)$, and the part of bounded
variation.

 The spaces $\overline{\ls}(\sqcup_{\alpha\in
\mathcal{A}} I_{\alpha})$ and $\overline{\ls}_0(\sqcup_{\alpha\in
\mathcal{A}} I_{\alpha})$ equipped with the norm
\[\|\varphi\|_{\lv}=\bl(\varphi)+\|g_\varphi\|_{\bv} \] become
Banach spaces for which ${\ls}(\sqcup_{\alpha\in \mathcal{A}}
I_{\alpha})$  or respectively $\ls_0(\sqcup_{\alpha\in
\mathcal{A}} I_{\alpha})$   are dense subspaces.

For every integrable function $f:I\to\R$ and a subinterval
$J\subset I$ let $m(f,J)$ stand for the mean value of $f$ on $J$,
this is
\[m(f,J)=\frac{1}{|J|}\int_Jf(x)\,dx.\]

\begin{proposition}\label{lemlos}
If $\varphi\in\overline{\ol}(\sqcup_{\alpha\in \mathcal{A}}
I_{\alpha})$ and $J\subset I_\alpha$ for some
$\alpha\in\mathcal{A}$ then
\begin{equation}\label{meanv1}
|m(\varphi,J)-m(\varphi,I_\alpha)|\leq
\lv(\varphi)\left(4+\frac{|I_\alpha|}{|J|}\right)
\end{equation}
and
\begin{equation}\label{meanv2}
\frac{1}{|J|}\int_J|\varphi(x)-m(\varphi,J)|\,dx\leq
8\lv(\varphi).
\end{equation}
\end{proposition}
\noindent The proof of Proposition~\ref{lemlos} is elementary, but
we include the proof for completeness in
Appendix~\ref{lemlosproof}.

\begin{definition}\label{Odef}
For every $\varphi\in\overline{\ls}(\sqcup_{\alpha\in \mathcal{A}}
I_{\alpha})$ and $\mathcal{O}\in\Sigma(\pi)$ set
\begin{align*}
\mathcal{O}(\varphi)=\lim_{x\to
0^+}\left(\sum_{\alpha\in\mathcal{A},\pi_0(\alpha)\in\mathcal{O}}
\varphi(r_\alpha-x)-
\sum_{\alpha\in\mathcal{A},\pi_0(\alpha)-1\in\mathcal{O}}\varphi(l_\alpha+x)\right).
\end{align*}
\end{definition}
\noindent In order to prove that $\mathcal{O}(\varphi)$ is finite,
we need  the strong symmetry condition (\ref{zerosym}).
\begin{lemma}\label{lemoszo}
For every $\varphi\in\overline{\ls}(\sqcup_{\alpha\in \mathcal{A}}
I_{\alpha})$ and $\mathcal{O}\in\Sigma(\pi)$,
$\mathcal{O}(\varphi)$ is  finite. Moreover, if
$\varphi\in\ls(\sqcup_{\alpha\in \mathcal{A}} I_{\alpha})$ then
\[|\mathcal{O}(\varphi)|\leq 2d\nu(A)\frac{1}{|I|}\int_I|\varphi(x)|\,dx+2d\lv(\varphi).\]
\end{lemma}
\begin{proof} Let $a:= \min\{|I_\alpha|:\alpha\in\mathcal{A}\}/2$. Then for
$x\in(0,a)$ we have
\[\varphi(r_\alpha-x)=-C_\alpha^-\log(x)+g^-_\alpha(x)\text{ and
}\varphi(l_\alpha+x)=-C_\alpha^+\log(x)+g^+_\alpha(x),\]  where
$g^{\pm}_\alpha:[0,a]\to\R$ is of bounded variation for
$\alpha\in\mathcal{A}$. Therefore, using  the symmetry condition
(\ref{zerosym})
\begin{equation*}
\begin{split}
 \Delta(x) := &\sum_{\alpha \in \mathcal{A}_{\mathcal{O}^-}} \varphi(r_\alpha-x)-
 \sum_{ \alpha\in\mathcal{A}_\mathcal{O}^+} \varphi(l_\alpha+x) =
 - \sum_{\alpha\in\mathcal{A}_\mathcal{O}^-} C^-_{\alpha} \log(x) +
 \sum_{\alpha \in \mathcal{A}_\mathcal{O}^-}g^-_\alpha(x) \\ &
+ \sum_{ \alpha\in\mathcal{A}_\mathcal{O}^+}C^+_{\alpha} \log(x)  -
\sum_{ \alpha\in\mathcal{A}_\mathcal{O}^+}g^+_\alpha(x)
 = \sum_{\alpha \in \mathcal{A}_\mathcal{O}^-} g^-_\alpha(x)-
\sum_{ \alpha\in\mathcal{A}_\mathcal{O}^+} g^+_\alpha(x).
\end{split}
\end{equation*}
%\begin{align*}
% \Delta&(x) := \sum_{\mathcal{A}_{\mathcal{O}^-}} \varphi(r_\alpha-x)-
% \sum_{\begin{subarray}{c} \alpha\in\mathcal{A}\\ \pi_0(\alpha)-1\in\mathcal{O}\end{subarray}} \varphi(l_\alpha+x)  \\& =
% - \sum_{\begin{subarray}{c} \alpha\in\mathcal{A} \\ \pi_0(\alpha)\in\mathcal{O}\end{subarray}} C^-_{\alpha} \log(x)
%- \sum_{\begin{subarray}{c} \alpha\in\mathcal{A} \\\pi_0(\alpha)-1\in\mathcal{O}\end{subarray}}C^+_{\alpha} \log(x)
% + \sum_{\begin{subarray}{c} \alpha\in\mathcal{A} \\ \pi_0(\alpha)\in\mathcal{O}\end{subarray}}
%g^-_\alpha(x)-
%\sum_{\begin{subarray}{c} \alpha\in\mathcal{A} \\ \pi_0(\alpha)-1\in\mathcal{O}\end{subarray}}g^+_\alpha(x)\nonumber \\
%& = \sum_{\begin{subarray}{c}\alpha\in\mathcal{A} \\
%\pi_0(\alpha)\in\mathcal{O}\end{subarray}} g^-_\alpha(x)-
%\sum_{\begin{subarray}{c} \alpha\in\mathcal{A} \\
%\pi_0(\alpha)-1\in\mathcal{O}\end{subarray}} g^+_\alpha(x).
%\end{align*}
It follows that $\mathcal{O}(\varphi)$ is finite and given by
\begin{equation}\label{Ousinggvalues}
\mathcal{O}(\varphi)=\Delta_+(0)=\sum_{\alpha \in \mathcal{A}_\mathcal{O}^-}
(g^-_\alpha)_+(0)-
\sum_{ \alpha\in\mathcal{A}_\mathcal{O}^+}(g^+_\alpha)_+(0).
\end{equation}
Suppose now  that $\varphi\in\ls(\sqcup_{\alpha\in \mathcal{A}}
I_{\alpha})$ is of the form (\ref{fform}).  Then $g^\pm_\alpha$
are absolutely continuous and  $|(g^+_\alpha)'(x)|\leq
\bl(\varphi)/a+|g_\varphi'(l_\alpha+x)|$ and
$|(g^-_\alpha)'(x)|\leq \bl(\varphi)/a+|g_\varphi'(r_\alpha-x)|$,
and hence
\[|\Delta'(x)|\leq
\frac{2d\bl(\varphi)}{a}+\sum_{\alpha\in\mathcal{A}}(|g_\varphi'(l_\alpha+x)|+|g_\varphi'(r_\alpha-x)|)\text{
for }x\in[0,a].\] Therefore, for $x,y\in[0,a]$,
\begin{align}\label{difdelta}
\begin{split}
|\Delta(x)-\Delta(y)|&\leq
2d\bl(\varphi)+\sum_{\alpha\in\mathcal{A}}(\int_x^y|g_\varphi'(l_\alpha+t)|dt+\int_x^y|g_\varphi'(r_\alpha-t)|dt)\\
&\leq
2d\bl(\varphi)+\sum_{\alpha\in\mathcal{A}}(\int_{l_\alpha}^{l_\alpha+a}|g_\varphi'(t)|dt+
\int_{r_\alpha-a}^{r_\alpha}|g_\varphi'(t)|dt)\\
&\leq 2d\bl(\varphi)+\int_{I}|g_\varphi'(t)|dt=2d\bl(\varphi)+\var g_\varphi.
\end{split}
\end{align} Moreover, using the definition of $a$ and (\ref{balancedintervals}), one has
\begin{align*}
|m(\Delta,[0,a])|&\leq\sum_{\alpha\in\mathcal{A},\pi_0(\alpha)\in\mathcal{O}}
\left|m(\varphi,[r_\alpha,r_\alpha-a])\right|+
\sum_{\alpha\in\mathcal{A},\pi_0(\alpha)-1\in\mathcal{O}}\left|m(\varphi,[l_\alpha,l_\alpha+a])\right|\\
&\leq \frac{1}{a}\int_I|\varphi(x)|\,dx\leq
2d\nu(A)\frac{1}{|I|}\int_I|\varphi(x)|\,dx.
\end{align*}
In view of the previous equation and (\ref{difdelta}), it follows
that for all $x\in[0,a]$,
\begin{equation*}
|\Delta(x)|\leq \sup_{y\in[0,a]} |\Delta(x)-\Delta(y)| + m(\Delta,
[0,a])  \leq
\frac{2d\nu(A)}{|I|}\int_I|\varphi(x)|\,dx+2d\bl(\varphi)+\var
g_\varphi ,
\end{equation*}
which completes the proof.
\end{proof}
Remark that if $\varphi\in\bv(\sqcup_{\alpha\in \mathcal{A}}
I_{\alpha})$ and $\mathcal{O}\in\Sigma(\pi)$
\begin{equation}\label{odlabv}
\mathcal{O}(\varphi)=
\sum_{\alpha\in\mathcal{A},\pi_0(\alpha)\in\mathcal{O}}\varphi_-(r_{\alpha})-
\sum_{\alpha\in\mathcal{A},\pi_0(\alpha)-1\in\mathcal{O}}\varphi_+(l_{\alpha}).
\end{equation}
Hence, Definition~\ref{Odef} extends the definition of the
operator $\mathcal{O}$ used by \cite{Co-Fr}  for
$\varphi\in\bv(\sqcup_{\alpha\in \mathcal{A}} I_{\alpha})$.
Moreover, if $\varphi\in\ac(\sqcup_{\alpha\in \mathcal{A}}
I_{\alpha})$ then
\begin{equation}
\label{sprzezo}
\sum_{\mathcal{O}\in\Sigma(\pi)}\mathcal{O}(\varphi)=\sum_{\alpha\in\mathcal{A}}\varphi_-(r_{\alpha})-
\sum_{\alpha\in\mathcal{A}}\varphi_+(l_{\alpha})=s(\varphi).
\end{equation}

\begin{remark}\label{Ohzeroreformulation}
If we identify the piecewise constant function $h(x) = \sum_\alpha
h_\alpha \chi_{I_\alpha}(x)$ (where $\chi_{I_\alpha}$ is the
characteristic function of $I_\alpha$) with the vector  $h =
(h_{\alpha})_{\alpha\in\mathcal{A}}$, note also that
\[\mathcal{O}(h)=\sum_{\pi_0(\alpha)\in\mathcal{O}}h_{\alpha}-
\sum_{\pi_0(\alpha)-1\in\mathcal{O}}h_{\alpha}=\sum_{\alpha\in\mathcal{A}}
(\chi_{\mathcal{O}}(\pi_0(\alpha))-\chi_{\mathcal{O}}(\pi_0(\alpha)-1))h_\alpha=\langle
h,b(\mathcal{O})\rangle,\] where $b(\mathcal{O})$, $\mathcal{O}\in
\Sigma$ are the vectors  defined in (\ref{bdef}).  In particular,
Lemma \ref{bcharh} can be restated saying that the vector $h\in
H_\pi$ if and only if for the corresponding function $h$ we have
$\mathcal{O}(h)=0$ for every $\mathcal{O}\in\Sigma(\pi)$.
\end{remark}

\section{Renormalization of cocycles}\label{renormalization:sec}
Assume that  $T$ is of periodic type and recall that we denote by
$T^{(k)} = \mathcal{R}^{kp}(T)$ the sequence or Rauzy iterates
corresponding  to multiples of the period $p>0$.
\begin{remark}\label{Keanecase}
The definitions and Lemmas in \S\ref{SpecialBS} hold more in
general for any IET satisfying the Keane condition and any
subsequence $( T^{(k)} )_{k\geq 0}$ which is of the form
$(\mathcal{R}^{n_k}T)_{k\geq0}$ for some subsequence $(n_k)_{k\geq
0}$ of iterates of Rauzy-Veech induction.
\end{remark}

\subsection{Special Birkhoff sums}\label{SpecialBS}
 For every measurable cocycle $\varphi:I^{(k)}\to\R$ for the IET
$T^{(k)}:I^{(k)}\to I^{(k)}$ and $k'>k$ denote by
$S(k,k')\varphi:I^{(k')}\to\R$ the renormalized cocycle for
$T^{(k')}$ given by
\[S(k,k')\varphi(x)=\sum_{0\leq i<Q_{\beta}(k,k')}\varphi((T^{(k)})^ix)\text{ for }x\in I^{(k')}_\beta.\]
We write $S(k)\varphi$ for $ S(0,k)\varphi$ and we  adhere to  the
convention that $S(k,k)\varphi=\varphi$. Sums of this form are
usually called \emph{ special Birkhoff sums}. If $\varphi$ is
integrable then
\begin{equation}\label{nase}
\| S(k,k')\varphi\|_{L^1(I^{(k')})}\leq
\|\varphi\|_{L^1(I^{(k)})}\;\;\text{ and }
\end{equation}
\begin{equation}\label{zachcal}
\int_{I^{(k')}}S(k,k')\varphi(x)\,dx=
\int_{I^{(k)}}\varphi(x)\,dx.
\end{equation}
 Note that the operator $S(k,k')$
maps $\overline{\ol}(\sqcup_{\alpha\in \mathcal{A}}
I^{(k)}_{\alpha})$ into $\overline{\ol}(\sqcup_{\alpha\in
\mathcal{A}} I^{(k')}_{\alpha})$. In view of (\ref{zachcal}),
$S(k,k')$ maps the space $\overline{\ol}_0(\sqcup_{\alpha\in
\mathcal{A}} I^{(k)}_{\alpha})$ into
$\overline{\ol}_0(\sqcup_{\alpha\in \mathcal{A}}
I^{(k')}_{\alpha})$. Moreover, we will show below
(Lemma~\ref{invariophi}) that it also maps $\ls(\sqcup_{\alpha\in
\mathcal{A}} I^{(k)}_{\alpha})$ into $\ls(\sqcup_{\alpha\in
\mathcal{A}} I^{(k')}_{\alpha})$.   If $ g\in\bv(\sqcup_{\alpha\in
\mathcal{A}}I^{(k)}_{\alpha})$ then
\begin{equation}\label{navar}
\var S(k,k') g\leq \var g.
\end{equation}

The following three Lemmas (Lemma~\ref{comparingconstants},
\ref{invariophi} and \ref{comparingsingularities}) allow us to
compare the singularities of $S(k,k') \varphi$ with the
singularities of $\varphi$. Here the assumption that $\varphi$ is
of geometric type plays a crucial role, since functions with symmetric
singularities not of geometric type are  not renormalized by the
operation of considering special Birkhoff sums.

%The first part of the Lemma relates the constants of the
%logarithmic singularities of $S(k,k') \varphi$ and $\varphi$,
%while the second part shows that, with exactly two exception, each
%orbit of a  discontinuties of $T^{(k')}$ under $T^{(k)}$ and
%$\hat{T}^{(k)}$ up to the first return time to $I^{(k')}$ contain
%exactly one discontinuity of $T^{(k)}$

%and it also shows that the the operator $S(k,k')$ maps
%$\ls(\sqcup_{\alpha\in\mathcal{A}} I^{(k)}_{\alpha})$ into
%$\ls(\sqcup_{\alpha\in\mathcal{A}} I^{(k')}_{\alpha})$.

\begin{lemma}\label{comparingconstants}
For each $k'\geq k\geq 0$ and for each
$\varphi\in\ol(\sqcup_{\alpha\in \mathcal{A}}I^{(k)}_{\alpha})$
 of the form
\[\varphi(x)=-\sum_{\alpha\in\mathcal{A}}\left(C^+_\alpha
\log
\left(|I^{(k)}|\left\{\frac{x-l^{(k)}_\alpha}{|I^{(k)}|}\right\}\right)+C^-_\alpha
\log
\left(|I^{(k)}|\left\{\frac{r^{(k)}_\alpha-x}{|I^{(k)}|}\right\}\right)\right)\]
there exists a permutation $\chi:\mathcal{A}\rightarrow
\mathcal{A}$ such that
\begin{align*}
S(k,k')\varphi(x)=&-\sum_{\alpha\in\mathcal{A}} C^+_\alpha
\log(|I^{(k')}|\{(x-l^{(k')}_\alpha)/|I^{(k')}|\}) \\&  -\sum_{\alpha\in\mathcal{A}}
C^-_{\chi(\alpha)}\log(|I^{(k')}|\{(r^{(k')}_\alpha-x)/|I^{(k')}|\}
+g(x),
\end{align*}
where $g\in\bv^1(\sqcup_{\alpha\in
\mathcal{A}}I^{(k')}_{\alpha})$. In particular,
$\bl(S(k,k')\varphi)=\bl(\varphi)$.\end{lemma}

%and $R^{p(k'-k)}$ is the $p(k'-k)$
%iterate\footnote{For the more general case in Remark
%\ref{Keanecase}, when $T$ is satisty the Keane condition and
%$T^{(k)}= \mathcal{R}^{n_k}T$,  the iterate is $R^{n_k'-n_k}$.} of
%a function $R: \mathbb{R}^\mathcal{A} \rightarrow
%\mathbb{R}^\mathcal{A} $.

% or $k'\geq k\geq 0 $ and
%$\varphi\in\ol(\sqcup_{\alpha\in \mathcal{A}} I^{(k)}_{\alpha}) $,
%\begin{equation}\label{invarianceL} \bl(S(k,k')\varphi)=\bl(\varphi).\end{equation}

\begin{proof}
We will prove the Lemma for special Birkhoff sums corresponding to
one single step of Rauzy induction. The proof then follows by
induction on Rauzy steps. Let $\alpha_0:= \pi_0^{-1}(d)$ and
$\alpha_1  := \pi_1^{-1}(d) $. Let write
$C^-=(C^-_\alpha)_{\alpha\in \mathcal{A}} $ for the vector in
$\mathbb{R}^\mathcal{A}$ whose components are the constants
$C_\alpha^-$.  For $\upsilon=0,1$ let
\[G_{(\pi,\lambda)}^\upsilon=\{C^-=(C^-_\alpha)_{\alpha\in\mathcal{A}}\in
\R^{\mathcal{A}}:C^-_{\alpha_\upsilon}=0\}.\] Let us consider
$R:G_{(\pi,\lambda)}^0\cup G_{(\pi,\lambda)}^1\to
G_{\mathcal{R}(\pi,\lambda)}^{\vep(\pi,\lambda)}$ be given by
\begin{equation} \label{Rdef}
R(C^-)_\alpha=\left\{\begin{array}{ccc}
C^-_\alpha&\text{ if }&\alpha\neq \alpha_0,\alpha_1, \\
C^-_{\alpha_0}+C^-_{\alpha_1}&\text{ if }&\alpha= \alpha_{1-\vep(\pi,\lambda)},\\
0&\text{ if }&\alpha= \alpha_{\vep(\pi,\lambda)}.
\end{array}\right.\end{equation}
Recall that for
$(\pi^{1},\lambda^{1})=\mathcal{R}(\pi,\lambda)$ we have
$\pi^{1}_{\vep(\pi,\lambda)}(\alpha_{\vep(\pi,\lambda)})=
\pi_{\vep(\pi,\lambda)}(\alpha_{\vep(\pi,\lambda)})=d$, so
$R(C^-)\in G_{\mathcal{R}(\pi,\lambda)}^{\vep(\pi,\lambda)}$.
If $\varphi\in\ol(\sqcup_{\alpha\in
\mathcal{A}}I_{\alpha})$ is of the form
\[\varphi(x)=-\sum_{\alpha\in\mathcal{A}}(C^+_\alpha
\log(|I|\{(x-l_\alpha)/|I|\})+C^-_\alpha\log(|I|\{(r_\alpha-x)/|I|\})),\]
then since the singularities are of geometric type, $C^-=(C^-_\alpha)_{\alpha\in\mathcal{A}}\in
G_{(\pi,\lambda)}^\upsilon$ for some $\upsilon=0,1$. Denote by
$S^1\varphi$ the special Birkhoff sum corresponding to one step of
Rauzy-Veech induction, given by \begin{equation}\label{S1def}
S^1\varphi(x)= \sum_{0\leq i<\Theta(T)_\beta} \varphi(T^i(x)),
\qquad \mathrm{for}\ x\in  I^1_{\beta}. \end{equation} Analyzing
the effect of one step of Rauzy induction, one can then verify
that
\begin{align}\label{onesteprenormalization}
 S^1\varphi(x) =&
-\sum_{\alpha\in\mathcal{A}} \left(C^+_\alpha
\log(|I^{1}|\{(x-l^{1}_\alpha)/|I^{1}|\})\right.\\&
+\nonumber
\left.R(C^-)_\alpha\log(|I^{1}|\{(r^{1}_\alpha-x)/|I^{1}|\})\right)
+g_1(x),
\end{align}
where $g_1\in\bv^1(\sqcup_{\alpha\in
\mathcal{A}}I^{1}_{\alpha})$. See Figure \ref{renormlog}.
%Note that if $\varphi \in \ol(\sqcup_{\alpha\in \mathcal{A}}
%I^{(0)}_{\alpha})$, the geometric type condition ensures that
%either $C^-_{\alpha_0}$ or $C^-_{\alpha_1}$ are zero and that,
%since we already remarked that
%$\pi^{1}_{\vep(\pi,\lambda)}(\alpha_{\vep(\pi,\lambda)})=d$, we
%know that  $\alpha_{\upsilon(\pi,\lambda)} \in \{\alpha_0 ,
%\alpha_1\}$. Thus, from the explicit expression of $R$ given above
%we have that  $\bl( S^1\varphi) = \bl(\varphi)$.
For $\upsilon=0,1$, define the permutation
$\chi_{(\pi,\lambda)}^{\upsilon}:\mathcal{A}\to\mathcal{A}$ by
\begin{equation*}
\chi_{(\pi,\lambda)}^{\upsilon}(\alpha_{\vep(\pi,\alpha)})=\alpha_{\upsilon},\;
\chi_{(\pi,\lambda)}^{\upsilon}(\alpha_{1-\vep(\pi,\alpha)})=\alpha_{1-\upsilon},
\; \chi_{(\pi,\lambda)}^{\upsilon}(\alpha)=\alpha \text{ for }
\alpha\notin \{ \alpha_{0}, \alpha_1\} .
\end{equation*}
Remark then that since $\varphi \in G^\upsilon$, $\alpha_\upsilon \in \{ \alpha_0,\alpha_1\}$
is such that $C^-_{\alpha_\upsilon}=0$. Thus, one can verify that
$R(C^-)_\alpha=C^-_{\chi(\alpha)}$ for all $\alpha\in\mathcal{A}$.
%where $R$ is the function defined in the proof of Lemma~\ref{comparingconstants}.
For $0\leq k<k'$ and $\varphi \in G^\upsilon$, if we denote by
$\vep_{j}=\vep(\pi^{j},\lambda^{j})$, we can let
$\chi:=\chi^\upsilon(k,k'):\mathcal{A}\to\mathcal{A}$ stand for the
permutation
\[\chi:= \chi^\upsilon(k,k')=\chi_{(\pi^{pk},\lambda^{pk})}^{\upsilon}\circ
\chi_{(\pi^{pk+1},\lambda^{pk+1})}^{\vep_{pk}}\circ \cdots
\circ\chi_{(\pi^{pk'-1},\lambda^{pk'-1})}^{\vep_{pk'-2}}.\]  Then
one can prove by induction on Rauzy steps that
$R^{p(k'-k)}(C^-)_{\alpha} = C^-_{\chi(\alpha)} $. This together
with $p(k'-k)$ iterations of (\ref{onesteprenormalization})
concludes the proof.
\end{proof}

%$R^{p(k'-k)}:G^{\vep_{pk}}_{(\pi^{pk},\lambda^{pk})}\to G^{\vep_{pk'}}_{(\pi^{pk'},\lambda^{pk'})}$ verifies
%\begin{equation}\label{Kchirelation}). %\end{equation}
%Remark that, since one can verify from the explicit definition of
%$\chi$ and $K$ that if $(R C^-)_{\alpha}\neq 0$  then
%$\chi(\alpha)\neq (\pi^1)_\upsilon^{-1}(d)$, we also have that if
%$(R^{k'-k}C^-)_{\alpha}\neq 0$  then $\chi(k,k')(\alpha)\neq
%\alpha^{(k)}_{\upsilon}$.\end{proof}

%\begin{corollary} \label{renormalizationLSG}
%For $k'\geq k\geq 0 $ and $\varphi\in\ol(\sqcup_{\alpha\in \mathcal{A}}
%I^{(k)}_{\alpha}) $,
%$\bl(S(k,k')\varphi)=\bl(\varphi)$.
%Moreover, for all $k'\geq k\geq 0 $  the operator $S(k,k')$ maps $\ol(\sqcup_{\alpha\in
%\mathcal{A}} I^{(k)}_{\alpha})$ into $\ol(\sqcup_{\alpha\in
%\mathcal{A}} I^{(k')}_{\alpha})$.
%\end{corollary}
%\begin{proof} Note that since $\varphi \in \ol(\sqcup_{\alpha\in \mathcal{A}}
%I^{(k)}_{\alpha})$, the geometric type condition ensures that either
%$\alpha_0$ or $\alpha_1$ are zero. Thus, from the explicit expression of $K$
%\end{proof}

\begin{figure}
\centering
\subfigure[Case $\lambda_{\alpha_0}> \lambda_{\alpha_1}$ or $\epsilon(\lambda, \pi) = 0$.]{\label{logsingre1}
\includegraphics[width=0.42\textwidth]{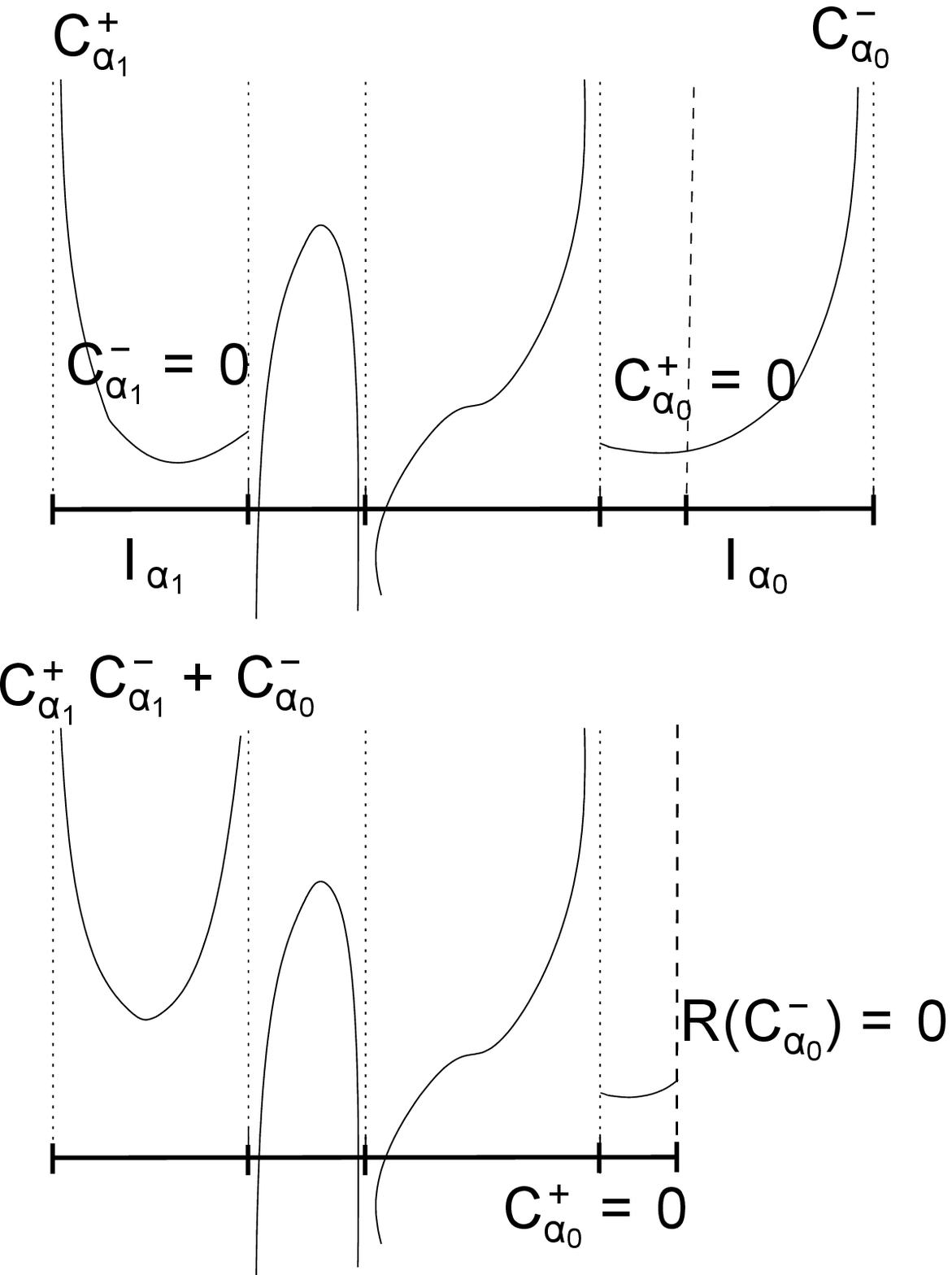}}
\hspace{6mm}
\subfigure[Case $\lambda_{\alpha_0}< \lambda_{\alpha_1}$ or $\epsilon(\lambda, \pi) = 1$.]{\label{logsingre2}
\includegraphics[width=0.45\textwidth]{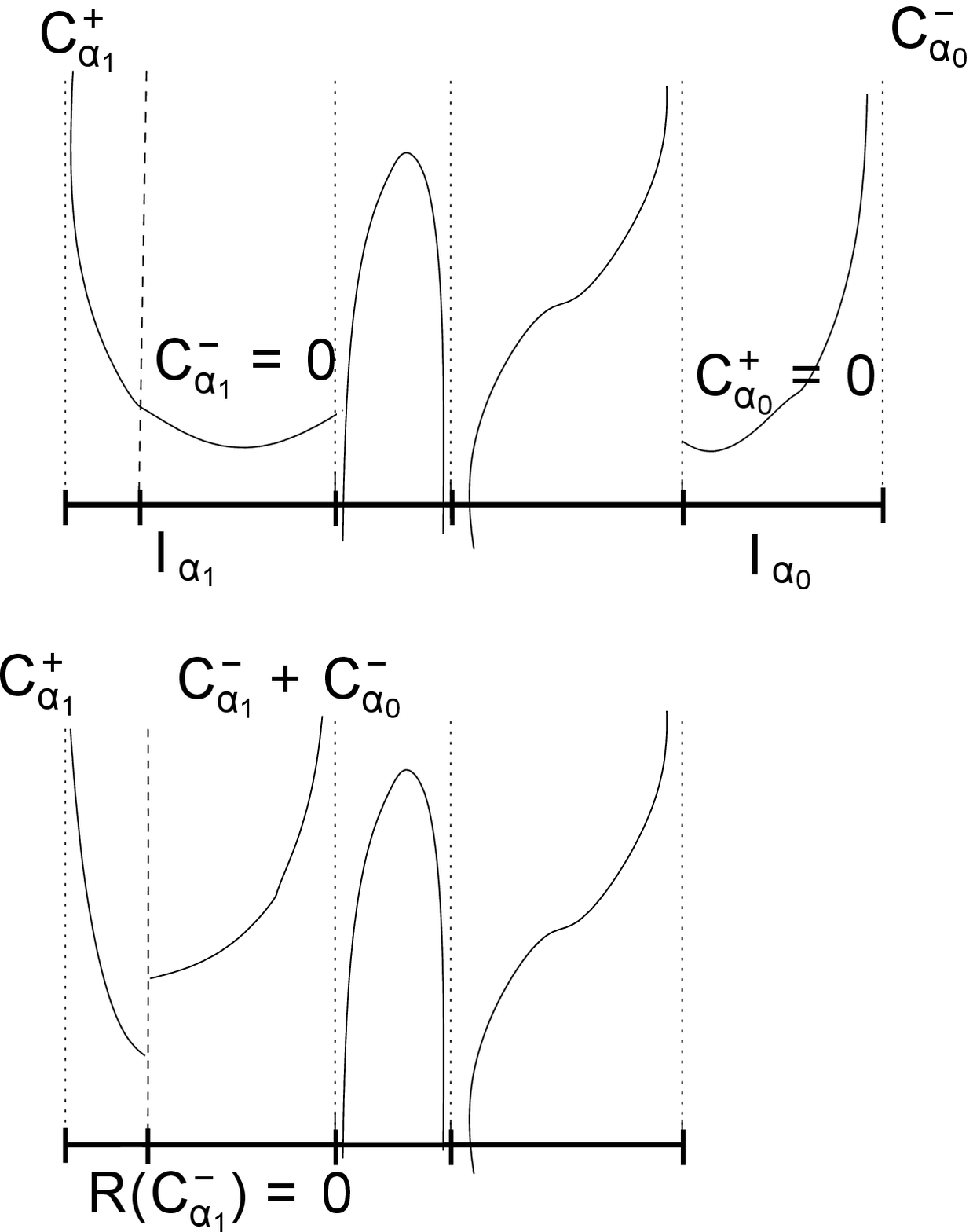}}
\caption{Renormalization of functions with logarithmic singularities of geometric type.}\label{renormlog}
\end{figure}

Consider the operator $\mathcal{O}(\varphi)$ defined in
Definition~\ref{Odef}.
\begin{lemma}\label{invariophi}
For each $k'\geq k\geq 0$ the operator $S(k,k')$ maps
$\overline{\ls}(\sqcup_{\alpha\in \mathcal{A}} I^{(k)}_{\alpha})$
 into
$\overline{\ls}(\sqcup_{\alpha\in \mathcal{A}} I^{(k')}_{\alpha})$ and
${\ls}(\sqcup_{\alpha\in \mathcal{A}} I^{(k)}_{\alpha})$ into
${\ls}(\sqcup_{\alpha\in \mathcal{A}} I^{(k')}_{\alpha})$.
Moreover, for every $\varphi \in \ls(\sqcup_{\alpha\in
\mathcal{A}} I^{(k)}_{\alpha})$ and $\mathcal{O}\in \Sigma(\pi)$,
we have $\mathcal{O}( S(k,k') \varphi  )= \mathcal{O}(\varphi)$.
\end{lemma}
\begin{proof}
Let $T=T_{\pi,\lambda}$,
$\varphi\in\overline{\ls}(\sqcup_{\alpha\in \mathcal{A}}
I_{\alpha})$ and consider the special Birkhoff sum
$\widetilde{\varphi}=S^1\varphi$ given by one step of Rauzy-Veech
induction (see (\ref{S1def})). Let $\xi$ be the correspondence
between $\Sigma(\pi)$ and $\Sigma(\pi^1)$ given by
Lemma~\ref{invario} and let $\mathcal{A}_{\mathcal{O}}^{\pm}$,
$\mathcal{O} \in \Sigma(\pi)$, the sets defined in (\ref{defAO})
and $\widetilde{\mathcal{A}}_{\mathcal{O}}^{\pm}$, $\mathcal{O}
\in \Sigma(\pi^1)$, the corresponding sets for
$(\widetilde{\pi},\widetilde{\lambda})=(\pi^1,\lambda^1)$. We will
show that
\begin{eqnarray}
\sum_{\alpha\in\mathcal{A}^+_\mathcal{O}}C^+_{\alpha} &=&
\sum_{\alpha\in\widetilde{\mathcal{A}}^+_{\xi\mathcal{O}}} C^+_{\alpha}\label{O+} \\
 \sum_{\alpha\in\mathcal{A}^-_\mathcal{O}} C^-_{\alpha} &=&
 \sum_{\alpha\in\widetilde{\mathcal{A}}^-_{\xi\mathcal{O}}} R(C^-)_{\alpha} \label{O-},
\end{eqnarray}
where $R$ is the operator defined in (\ref{Rdef}) in the proof of
Lemma~\ref{comparingconstants}. Since by
(\ref{onesteprenormalization}) the logarithmic constants for
$S^1\varphi$ are the ones which appear in the right hand side,
these two equations show that if the symmetry condition
(\ref{zerosym}) holds for $\varphi$ for all $\mathcal{O}\in
\Sigma(\pi)$, since $\xi: \Sigma(\pi) \to \Sigma(\pi^1)$ is a
bijection, the symmetry condition  holds also for $S^1\varphi$ for
all $\mathcal{O}\in \Sigma(\pi^1)$. By induction on Rauzy steps,
this shows that $S(k,k')\varphi \in
\overline{\ls}(\sqcup_{\alpha\in \mathcal{A}} I^{(k')}_{\alpha})$
for each $k'\geq k$. Let us prove (\ref{O+}, \ref{O-}). Since  $
\widetilde{\mathcal{A}}^+_{\xi\mathcal{O}}={\mathcal{A}}^+_\mathcal{O}$
by Lemma \ref{orbitcorrespondence}, (\ref{O+}) holds trivially.
%Moreover, since by the definition (\ref{Rdef}) of $R$,
%$R(C^-)_\alpha=C^-_\alpha$ for all $\alpha\notin \{
%\alpha_0,\alpha_1\}$ and  if $\mathcal{O} \notin \{\mathcal{O}_0,
%\mathcal{O}_1\} $ (where $\mathcal{O}_0,  \mathcal{O}_1$ were
%defined before Lemma~\ref{orbitcorrespondence}) or if $\mathcal{O}
%=\mathcal{O}_0= \mathcal{O}_1 $, also
%$\mathcal{A}^-_\mathcal{O}=\widetilde{\mathcal{A}}^-_\mathcal{O}$
%by Lemma~\ref{orbitcorrespondence},  also (\ref{O-}) holds
%trivially in these cases.
From the definition (\ref{Rdef})  of $R$, one immediately sees
that if $\mathcal{A}'\subset \mathcal{A}$ is a subset such that
either $\{\alpha_0,\alpha_1\}\subset \mathcal{A'}$ or
$\{\alpha_0,\alpha_1\} \cap \mathcal{A}' = \emptyset$, then
$\sum_{\alpha\in\mathcal{A}'}C^-_{\alpha} - \sum_{\alpha \in
\mathcal{A}'} R(C^-)_{\alpha}=0$. Since $\{
\alpha_0,\alpha_1\}\subset \mathcal{A}^-_{\mathcal{O}_\epsilon}$
(recall that $\pi_0(\alpha_0)=\pi_1(\alpha_1)=d\in
\mathcal{O}_\epsilon$ by definition of $\mathcal{O}_\epsilon$) and
thus $\{ \alpha_0,\alpha_1\}\cap
\mathcal{A}^-_\mathcal{O}=\emptyset$ for all $\mathcal{O}\neq
{\mathcal{O}_\epsilon}$, it follows that
\[\sum_{\alpha\in\mathcal{A}^-_\mathcal{O}}C^-_{\alpha} - \sum_{\alpha \in
\mathcal{A}^-_\mathcal{O}} R(C^-)_{\alpha}=0\quad\text{ for each
}\quad \mathcal{O}\in\Sigma(\pi).\]
 Thus, (\ref{O-}) holds also for $\mathcal{O}\in
\{\mathcal{O}_0, \mathcal{O}_1\} $ (where $\mathcal{O}_0,
\mathcal{O}_1$ were defined before Lemma
\ref{orbitcorrespondence}) or if $\mathcal{O} = \mathcal{O}_0=
\mathcal{O}_1 $, since in these cases by Lemma
\ref{orbitcorrespondence} we can have $
\mathcal{A}^-_\mathcal{O}=\widetilde{\mathcal{A}}^-_{\xi\mathcal{O}}$.
Thus, we are left to consider the case in which $\mathcal{O}
\notin \{ \mathcal{O}_0,  \mathcal{O}_1\} $ and at the same time
$\mathcal{O}_0  \neq \mathcal{O}_1$. In these cases, since by
Lemma~\ref{orbitcorrespondence} we have
$\widetilde{\mathcal{A}}^-_{\xi\mathcal{O}_\varepsilon} =
\mathcal{A}^-_{\mathcal{O}_\varepsilon}\backslash
\{\alpha_\varepsilon\}$ and $
\widetilde{\mathcal{A}}^-_{\xi\mathcal{O}_{1-\varepsilon}} =
\mathcal{A}^-_{\mathcal{O}_{1-\varepsilon}} \cup
\{\alpha_{\varepsilon}\}$, we can add or subtract
$R(C^-)_{\alpha_\varepsilon} $, which by (\ref{Rdef}) is equal to
zero, to get respectively
% have a set $\mathcal{A}'$ with the above property and recall that  $R(C^-)_{\alpha_\varep}=0$ by
\begin{align*}
&\nonumber \sum_{\alpha\in\mathcal{A}^-_{\mathcal{O}_\varepsilon}}
C^-_{\alpha} -
\sum_{\alpha\in\widetilde{\mathcal{A}}^-_{\xi\mathcal{O}_\varepsilon}}
R (C^-)_{\alpha}  =
\sum_{\alpha\in\mathcal{A}^-_{\mathcal{O}_\varepsilon}}
(C^-_{\alpha} - R(C^-)_{\alpha})  + R(C^-)_{\alpha_\varepsilon}=0,
\\ & \nonumber \sum_{\alpha\in\mathcal{A}^-_{\mathcal{O}_{1-\varepsilon}}}
C^-_{\alpha} -
\sum_{\alpha\in\widetilde{\mathcal{A}}^-_{\xi\mathcal{O}_{1-\varepsilon}}}
R (C^-)_{\alpha} =
\sum_{\alpha\in\mathcal{A}^-_{\mathcal{O}_{1-\varepsilon}}}
(C^-_{\alpha} -  R (C^-)_{\alpha})  -
R(C^-)_{\alpha_{\varepsilon}}=0,
\end{align*}
which concludes the proof of (\ref{O-}). This, together with
Lemma~\ref{comparingconstants}, is enough to conclude that
$S(k,k')$ maps the space $\overline{\ls}(\sqcup_{\alpha\in
\mathcal{A}} I^{(k)}_{\alpha})$
 into
$\overline{\ls}(\sqcup_{\alpha\in \mathcal{A}} I^{(k')}_{\alpha})$
and ${\ls}(\sqcup_{\alpha\in \mathcal{A}} I^{(k)}_{\alpha})$ into
${\ls}(\sqcup_{\alpha\in \mathcal{A}} I^{(k')}_{\alpha})$.

%the second part of the Lemma. Assume that $\varphi\in\ls(\sqcup_{\alpha\in
%\mathcal{A}} I_{\alpha})$ and $C^-\in R^\upsilon$ for $\upsilon \in \{0,1\}$.
%for $\upsilon=0,1$ we can again apply the above remark.

\smallskip
Assume now that $\varphi\in\ls(\sqcup_{\alpha\in \mathcal{A}}
I^{(k)}_{\alpha})$. Let us now prove that for each
$\mathcal{O}\in\Sigma(\pi)$, we have
$(\xi\mathcal{O})(\widetilde{\varphi})=\mathcal{O}(\varphi)$,
where $\xi$ is the bijection  given by Lemma~\ref{invario}. Let
${g}_\alpha^\pm$, $\alpha\in\mathcal{A}$, be the absolutely
continuous functions defined as in the proof of
Lemma~\ref{lemoszo}. Similarly, define also for
$\widetilde{\varphi}=S^1\varphi$  the absolutely continuous
functions
\begin{equation*} \widetilde{g}^-_\alpha(x):= \widetilde{\varphi}(r^1_\alpha-x) +
R(C^-)_\alpha\log(x), \qquad \widetilde{g}^+_\alpha(x) :=
\widetilde{\varphi}(l^1_\alpha+x) +C_\alpha^+\log(x).
\end{equation*}
In virtue of (\ref{Ousinggvalues}) and the analogous equality for
$\widetilde{\varphi}$, to prove that
$(\xi\mathcal{O})(\widetilde{\varphi})=\mathcal{O}(\varphi)$ it is
enough to prove that
\begin{equation}\label{equalitygs}
\sum_{\alpha\in \mathcal{A}^-_\mathcal{O}}%{\alpha\in\mathcal{A},\pi_0(\alpha)\in\mathcal{O}}
g^-_\alpha(0)-
\sum_{\alpha\in \mathcal{A}^+_\mathcal{O}}
%{\alpha\in\mathcal{A},\pi_0(\alpha)-1\in\mathcal{O}}
g^+_\alpha(0) = \sum_{\alpha\in \widetilde{\mathcal{A}}^-_{\xi \mathcal{O}}}
%{\alpha\in\mathcal{A},\pi_0(\alpha)\in\xi\mathcal{O}}
\widetilde{g}^-_\alpha(0)-
\sum_{\alpha\in \widetilde{\mathcal{A}}^+_{\xi\mathcal{O}}}
%{\alpha\in\mathcal{A},\pi_0(\alpha)-1\in\xi\mathcal{O}}
\widetilde{g}^+_\alpha(0),
\end{equation}
where ${\mathcal{A}}^\pm_{\mathcal{O}}$ are the sets defined in
(\ref{defAO}). The analysis of one step of Rauzy-Veech induction
shows that for all $\alpha\neq \alpha_0,\alpha_1$, we have
$\widetilde{g}^{\pm}_{\alpha}(x)=g^{\pm}_{\alpha}(x)$, while for
$\alpha \in \{\alpha_0, \alpha_{1}\}$, if
$\varepsilon=\varepsilon(\pi,\lambda)$ (see~(\ref{epsilondef})),
we have \begin{equation*}
\begin{array}{ll}
\widetilde{g}^{+}_{\alpha_\varepsilon}(x)=g^{+}_{\alpha_\varepsilon}(x),
&
\widetilde{g}^{-}_{\alpha_\varepsilon}(x)=\varphi\circ T^{-\vep}(|\lambda^1|-x); \\
\widetilde{g}^{+}_{\alpha_{1-\varepsilon}}(x)=g^{+}_{\alpha_{1-\varepsilon}}(x)
+ \varphi\circ T^{-\vep}(|\lambda^1|+x),
&\widetilde{g}^{-}_{\alpha_{1-\varepsilon}}(x)={g}^{-}_{\alpha_{1-\varepsilon}}(x)
+ {g}^{-}_{\alpha_{\varepsilon}}(x) .
\end{array}
\end{equation*}
Combining these expressions with the relations between
$\mathcal{A}^\pm_\mathcal{O}$ and
$\widetilde{\mathcal{A}}^\pm_{\xi \mathcal{O}}$ given by
Lemma~\ref{orbitcorrespondence} and recalling the definition of
$\mathcal{O}_1$ and  $\mathcal{O}_2$,
%between $\mathcal{A}^\pm_\mathcal{O}$ and $\mathcal{A}^\pm_{\xi \mathcal{O}}$
%$\widetilde{\mathcal{A}}^-_{\mathcal{O}_\varepsilon} =
%$\widetilde{\mathcal{A}}^-_{\xi
%\mathcal{O}_\varepsilon}=\mathcal{A}^-_{\mathcal{O}_\varepsilon}\setminus
%\{\alpha_\varepsilon\}$ and $
%\widetilde{\mathcal{A}}^-_{\xi\mathcal{O}_{1-\varepsilon}} =
%\mathcal{A}^-_{\mathcal{O}_{1-\varepsilon}} \cup
%\{\alpha_{\varepsilon}\}$ (see Lemma~\ref{orbitcorrespondence})
one can verify case by case that (\ref{equalitygs}) holds and thus
$(\xi\mathcal{O})(\widetilde{\varphi})=\mathcal{O}(\varphi)$. By
induction on Rauzy steps and in  view of Remark~\ref{assumexiid}
and one gets $\mathcal{O}( S(k,k') \varphi  )=
\mathcal{O}(\varphi)$.
\end{proof}

The last lemma allows us to keep track of how discontinuities of
$T^{(k')}$ are related to discontinuities of $T^{(k)}$. Let
$\alpha^{(k)}_0:= (\pi^{(k)}_0)^{-1}(d)$ and $\alpha^{(k)}_1  :=
(\pi^{(k)}_1)^{-1}(d) $.
\begin{lemma}\label{comparingsingularities}
For each $k'\geq k \geq0$, for each $\alpha \in \mathcal{A}$, we have
\begin{equation}\label{lcorrespondencekl}
l^{(k)}_{\alpha} \in\{ (T^{(k)})^j l^{(k')}_\alpha , \, 0\leq j <
Q_{\alpha} (k,k') \}.
\end{equation}
Moreover, if $\chi : \mathcal{A} \to \mathcal{A}$ is one of the
permutations\footnote{Let us point out that there are two
permutations $\chi= \chi^{0}(k,k')$, $\chi= \chi^{1}(k,k')$, given
by Lemma~\ref{comparingconstants}. In
Lemma~\ref{comparingconstants} we are given $\varphi
\in \overline{\ol}$ and if $C^- \in G^\upsilon$ (see
Lemma \ref{comparingconstants}) the function $\chi$ for which the
Lemma hold is $\chi^{\upsilon}$. On the other hand, both $\chi= \chi^{0}(k,k')$,
$\chi= \chi^{1}(k,k')$ satisfy the conclusion of Lemma \ref{comparingsingularities}.}
given by Lemma~\ref{comparingconstants},
% one of the two following holds:
% \footnote{More precisely, if $\varphi \in G^\upsilon$ so that
% $C^-{\alpha_\upsilon}=0$ and $\chi=\chi^\upsilon(k,k')$ is the
% permutation given by Lemma~\ref{comparingconstants},
% (\ref{rcorrespondencekldod}) holds for
% $\alpha=\chi^{-1}(\alpha_{1-\upsilon})$, while
% (\ref{rcorrespondencekl}) hold for all $\alpha \notin \{
% \chi^{-1}(\alpha_{1-\upsilon}),  \chi^{-1}(\alpha_{\upsilon}) =
%\alpha^{(k)}_{\epsilon(\pi^{(k')}\lambda^{(k)})}$.}
%\neq \chi^{-1} \alpha^{(k)}_{\upsilon}$,
\begin{equation}\label{rcorrespondencekl}
 r^{(k)}_{\chi(\alpha)}   \in\{  (\widehat{T^{(k)}})^j
r^{(k')}_\alpha  , \, 0\leq j < Q_{\alpha} (k,k') \} \text{ if }\
\alpha\neq \alpha^{(k')}_{\vep(\pi^{pk'-1},\lambda^{pk'-1})},
\end{equation}
while there exists ${\alpha}_\ast \in \mathcal{A} \setminus \{
\alpha^{(k')}_{\vep(\pi^{pk'-1},\lambda^{pk'-1})}\}$ such that
\begin{equation}
\label{rcorrespondencekldod} r^{(k)}_{\alpha^{(k)}_0},
r^{(k)}_{\alpha^{(k)}_1}  \in\{ (\widehat{T^{(k)}})^j
r^{(k')}_{{\alpha}_\ast}  , \, 0\leq j < Q_{\alpha} (k,k') \}.
% \qquad  \text{if}\ \chi(\alpha')=\alpha^{(k)}_{1-\upsilon}.
\end{equation}
%for some $\alpha\neq \alpha^{(k')}_{\vep(\pi^{pk'-1},\lambda^{pk'-1})}$, and
%$\chi(\alpha^{(k')}_{\vep(\pi^{pk'-1},\lambda^{pk'-1})})=\alpha^{(k)}_\upsilon$.
Moreover, if $C^-_{\chi(\alpha)}\neq 0$ then $\alpha\neq
\alpha^{(k')}_{\vep(\pi^{pk'-1},\lambda^{pk'-1})}$ and
(\ref{rcorrespondencekl})  holds.
%$(R^{k'-k}K)_{\alpha}\neq 0$  then
%$\chi(k,k')(\alpha)\neq \alpha^{(k)}_{\upsilon}$, so if
%$(R^{k'-k}C^-)_{\alpha}\neq 0$ then (\ref{rcorrespondencekl}) holds.
\end{lemma}
\begin{proof}
Let us prove the Lemma for one step of Rauzy induction. We refer
the reader to Figure \ref{Rauzy}. Let
$\chi=\chi_{(\pi,\lambda),\upsilon}:\mathcal{A}\to\mathcal{A}$ by
the permutation for one step of Rauzy-Veech induction defined in
the proof of Lemma~\ref{comparingconstants}.
%setting $\chi(\alpha_{\vep(\pi,\alpha)})=\alpha_{\upsilon}$,
%$\chi(\alpha_{1-\vep(\pi,\alpha)})=\alpha_{1-\upsilon}$ and
%$\chi(\alpha)=\alpha$, for all other $\alpha$. Remark then that
%$R(C^-)_\alpha=C^-_{\chi(\alpha)}$ for all $\alpha\in\mathcal{A}$
%where $R$ is the function defined in t
 Let $\vep=\vep(\pi,\lambda)$. Then $\chi(\alpha_{\vep})=\alpha_{\upsilon}$. By the
definition of Rauzy-Veech induction, if $l^1_\alpha$ and
$r^1_\alpha$ denote the endpoints of $T^1=\mathcal{R}(T)$, we have
$l_{\alpha}=l^1_{\alpha}$ for $\alpha\neq \alpha_{1-\vep}$ and
$l_{\alpha_{1-\vep}}=T^{\vep}l^1_{\alpha_{1-\vep}}$. Moreover,
$r_{\alpha}=r^1_{\alpha}$ for $\alpha\neq \alpha_0,\alpha_1$, and
$r_{\alpha_0}=\widehat{T}r^1_{\alpha_{1-\vep}}$,
$r_{\alpha_1}=r^1_{\alpha_{1-\vep}}$. Since $\Theta(T)_\alpha=1$
for $\alpha\neq \alpha_{1-\vep}$ and
$\Theta(T)_{\alpha_{1-\vep}}=2$, it follows that for every
$\alpha\in\mathcal{A}$ we have $l_{\alpha}=T^jl^1_\alpha$ for some
$0\leq j<\Theta(T)_\alpha$ and for every $\alpha\neq
\alpha_{\vep}$ (equivalently $\chi(\alpha)\neq \alpha_{\upsilon}$)
we have $r_{\chi(\alpha)}=\widehat{T}^jr^1_{\alpha}$ for some
$0\leq j<\Theta(T)_\alpha$. Moreover,
$r_{\alpha_\upsilon}=\widehat{T}^jr^1_{\alpha,}$ for some $0\leq
j<\Theta(T)_{\alpha'}$, where $\chi(\alpha')=\alpha_{1-\upsilon}$.
The proof of the formulas in the Lemma then follows by induction
on Rauzy steps. We are left to prove the last remark.

If $C^-_{\chi^\upsilon(k,k')(\alpha)}\neq 0$ then since
$R^{p(k'-k)}(C^-)_\alpha = C^-_{\chi^\upsilon(k,k')(\alpha)}$ (see
the end of the proof of Lemma \ref{comparingconstants}) also
$R^{p(k'-k)}(C^-)_\alpha  \neq 0$. Since $R^{p(k'-k)}$ maps the
space $G^0_{(\pi^{(k)},\lambda^{(k)})}\cup
G^1_{(\pi^{(k)},\lambda^{(k)})}$ to
$G^{\vep(\pi^{pk'-1},\lambda^{pk'-1})}_{(\pi^{(k')},\lambda^{(k')})}$,
which is the space of functions with
$R^{p(k'-k)}(C^-)_{\alpha^{(k')}_{\vep(\pi^{pk'-1},\lambda^{pk'-1})}}
= 0$,
%$R^{p(k'-k)}:G^0_{(\pi^{(k)},\lambda^{(k)})}\cup G^1_{(\pi^{(k)},\lambda^{(k)})}
%\to G^{\vep(\pi^{pk'-1},\lambda^{pk'-1})}_{(\pi^{(k')},\lambda^{(k')})}$,
this shows that  $\alpha\neq
\alpha^{(k')}_{\vep(\pi^{pk'-1},\lambda^{pk'-1})}$.
\end{proof}

\begin{remark}\label{notidentity}
Even if $T$ is of periodic type, we cannot, up to replacing $p$ by
a multiple, assume that $R: \mathbb{R}^\mathcal{A} \rightarrow
\mathbb{R}^\mathcal{A} $ and $\chi : \mathcal{A}\rightarrow
\mathcal{A}$ are the identity maps. This can be assumed, though,
if we replace $T$ by $\mathcal{R}(T)$.
\end{remark}

%The following Lemma will be useful in the proof of ergodicity in \S\ref{ergodicitysec}.
%\begin{lemma}\label{remark:wybj} For each $k\geq0$ let
%$\overline{\alpha}\in\mathcal{A}$ be such that $\pi_0^{(k)}
%(\overline{\alpha}) = 1$. Let $k'\geq k+1$. Then
%\begin{equation}\label{controlledheight}
%Q_{\overline{\alpha}}(k,k'-1)\leq j_\alpha<Q_\alpha(k,k')\text{
%whenever }C^-_{\chi(k,k')(\alpha)}\neq 0.
%\end{equation}
%%\end{lemma}
%%\begin{proof}
% By Lemma~\ref{comparingconstants}, for each
%$\alpha\in\mathcal{A}$, with $\chi(k,k')(\alpha)\neq
%\alpha_\upsilon$, there exists $0\leq j_\alpha<Q_\alpha(k,k')$
%such that
%$(\widehat{T^{(k)}})^{j_\alpha}r^{(k')}_{\alpha}=r^{(k)}_{\chi(k,k')(\alpha)}$.
%Remark that since $Z(k-1,k)$  has positive entries we know that
%$I^{(k')}\subset I^{(k'-1)}$, since each $x\in I^{(k')}$ has to
%visit
%$I^{(k'-1)}_{\alpha}$  $\Int I^{(k'')}_{\overline{\alpha}}$, where
%$\overline{\alpha}\in\mathcal{A}$ satisfies $\pi^{k''}_0
%(\overline{\alpha}) = 1$. Moreover, the intervals
%$(\widehat{T^{(k)}})^j\Int I^{(k'')}_{\overline{\alpha}}$ for
%$0\leq j<Q_{\overline{\alpha}}(k,k'')$ do not contain any point
%$r^{(k)}_\alpha$, $\alpha\in\mathcal{A}$. It is then easy to see
%that (\ref{controlledheight}) holds.
%\end{proof}

\subsection{Cancellations for symmetric singularities.}\label{symmetric:sec}
The following property of cocycles with symmetric  logarithmic
singularities was proved by the second author in \cite{Ul:abs}
(see Proposition 4.1) and will play a crucial role to renormalize
cocycles with symmetric logarithmic singularities and in the proof
of ergodicity.
%Let $S_r(\varphi, T)(x)$ denote the Birkhoff sum  $\sum_{i=0}^{r-1} \varphi(T^i(x))$.

\begin{proposition}[\cite{Ul:abs}] \label{cancellations_prop}
Let $\pi\in\mathcal{S}^0_{\mathcal{A}}$. For a.e.\
$\lambda\in\R^{\mathcal{A}}_+$, $|\lambda|=1$ there exist a
constant $M$ and sequence of induction times $(n_k)_{k\in
\mathbb{N}}$ for the corresponding IET $T_{(\pi,\lambda)}$ such
that for each $\varphi\in\logs(\sqcup_{\alpha\in \mathcal{A}}
I_{\alpha})$ with $g_\varphi'=0$, whenever $x \in
I^{(n_{k})}_{\beta}$ for some $k\geq 0$ and $0< r \leq Q_{\beta}(
n_k)$, we have\footnote{In the statement of Proposition 4.1
\cite{Ul:abs}, only $\varphi^{(r)} (x)$ appears in the absolute
value, while $ \sum_{\alpha \in \mathcal{A}}
\frac{C_\alpha^+}{x_\alpha^{l}} $ and $ \sum_{\alpha \in
\mathcal{A}} \frac{C_\alpha^-}{x_\alpha^{r}} $ appear as bounds.
In the proof, though, the contribution of the closest points is
subtracted first and the statement here given is proven. The
explicit dependence of the constant $M $ in  Proposition 4.1
\cite{Ul:abs} on $\varphi$ (via $\bl(\varphi)$) can also be easily
extrapolated from the proof.}
\begin{equation} \label{generalsum}
\left| (\varphi')^{(r)} (x) - \sum_{\alpha \in \mathcal{A}}
\frac{C_\alpha^+}{x_\alpha^{l}} + \sum_{\alpha \in \mathcal{A}}
\frac{C_\alpha^-}{x_\alpha^{r}} \right| \leq  M \bl(\varphi) r,
\end{equation}
where $x_\alpha^{l}$ and $x_\alpha^{r}$ are the closest points
respectively to $l_\alpha$ and $r_\alpha$, which, denoting by
$(x)^{+}$ the positive part of $x$ (i.e.~$(x)^{+}=x$ if $x\geq 0$
and $(x)^{+}=\infty$ if $x<0$, so that if $x<0$ then $1/(x)^+$ is zero) are given by
\begin{equation*}
x_\alpha^{l} = \min_{0\leq i < r} (T^i x - l_\alpha)^+, \qquad
x_\alpha^{r} = \min_{0\leq i < r} (r_\alpha -T^i x)^+  .
\end{equation*}
\end{proposition}
\begin{remark}\label{cancellationsperiodic}
One can check that if $T$ is of periodic type,  the estimate in
Proposition~\ref{cancellations_prop} holds and furthermore one can
take as $(n_k)_{k\in \mathbb{N}}$ simply the multiples of a period
of Rauzy-Veech induction\footnote{The interested reader can
patiently go through the definitions of further accelerations of
Rauzy-induction in \cite{Ul:abs} which lead to the construction of
sequence $(n_k)_{k\in \mathbb{N}}$ in Proposition
\ref{cancellations_prop} and check that if $T$ is of periodic type
the period multiples satisfies all the assumptions without need of
extracting subsequences.}, i.e.~one can take $n_k=pk$ where $p$ is
the period. Moreover, the constant $M$ depends only on the period
matrix of Rauzy Veech induction.
\end{remark}
In virtue of the Remark, applying the estimate
(\ref{generalsum}) to each renormalized iterate of Rauzy-Veech
induction for a IET of periodic type, we get the following.
% to each IET $\overline{T}^{(k)}$ obtained renormalizing $T^{(k)}$
%by $|I^{(k)}|$ so that the exchanged subintervals $$, we get the following.

\begin{corollary}\label{cancellations_periodic}
If $T$ is of periodic type, there exist a constant $M$  such that
the following hold.   For all $0\leq k<k' $ for each
$\varphi\in\logs(\sqcup_{\alpha\in \mathcal{A}} I^{(k)}_{\alpha})$
with $g_\varphi'=0$, whenever $x \in I^{(k')}_{\beta}$,
$\beta\in\mathcal{A}$  and $0< r \leq Q_{\beta}(k,k')$, we have
\begin{equation}\label{generalsum1}
\left| \sum_{0\leq j<r}\varphi'((T^{(k)})^jx) - \sum_{\alpha \in
\mathcal{A}} \frac{C_\alpha^+}{({x}_\alpha^l)^{(k)}} +
\sum_{\alpha \in \mathcal{A}}
\frac{C_\alpha^-}{({x}_\alpha^r)^{(k)}}  \right| \leq
\frac{1}{|I^{(k)}|}M \bl(\varphi) r,
\end{equation}
where   $({x}_\alpha^l)^{(k)}$ and $({x}_\alpha^r)^{(k)}$ are
given by
\begin{equation*}
({x}_\alpha^l)^{(k)} = \min_{0\leq i < r} (({T}^{(k)})^i x -
{l}^{(k)}_\alpha)^+, \qquad ({x}_\alpha^r)^{(k)} = \min_{0\leq i <
r} ({r}^{(k)}_\alpha -({T}^{(k)})^i x)^+  .
\end{equation*}
\end{corollary}

\begin{proof}
Let us denote by $\overline{T}^{(k)}:I^{(0)}\to I^{(0)}$
($I^{(0)}=[0,1)$) the normalized IET associated to $T^{(k)}$,
i.e.\ $\overline{T}^{(k)}x={|I^{(k)}|}^{-1} T^{(k)}(|I^{(k)}|x)$.
As $T$ is of periodic type, $\overline{T}^{(k)}=T$. Let us
consider $\varphi_k:I^{(0)}\to\R$ given by
$\varphi_k(x)=\varphi(|I^{(k)}|x)$. Then one can check that
$\varphi_k\in\logs(\sqcup_{\alpha\in \mathcal{A}}
I^{(0)}_{\alpha})$ with $\bl(\varphi_k)=\bl(\varphi)$ and
$g_{\varphi_k}'=0$. By Proposition~\ref{cancellations_prop}  and
Remark \ref{cancellationsperiodic}, whenever $y \in
I^{(k'-k)}_{\beta}$, $\beta\in\mathcal{A}$  and $0< r \leq
Q_{\beta}(k-k')$, we have
\begin{equation}\label{nierpos}
\left| (\varphi_k')^{(r)} (y) - \sum_{\alpha \in \mathcal{A}}
\frac{C_\alpha^+}{y_\alpha^{l}} + \sum_{\alpha \in \mathcal{A}}
\frac{C_\alpha^-}{y_\alpha^{r}} \right| \leq  M \bl(\varphi) r.
\end{equation}
Fix $x\in I^{(k')}_\beta$ and $0< r \leq
Q_{\beta}(k,k')=Q_{\beta}(k-k')$.  Since
$l_\alpha^{(j)}=|I^{(j)}|l_\alpha$,
$r_\alpha^{(j)}=|I^{(j)}|r_\alpha$ for all $\alpha\in\mathcal{A}$
and $j\geq 0$, we have $y:=x/|I^{(k)}|\in
I^{(k'-k)}_\beta$ and
\[({T}^{(k)})^i x -{l}^{(k)}_\alpha=|I^{(k)}|((\overline{T}^{(k)})^iy-l_{\alpha}),\;\;
{r}^{(k)}_\alpha-({T}^{(k)})^i
x=|I^{(k)}|(r_{\alpha}-(\overline{T}^{(k)})^iy).
\]
Therefore, $|I^{(k)}|y_\alpha^l=(x_\alpha^l)^{(k)}$ and
$|I^{(k)}|y_\alpha^r=(x_\alpha^r)^{(k)}$. As
$\varphi'_k(y)=|I^{(k)}|\varphi'(|I^{(k)}|y)=|I^{(k)}|\varphi'(x)$,
in view of (\ref{nierpos}), it follows that
\begin{multline*}
\left| \sum_{0\leq j<r}\varphi'((T^{(k)})^jx) - \sum_{\alpha \in
\mathcal{A}} \frac{C_\alpha^+}{({x}_\alpha^l)^{(k)}} +
\sum_{\alpha \in \mathcal{A}}
\frac{C_\alpha^-}{({x}_\alpha^r)^{(k)}}  \right|\\
\left| \sum_{0\leq
j<r}\frac{\varphi_k'((\overline{T}^{(k)})^jy)}{|I^{(k)}|} -
\sum_{\alpha \in \mathcal{A}}
\frac{C_\alpha^+}{|I^{(k)}|y_\alpha^l} + \sum_{\alpha \in
\mathcal{A}} \frac{C_\alpha^-}{|I^{(k)}|y_\alpha^r}
\right|\leq\frac{M\bl(\varphi)r}{|I^{(k)}|},
\end{multline*}
which completes the proof.
\end{proof}

Let us show that functions with logarithmic singularities of
geometric type behave well under the renormalization given by
taking special Birkhoff sums.

\begin{proposition}\label{lemlogreno}
If $T$ has periodic type %and period $p$
 then there exists $c>0$
such that if $\varphi\in\logs(\sqcup_{\alpha\in \mathcal{A}}
I^{(k)}_{\alpha})$ and
\begin{equation*}\label{formvarphik}
\varphi(x)=-\sum_{\alpha\in\mathcal{A}}(C^+_\alpha\log(|I^{(k)}|\{(x-l^{(k)}_\alpha)/|I^{(k)}|\})+
C^-_\alpha\log(|I^{(k)}|\{(r^{(k)}_\alpha-x)/|I^{(k)}|\})),
\end{equation*}
then for every $k'\geq k$ we have
$S(k,k')\varphi=\overline{\varphi}+\widetilde{\varphi}$, where
\begin{align}\label{varphibar}
\begin{split}
\overline{\varphi}(x)=&-\sum_{\alpha\in\mathcal{A}}\left(C^+_\alpha\log(|I^{(k')}|\{(x-l^{(k')}_\alpha)/|I^{(k')}|\})\right.\\&+\left.
C^-_{\chi(\alpha)}\log(|I^{(k')}|\{(r^{(k')}_\alpha-x)/|I^{(k')}|\})\right),
\end{split}
\end{align}
$\chi:\mathcal{A}\to\mathcal{A}$ is a permutation and
$\widetilde{\varphi}\in\bv^1(\sqcup_{\alpha\in \mathcal{A}}
I^{(k')}_{\alpha})$ with $\|\widetilde{\varphi}'\|_{\sup}\leq
\frac{c\bl(\varphi)}{|I^{(k')}|}$.
\end{proposition}

\begin{proof}
Let $\chi:\mathcal{A}\to\mathcal{A}$ be the permutation given by
Lemma~\ref{comparingconstants}. If  $\overline{\varphi}$ is
defined by (\ref{varphibar}), Lemma~\ref{comparingconstants} gives
that $S(k,k')\varphi=\overline{\varphi}+\widetilde{\varphi}$ where
$\widetilde{\varphi}\in\bv^1(\sqcup_{\alpha\in \mathcal{A}}
I^{(k')}_{\alpha})$ (where $\widetilde{\varphi}$ is the $g$ in
Lemma~\ref{comparingconstants}). Thus, we need to estimate
$\|\widetilde{\varphi}'\|_{\sup}$. By differentiating
$\widetilde{\varphi} = S(k,k')\varphi -\overline{\varphi}$, we
have
\begin{equation}\label{firstdifference}
\widetilde{\varphi}'(x) =S(k,k')\varphi'(x) -
\sum_{\alpha\in\mathcal{A}} \frac{C^+_\alpha}{|I^{(k')}| \left\{
\frac{x-l^{(k')}_\alpha}{|I^{(k')}|}\right\}} +
\sum_{\alpha\in\mathcal{A}} \frac{C^-_{\chi(\alpha)}}{ |I^{(k')}|
\left\{ \frac{r^{(k')}_\alpha -x}{ |I^{(k')} |}\right\}}.
\end{equation}
From Corollary~\ref{cancellations_periodic}, if $x \in
I^{(k')}_\beta$ then
% and the trivial inequality  $Q_{\alpha}(k,k' )|I^{(n)}| \leq 1$ we also have
\begin{equation}\label{fromcancellations}
\left| S(k,k')\varphi'(x) - \sum_{\alpha \in \mathcal{A}}
\frac{C_\alpha^+}{( x_\alpha^{l})^{(k)}} +   \sum_{\alpha \in
\mathcal{A}} \frac{C_\alpha^-}{(x_\alpha^{r})^{(k)}}  \right|
\leq\frac{M\bl(\varphi)Q_\beta(k,k')}{|I^{(k)}|} ,
\end{equation}
where  \[( x_\alpha^{l})^{(k)} = \min_{0\leq i < Q_{\beta}(k,k')}
((T^{(k)})^i x - l^{(k)}_\alpha)^+,\;\;(x_\alpha^{r})^{(k)} =
\min_{0\leq i < Q_{\beta}(k,k')} (r^{(k)}_\alpha -(T^{(k)})^i
x)^+.\] Recall that, by \eqref{balancedintervals},
$|I_\beta^{(k')}|\geq |I^{(k')}|/d\nu(A)$ for any symbol $\beta
\in \mathcal{A}$ and from \eqref{areas}
\begin{equation}\label{szaqi}
|I^{(k')}|Q_\beta(k,k')\leq |I^{(k)}|.
\end{equation}
 Let us now show that for each $\alpha \in
\mathcal{A}$,
\begin{eqnarray}
&& \left|  \frac{C_\alpha^+}{(x_\alpha^{l})^{(k)}} -
\frac{C^+_\alpha}{|I^{(k')}| \left\{
\frac{x-l^{(k')}_\alpha}{|I^{(k')}|}\right\} }\right| \leq \frac{2
d\nu(A)\bl(\varphi)}{|I^{(k')}|} ,\qquad
\label{singularitiescomparisons1}
\\ && \left| \frac{C_{\chi(\alpha)}^-}{(x_{\chi(\alpha)}^{r})^{(k)}} - \frac{C^-_{\chi(\alpha)}}{
|I^{(k')}| \left\{ \frac{r^{(k')}_\alpha -x}{ |I^{(k')}
|}\right\}} \right| \leq \frac{2 d\nu(A)\bl(\varphi)}{|I^{(k')}|}.
\label{singularitiescomparisons2}
\end{eqnarray}
By (\ref{lcorrespondencekl}) in Lemma~\ref{comparingsingularities}, for every
$\alpha\in\mathcal{A}$ there exists $0\leq j_\alpha <
Q_{\alpha}(k,k')$ such that $(T^{(k)})^{j_\alpha} l_\alpha^{(k')}
= l_\alpha^{(k)}$.
Assume that $x \in I^{(k')}_\beta$. Since the iterates
$(T^{(k)})^j x$ for $0\leq j < Q_{\beta}(k,k')$ each belong to a
$T^j I^{(k')}_\beta$, which, for the $j$ considered are all
disjoint, we have that
\[(x_\beta^{l})^{(k)} = \min_{0\leq i < Q_\beta(k,k')} ((T^{(k)})^i
x-l_\beta^{(k)})^+ = (T^{(k)})^{j_\beta} x -l_\beta^{(k)} .\]
Moreover, since $(T^{(k)})^{j_\beta} $ is an isometry on
$I^{(k')}_\beta$
\[(x_\beta^{l})^{(k)} =  (T^{(k)})^{j_\beta} x -(T^{(k)})^{j_\beta}l_\beta^{(k')} = x -
l_\beta^{(k')}={|I^{(k')}|\{(x-l^{(k')}_\beta)/|I^{(k')}|\}},\]
which shows that in this case the left hand side of
(\ref{singularitiescomparisons1})  is zero and
(\ref{singularitiescomparisons1}) holds trivially  for
$\alpha=\beta$. Consider now $\alpha \in \mathcal{A}\setminus \{
\beta \}$. Since only $(T^{(k)})^{j_\alpha} I^{(k')}_\alpha$
contains $l_\alpha^{(k)}$ as left endpoint and it is disjoint from
$(T^{(k)})^j I^{(k')}_\beta$ for $0\leq j<Q_\beta(k,k')$, we have
that both $|I^{(k')}|\{(x-l^{(k')}_\alpha)/|I^{(k')}|\}$ and
$x_\alpha^{l} $ are greater than $|I^{(k')}_\alpha |\geq
|I^{(k')}|/d\nu(A) $. This concludes  the proof of the upper bound
in (\ref{singularitiescomparisons1}) for all $\alpha \in
\mathcal{A}$.

To prove (\ref{singularitiescomparisons2}), recall that
Lemma~\ref{comparingsingularities} also gives that whenever $ C^-_{\chi(\alpha)}\neq 0$
\begin{equation}\label{rcorrespondencekl1}
r^{(k)}_{\chi(\alpha)}   \in\{  (\widehat{T^{(k)}})^j
r^{(k')}_\alpha  , \, 0\leq j < Q_{\alpha} (k,k') \}.
\end{equation}
Thus, when $C_{\chi(\alpha)}^{-} \neq 0$,
(\ref{singularitiescomparisons2}) can be proved using
(\ref{rcorrespondencekl1}) in a completely analogous way. On the
other hand, if $C^-_{\chi(\alpha)}= 0$,  there is nothing to
prove, since the left hand side of
(\ref{singularitiescomparisons2}) is identically zero. We now get
$\|\widetilde{\varphi}'\|_{\sup}\leq
\frac{C\bl(\varphi)}{|I^{(k')}|}$ by combining
(\ref{firstdifference}), (\ref{fromcancellations})  and
(\ref{szaqi}) with the sum over $\alpha \in \mathcal{A}$ of
(\ref{singularitiescomparisons1},
\ref{singularitiescomparisons2}).
\end{proof}

\begin{proposition}\label{lemlogreno1}
If $T$ has periodic type then there exists $C>0$ such that, for
all $0\leq k\leq k'$,  if
$\varphi\in\overline{\logs}(\sqcup_{\alpha\in \mathcal{A}}
I^{(k)}_{\alpha})$ then
\begin{equation}\label{nave}
\lv(S(k,k')\varphi)\leq C\lv(\varphi).
\end{equation}
\end{proposition}

\begin{proof}
Let $\varphi=\varphi_0+g$ be the decomposition with $
g\in\bv(\sqcup_{\alpha\in \mathcal{A}} I^{(k)}_{\alpha})$ and
\[\varphi_0(x)=-\sum_{\alpha\in\mathcal{A}}(C^+_\alpha\log(|I^{(k)}|\{(x-l^{(k)}_\alpha)/|I^{(k)}|\})+
C^-_\alpha\log(|I^{(k)}|\{(r^{(k)}_\alpha-x)/|I^{(k)}|\})).\] By
Proposition~\ref{lemlogreno},
$S(k,k')\varphi_0=\overline{\varphi}+\widetilde{\varphi}$, where
\[\overline{\varphi}(x)\!=-\!\sum_{\alpha\in\mathcal{A}}(C^+_\alpha\!\log(|I^{(k')}|\{(x-l^{(k')}_\alpha)/|I^{(k')}|\})+
C^-_{\chi(\alpha)}\!\log(|I^{(k')}|\{(r^{(k')}_\alpha-x)/|I^{(k')}|\}))\]
for a permutation $\chi:\mathcal{A}\to\mathcal{A}$, and a function
$\widetilde{\varphi}\in\bv^1(\sqcup_{\alpha\in \mathcal{A}}
I^{(k')}_{\alpha})$ with $\|\widetilde{\varphi}'\|_{\sup}\leq
c{\bl(\varphi)}/{|I^{(k')}|}$. Thus,
\[\var\widetilde{\varphi}=
\sum_{\alpha\in\mathcal{A}}\int_{I^{(k')}_\alpha}
|\widetilde{\varphi}'(x)|\,dx\leq c\, \bl(\varphi).\] Since
$\var(S(k,k') g)\leq\var g$ and
$\bl(\overline{\varphi})=\bl(\varphi)$, it follows that
\[\lv(S(k,k')\varphi)=
\bl(\overline{\varphi})+\var(\widetilde{\varphi}+ S(k,k')g)\leq
(c+1)\bl(\varphi)+\var g\leq (c+1)\lv(\varphi).\]
\end{proof}

%The following Remark will be used later.
%\begin{remark}\label{remarkcohzeroo}
%%If $\varphi\in\ac(\sqcup_{\alpha\in \mathcal{A}} I_{\alpha})$ is a
%coboundary with a continuous transfer function then
%$\mathcal{O}(\varphi)=0$ for all $\mathcal{O}\in\Sigma(\pi)$.
%Indeed, if $\varphi=g-g\circ T$ and $g:I\to\R$ is continuous then by (\ref{})
%\[|\mathcal{O}(\varphi)|=|\mathcal{O}(S(k)\varphi)|\leq 2d\|S(k)\varphi\|_{\sup}
%\leq 2d\sup\{|g(x)-g(x')|:x,x'\in I^{(k)}\}\to 0\] as
%$k\to\infty$, hence $\mathcal{O}(\varphi)=0$.\end{remark}

\section{Correction operators}\label{correction:sec}
In this section we define the operator which allows us to
\emph{correct} a cocycle with logarithmic singularities by a
piecewise constant  function, so that the special Birkhoff sums of
the corrected cocycle have controlled growth in $L_1$ norm. A
similar operator appears in \cite{MMYlinearization}, based on the
correction procedure introduced in \cite{Ma-Mo-Yo}. In our
setting, we need to use of the $L_1$ norm, since the $L_\infty$
norm is unbounded due to the presence of singularities. We control
the contribution coming from the singularities through the results
in \S\ref{symmetric:sec}.

Recall that $\overline{\ls}_0(\sqcup_{\alpha\in \mathcal{A}}
I_{\alpha})={\ls}_0(\sqcup_{\alpha\in \mathcal{A}}
I_{\alpha})+{\bv}_0(\sqcup_{\alpha\in \mathcal{A}} I_{\alpha})$
(see \S\ref{cocycles:sec}).
\begin{comment}
Let us consider the space $\overline{\ol}_0(\sqcup_{\alpha\in
\mathcal{A}} I_{\alpha})={\ol}_0(\sqcup_{\alpha\in \mathcal{A}}
I_{\alpha})+{\bv}_0(\sqcup_{\alpha\in \mathcal{A}} I_{\alpha})$,
i.e.\ the space of all functions with logarithmic singularities of
geometric type and zero mean of the form (\ref{fform}) for which
we  require only that $g_\varphi\in\bv( \sqcup_{\alpha\in
\mathcal{A}} I_{\alpha})$. This space equipped with the norm
$\|\varphi\|_{\lv}=\bl(\varphi)+\|g_\varphi\|_{\bv}$ becomes a
Banach space for which $\ol_0(\sqcup_{\alpha\in \mathcal{A}}
I_{\alpha})$ is a dense subspace.
\end{comment}
\begin{theorem}\label{operatorcorrection}
Assume that  $T$ is of periodic type.  There exists a bounded
linear operator $\mathfrak{h} : \overline{\ls}_0(\sqcup_{\alpha\in
\mathcal{A}} I_{\alpha}^{(0)}) \to \Gamma$, where $\Gamma$  is the
space of functions which are constant on each $I_{\alpha}$, whose
image is a $g-1$ dimensional space and such that:
\begin{itemize}
\item[(1)] There exist $C_1,C_2>0$ such that, if $\varphi\in {\overline{\ls}_0(\sqcup_{\alpha\in
\mathcal{A}} I_{\alpha}^{(0)})}$ and $\mathfrak{h}(\varphi) = 0$,
then for each $k\geq 1$ we have
\[\frac{1}{|I^{(k)}|}\|S(k)({\varphi})\|_{L^1(I^{(k)})} \leq
C_1\lv(\varphi)k^{M}
+C_2\|{\varphi}\|_{L^1(I^{(0)})}/|I^{(0)}|k^{M-1},\]
%\left( C_1 \lv(\varphi)+C_2\frac{1}{|I^{(0)}|}\|{\varphi}\|_{L^1(I^{(0)})} \right) k^{M-1} , \]
where $M$ is the  maximal size of Jordan
blocks in the Jordan decomposition  of the period matrix of $T$.
\item[(2)] If additionally  $T$  is of hyperbolic periodic type
and $\varphi\in\ls_0(\sqcup_{\alpha\in \mathcal{A}}
I^{(0)}_{\alpha})$ satisfies  $\mathfrak{h}(\varphi)=0$, then for
each $k\geq 0$ we have
\[\frac{1}{|I^{(k)}|}\|S(k)({\varphi})\|_{L^1(I^{(k)})}\leq
C_1 \lv(\varphi)+C_2\frac{1}{|I^{(0)}|}\|{\varphi}\|_{L^1(I^{(0)})}.\]
\end{itemize}
\end{theorem}
Part (2) will be used to prove ergodicity of $T_\varphi$ in
\S\ref{ergodicity:sec}, while part (1) will be used in the
cohomological reduction in Appendix~\ref{cohreduction:sec}. We
prove them in parallel since the proofs have similar structure.

\smallskip
Let $\Gamma^{(k)}$ be the space of real valued functions on
$I^{(k)}$ which are constant on each $I^{(k)}_{\alpha}$,
$\alpha\in\mathcal{A}$ and $\Gamma^{(k)}_0$ is the subspace of
functions with zero mean.  Then
\[S(k,k')\Gamma^{(k)}=\Gamma^{(k')}\;\;\text{ and }\;\;S(k,k')\Gamma^{(k)}_0=\Gamma^{(k')}_0.\]
Let us identify every function $\sum_{\alpha\in \mathcal{A}}
h_{\alpha}\chi_{I^{(k)}_{\alpha}}$ in $\Gamma^{(k)}$ with the
vector $h=(h_{\alpha})_{\alpha\in
\mathcal{A}}\in\R^{\mathcal{A}}$. Clearly $\Gamma^{(k)}$ is
isomorphic to $\mathbb{R}^{\mathcal{A}}$ ($\simeq\mathbb{R}^d$).
Under the identification,
\[\Gamma^{(k)}_0=Ann(\lambda^{(k)}):=\{h=(h_{\alpha})_{\alpha\in \mathcal{A}}
\in\R^{\mathcal{A}}:\langle h, \lambda^{(k)}\rangle=0\}\] and the
operator $S(k,k')$ is the linear automorphism of
$\R^{\mathcal{A}}$ whose matrix in the canonical basis is
$Q(k,k')^t$ (see for example \cite{Ma-Mo-Yo}). Thus
$S(k,k')^{-1}:\Gamma^{(k')}\to \Gamma^{(k)}$ is well defined.

Suppose now that $T$ is of periodic type, with period matrix $A$.
Then the $L^1$-norm on $\Gamma^{(k)}$ is equivalent to the vector
norm and, by (\ref{balancedintervals}),
\begin{equation}\label{compnorm}
\frac{1}{d\nu(A)}|I^{(k)}|\|h\|\leq\min_{\alpha\in\mathcal{A}}
|I_\alpha^{(k)}|\|h\|\leq\|h\|_{L^1(I^{(k)})}\leq
|I^{(k)}|\|h\|.
\end{equation}

 Let us consider the linear subspaces
\begin{align*}
&
\Gamma^{(k)}_{cs}=\{h\in\Gamma^{(k)}:\limsup_{j\to+\infty}\frac{1}{j}
\log\|S(k,j)h\|=\limsup_{j\to+\infty}\frac{1}{j}\log\|(A^t)^{j-k}
h\|\leq 0\},\\
&\Gamma^{(k)}_{s}=\{h\in\Gamma^{(k)}:\limsup_{j\to+\infty}\frac{1}{j}
\log\|S(k,j)h\|=\limsup_{j\to+\infty}\frac{1}{j}\log\|(A^t)^{j-k}
h\|<
0\},\\
&\Gamma^{(k)}_{u}=\{h\in\Gamma^{(k)}:\limsup_{j\to+\infty}\frac{1}{j}
\log\|(A^t)^{k-j}h\|< 0\}.
\end{align*}
 Let $M$ be the  maximal size of Jordan
blocks in the Jordan decomposition  of the period matrix $A$.
Note that for every natural $k$ the subspace
$\Gamma^{(k)}_{cs}$ (respectively $\Gamma^{(k)}_{s},\Gamma^{(k)}_{u})\subset\R^{\mathcal{A}}$
is the direct sum of invariant subspaces associated to Jordan
blocks of $A^t$ with non-positive (respectively negative, positive)
Lyapunov exponents. It follows that there exist
$C,\theta_+,\theta_->0$ such that
\begin{eqnarray}\label{censtab}
\|(A^t)^{n}h\| &\leq& Cn^{M-1}\|h\|\text{ for all
}h\in\Gamma^{(k)}_{cs}\text{ and }n\geq 0.
\\
\label{stab}
\|(A^t)^{n}h\|&\leq& C\exp(-n\theta_-)\|h\|\text{ for all }
h\in\Gamma^{(k)}_{s}\text{ and }n\geq 0.
\\ \label{unstab}
\|(A^t)^{-n}h\|&\leq& C\exp(-n\theta_+)\|h\|\text{ for all }
h\in\Gamma^{(k)}_{u}\text{ and }n\geq 0.
\end{eqnarray}

It is easy to show that $\Gamma^{(k)}_{cs}\subset \Gamma^{(k)}_0$.
Denote by
\[U^{(k)}:\overline{\ls}(\sqcup_{\alpha\in \mathcal{A}} I^{(k)}_{\alpha})\to
\overline{\ls}(\sqcup_{\alpha\in \mathcal{A}}
I^{(k)}_{\alpha})/\Gamma^{(k)}_{cs}\] the projection on the
quotient space.
 Let us consider two linear operators
$C^{(k)}:\overline{\ls}_0(\sqcup_{\alpha\in \mathcal{A}}
I^{(k)}_{\alpha})\to\Gamma^{(k)}_0$ and
$P^{(k)}_0:\overline{\ls}_0(\sqcup_{\alpha\in \mathcal{A}}
I^{(k)}_{\alpha})\to\overline{\ls}_0(\sqcup_{\alpha\in
\mathcal{A}} I^{(k)}_{\alpha})$ given by
\[
C^{(k)}\varphi=\sum_{\alpha\in\mathcal{A}}m(\varphi,I^{(k)}_\alpha)\chi_{I^{(k)}_\alpha}\quad\text{ and
}\quad P^{(k)}_0\varphi=\varphi-C^{(k)}\varphi.
\]
Then $m(P^{(k)}_0\varphi,I^{(k)}_\alpha)=0$ for each $\alpha\in
\mathcal{A}$. Moreover,
\begin{equation}\label{nace}
\|C^{(k)}\varphi\|_{L^1(I^{(k)})}\leq\|\varphi\|_{L^1(I^{(k)})}
\end{equation}
and, by equation (\ref{meanv2}) in Proposition~\ref{lemlos},
\begin{equation}\label{nape}
\|P^{(k)}_0\varphi\|_{L^1(I^{(k)})}\leq 8 |I^{(k)}|\lv (\varphi).
\end{equation}

Since $S(k,k')\Gamma^{(k)}_{cs}=\Gamma^{(k')}_{cs}$ and $S(k,k'):
\Gamma^{(k)}\to\Gamma^{(k')}$ is invertible (see \cite{Ma-Mo-Yo}),
the quotient linear transformation
\[S_u(k,k'):\overline{\ls}(\sqcup_{\alpha\in \mathcal{A}} I^{(k)}_{\alpha})/\Gamma^{(k)}_{cs}
\to\overline{\ls}(\sqcup_{\alpha\in \mathcal{A}}
I^{(k')}_{\alpha})/\Gamma^{(k')}_{cs}\] is well defined and
$S_u(k,k'):\Gamma^{(k)}/\Gamma^{(k)}_{cs}\to\Gamma^{(k')}/\Gamma^{(k')}_{cs}$
is invertible. Moreover,
\begin{equation}\label{splatanies}
S_u(k,k')\circ U^{(k)}\varphi=U^{(k')}\circ S(k,k')\varphi\text{
for }\varphi\in\overline{\ls}(\sqcup_{\alpha\in \mathcal{A}}
I^{(k)}_{\alpha}).
\end{equation}

Since
$\R^{\mathcal{A}}=\Gamma^{(0)}=\Gamma^{(0)}_{cs}\oplus\Gamma^{(0)}_{u}$,
the linear operators $A^t:\Gamma^{(0)}_{u}\to\Gamma^{(0)}_{u}$ and
$A^t:\Gamma^{(0)}/\Gamma^{(0)}_{cs}\to\Gamma^{(0)}/\Gamma^{(0)}_{cs}$
are isomorphic. In view of (\ref{unstab}), it follows that there
exists $C'>0$ such that
\[\|(A^t)^{-n}(h+\Gamma^{(0)}_{cs})\|\leq C'\exp(-n\theta_+)\|h+\Gamma^{(0)}_{cs}\|\]
for all $h+\Gamma^{(0)}_{cs}\in\Gamma^{(0)}/\Gamma^{(0)}_{cs}$ and
$n\geq 0$. Consequently,
\begin{equation}\label{szacowanieniest}
\|(S_u(k,k'))^{-1}(h+\Gamma^{(k')}_{cs})\|\leq
C'\exp(-(k'-k)\theta_+)\|h+\Gamma^{(k)}_{cs}\|
\end{equation}
for  $h+\Gamma^{(k')}_{cs}\in\Gamma^{(k')}/\Gamma^{(k')}_{cs}$,
$0\leq k< k'$.

\begin{lemma}\label{thmcorre}
For every function $\varphi\in \overline{\ls}_0(\sqcup_{\alpha\in
\mathcal{A}} I^{(k)}_{\alpha})$, the following limit exists in
$\Gamma_0^{(k)}/\Gamma_{cs}^{(k)}$:
\begin{equation}\label{ciagdop}
\Delta P^{(k)}\varphi = \lim_{k'\rightarrow \infty} U^{(k)}\circ
S(k,k')^{-1}\circ \left(S(k,k')\circ P^{(k)}_0-P_0^{(k')}\circ
S(k,k')\right)\varphi. \end{equation} Moreover, there exists
$K>0$ such that
\begin{equation}\label{szacdelty}
\|\Delta P^{(k)}\varphi\|\leq K  \lv (\varphi)\text{ for every
}\varphi\in\overline{\ls}_0(\sqcup_{\alpha\in \mathcal{A}}
I^{(k)}_{\alpha})\text{ and }k\geq 0.
\end{equation}
\end{lemma}
\begin{proof}
Let us first show that given
$\varphi\in\overline{\ls}_0(\sqcup_{\alpha\in \mathcal{A}}
I^{(k)}_{\alpha})$, one has
\begin{equation}\label{prelobs}
(S(k,k')\circ P_0^{(k)}-P_0^{(k')}\circ
S(k,k'))\varphi=C^{(k')}\circ S(k,k')\circ
P_0^{(k)}\varphi\in\Gamma^{(k')}_0.
\end{equation} As
$\varphi=P_0^{(k)}\varphi+C^{(k)}\varphi$, we have
\[P_0^{(k')}\circ S(k,k')\varphi=P_0^{(k')}\circ S(k,k')\circ P_0^{(k)}\varphi+
P_0^{(k')}\circ S(k,k')\circ C^{(k)}\varphi.\]
Since $S(k,k')\circ C^{(k)}\varphi\in\Gamma^{(k')}_0$, we obtain
$P_0^{(k')}\circ S(k,k')\circ C^{(k)}\varphi=0$, and hence
\begin{align*}
S(k,k')\circ P_0^{(k)}\varphi-P_0^{(k')}\circ S(k,k')\varphi
&=S(k,k')\circ P_0^{(k)}\varphi-P_0^{(k')}\circ S(k,k')\circ P_0^{(k)}\varphi\\
&=C^{(k')}\circ S(k,k')\circ P_0^{(k)}\varphi\in\Gamma^{(k')}_0.
\end{align*}
In view of (\ref{prelobs}), for $0\leq k\leq k'$, using the
telescopic nature of the expression below, we have
\begin{align*}
S&(k,k')\circ P_0^{(k)}-P_0^{(k')}\circ S(k,k')
\\
&=\sum_{k\leq r<k'}\left(S(r,k')\circ P_0^{(r)}\circ
S(k,r)-S(r+1,k')\circ P_0^{(r+1)}\circ S(k,r+1)\right)
\\
&=\sum_{k\leq r<k'}\left(S(r+1,k')\circ \left(S(r,r+1)\circ  P_0^{(r)}-
P_0^{(r+1)}\circ S(r,r+1)\right) \circ S(k,r)\right)
\\
&=\sum_{k\leq r<k'}S(r+1,k')\circ C^{(r+1)}\circ S(r,r+1) \circ
P_0^{(r)}\circ S(k,r)
\end{align*}
and the operator takes values in the subspace $\Gamma^{(k')}_0$
which is included in the domain of the operator $S(k,k')^{-1}$.
Thus, in view of (\ref{splatanies}),
\begin{align*}
U^{(k)}&\circ S(k,k')^{-1}\circ( S(k,k')\circ P_0^{(k)}-
P_0^{(k')}\circ S(k,k'))
\\
&=\sum_{k\leq r<k'} U^{(k)} \circ S(k,r+1)^{-1}\circ
C^{(r+1)}\circ S(r,r+1)\circ P_0^{(r)}\circ S(k,r)
\\
&=\sum_{k\leq r<k'}S_u(k,r+1)^{-1} \circ U^{(r+1)}\circ
C^{(r+1)}\circ S(r,r+1)\circ P_0^{(r)}\circ S(k,r).
\end{align*}
Moreover, using (\ref{nace}), (\ref{nase}), (\ref{nape}) and
(\ref{nave})
 consecutively we obtain for $ k\leq r< k'$,
\begin{align*}
\|C^{(r+1)}&\circ S(r,r+1)\circ P_0^{(r)}\circ
S(k,r)\varphi\|_{L^1(I^{(r+1)})}
\\ & \leq\|S(r,r+1)\circ P_0^{(r)}\circ S(k,r)\varphi\|_{L^1(I^{(r+1)})}\leq
\|P_0^{(r)}\circ S(k,r)\varphi\|_{L^1(I^{(r)})}\\ & \leq
8|I^{(r)}|\cdot \lv (S(k,r)\varphi)\leq 8
C|I^{(r)}|\lv(\varphi).
\end{align*}
  By (\ref{compnorm}),
\begin{align*}
\|&C^{(r+1)}\circ S(r,r+1)\circ P_0^{(r)}\circ
S(k,r)\varphi\|\\&\leq 8d\nu(A)C\frac{|I^{(r)}|}{|I^{(r+1)}|}
\lv(\varphi)\leq 8d\nu(A)\|A\|
C\lv(\varphi).
\end{align*}
Next let consider the series in
$\Gamma^{(k)}_0/\Gamma^{(k)}_{cs}$ given by
\begin{equation}\label{defoperdel}
\sum_{r\geq k}(S_u(k,r+1))^{-1}\circ U^{(r+1)}\circ C^{(r+1)}\circ
S(r,r+1)\circ P_0^{(r)}\circ S(k,r)\varphi.
\end{equation}
Since $\|U^{(r+1)}\|=1$ and $U^{(r+1)}\circ C^{(r+1)}\circ
S(r,r+1)\circ P_0^{(r)}\circ S(k,r)\varphi\in
\Gamma_0^{(r+1)}/\Gamma^{(r+1)}_{cs}$, by (\ref{szacowanieniest}),
the norm of the $r$-th element of the series  (\ref{defoperdel})
is bounded from above by
$8dC'C\nu(A)\|A\|\exp(-(r-k)\theta_+)
 \lv(\varphi)$. As
\begin{equation*}K:= \sum_{r\geq
k}8dC'C\nu(A)\|A\|\exp(-(r-k)\theta_+)  <+\infty,
\end{equation*} the series
(\ref{defoperdel}) %is bounded by $K \lv(\varphi)$, so it
 converges in
$\Gamma_0^{(k)}/\Gamma^{(k)}_{cs}$. Since,  as shown above, the limit in
(\ref{ciagdop}) is the limit of the sequence of partial sums of
the series (\ref{defoperdel}), this gives that $\Delta
P^{(k)}\varphi$ is well defined. Moreover, since the constant $K$
 is independent on $k$, we get
(\ref{szacdelty}). The proof is complete.
\end{proof}

\begin{definition}\label{wzorp} Let $P^{(k)}:\overline{\ls}_0(\sqcup_{\alpha\in
\mathcal{A}}
I^{(k)}_{\alpha})\to\overline{\ls}_0(\sqcup_{\alpha\in
\mathcal{A}} I^{(k)}_{\alpha})/\Gamma^{(k)}_{cs}$ be the operator
given by $P^{(k)}=U^{(k)}\circ P_0^{(k)}-\Delta P^{(k)}$.
\end{definition}
\begin{remark}\label{Pzero}
Note that if $\varphi\in\Gamma^{(k)}_0$ then
$P_0^{(k')}(S(k,k')\varphi)=0$ for all $k'\geq k$, hence $\Delta
P^{(k)}\varphi=0$ and $P^{(k)}\varphi=0$.
\end{remark}
\noindent The correction $\Delta P^{(k)}$ is defined so that
$P^{(k)}$ has the crucial property of commuting with the special
Birkhoff sums operators, as shown by the next Lemma.
\begin{lemma}\label{commutes}
For all $0\leq k\leq k'$ and
$\varphi\in\overline{\ls}_0(\sqcup_{\alpha\in \mathcal{A}}
I^{(k)}_{\alpha})$ we have
\begin{equation}
\label{przem}S_u(k,k')\circ P^{(k)}\varphi=P^{(k')}\circ
S(k,k')\varphi.
\end{equation}
Moreover,
\begin{equation}\label{szak}
\|P^{(k)}\varphi\|_{L^1(I^{(k)})/\Gamma^{(k)}_{cs}}\leq(8+K)|I^{(k)}|
 \lv (\varphi).
\end{equation}
\end{lemma}
\begin{proof}%of}{Lemma}{commutes}
For $k\leq k'\leq r$,  one can verify that
\begin{align*}
S&(k,k')\circ \left(P^{(k)}_0-S(k,r)^{-1}\circ \left(S(k,r)\circ
P^{(k)}_0-P^{(r)}_0\circ S(k,r)\right)\right)\\&= \left(
P^{(k')}_0-S(k',r)^{-1}\circ\left(S(k',r)\circ
P^{(k')}_0-P^{(r)}_0\circ S(k',r)\right)\right)\circ S(k,k').
\end{align*}
 In view of (\ref{splatanies}), it follows that
\begin{align*}
S_u&(k,k')\circ U^{(k)}\circ \left(P^{(k)}_0-S(k,r)^{-1} \circ
\left(S(k,r)\circ P^{(k)}_0-P_0^{(r)}\circ
S(k,r)\right)\right)\\
&=U^{(k')}\circ S(k,k')\circ \left(P^{(k)}_0-S(k,r)^{-1} \circ
\left(S(k,r)\circ P^{(k)}_0-P_0^{(r)}\circ
S(k,r)\right)\right)\\
&=U^{(k')}\circ\left(
P^{(k')}_0-S(k',r)^{-1}\circ\left(S(k',r)\circ
P^{(k')}_0-P^{(r)}_0\circ S(k',r)\right)\right)\circ S(k,k').
\end{align*}
Taking the limit as $r\to \infty$, since for $j=k$ and $j=k'$ one has
\[
\lim_{r\to\infty}U^{(j)}\circ \left(P^{(j)}_0-S(j,r)^{-1} \circ
\left(S(j,r)\circ P^{(j)}_0-P_0^{(r)}\circ
S(j,r)\right)\right)\varphi= P^{(j)} \varphi\]
we get
$S_u(k,k')\circ P^{(k)}\varphi=P^{(k')}\circ S(k,k')\varphi$,
i.e.~(\ref{przem}).

Moreover, from the Definition~\ref{wzorp}, $\|{U^{(k)}}\|= 1$,
(\ref{nape}) and (\ref{szacdelty}), we get
\[\|P^{(k)}\varphi\|_{L^1(I^{(k)})/\Gamma^{(k)}_{cs}}
\leq\|P_0^{(k)}\varphi\|_{L^1(I^{(k)})}+|I^{(k)}|\|\Delta
P^{(k)}\varphi\| \leq (8+K)|I^{(k)}| \lv (\varphi),\]
which proves (\ref{szak}) and completes the proof.
\end{proof}

%Let
%$T_{(\widetilde{\pi},\widetilde{\lambda})}=\mathcal{R}(T_{(\pi,\lambda)})$
%and denote by $\widetilde{\varphi}:\widetilde{I}\to\R$ the renormalized
%cocycle, this is \[\widetilde{\varphi}(x)=\sum_{0\leq
%i<\Theta_{\beta}(\pi,\lambda)}\varphi(T_{(\pi,\lambda)}^ix)\text{
%for }x\in \widetilde{I}_\beta.\]

%Let $T=T_{(\pi,\lambda)}$ be an IET of periodic type and let $A$ be its period matrix.
Assume additionally that $T$ is of hyperbolic periodic type, i.e.\
$\theta_g>0$. By Lemma~\ref{invario}, there exists a bijection
$\xi:\Sigma(\pi)\to\Sigma(\pi)$ such that
$A^{-1}b(\mathcal{O})=b(\xi\mathcal{O})$ for
$\mathcal{O}\in\Sigma(\pi)$.
%Since $\xi^N=Id_{\Sigma(\pi)}$ for some $N\geq 1$, multiplying
%the period of $T$ by $N$, we can assume that $\xi=Id_{\Sigma(\pi)}$.
Moreover, by Remark~\ref{assumexiid}, we can assume that
$Ab(\mathcal{O})=b(\mathcal{O})$ for each
$\mathcal{O}\in\Sigma(\pi)$,  and hence
$A|_{\ker\Omega_{\pi}}=Id$. It follows that the Jordan canonical
form of $A^t$ has $\kappa-1$ simple eigenvalues $1$ as $A$, hence
the dimension of $\Gamma^{(0)}_c=\{h\in\R^{\mathcal{A}}:A^th=h\}$
is greater or equal than $\kappa-1$.  Since $\theta_g>0$ and
$2g+\kappa-1=d$, it follows that
$\dim\Gamma^{(0)}_s=\dim\Gamma^{(0)}_u=g$,
$\dim\Gamma^{(0)}_c=\kappa-1$ and
\[\R^{\mathcal{A}}=\Gamma^{(0)}=\Gamma^{(0)}_s\oplus\Gamma^{(0)}_c\oplus\Gamma^{(0)}_u\]
is an $A^t$--invariant decompositions. As
$\Gamma^{(0)}_s\oplus\Gamma^{(0)}_c=\Gamma^{(0)}_{cs}\subset\Gamma^{(0)}_0$,
we also have
\[\Gamma^{(0)}_0=\Gamma^{(0)}_s\oplus\Gamma^{(0)}_c\oplus(\Gamma^{(0)}_u\cap\Gamma^{(0)}_0).\]
Recall that $\Gamma^{(0)}_s\oplus\Gamma^{(0)}_u\subset H_{\pi}$.
Thus, when $T$ is of hyperbolic periodic type these subspace have
the same dimension, so they are equal. It follows that
\begin{equation}\label{trzydeco}
\Gamma^{(k)}=\Gamma^{(k)}_s\oplus\Gamma^{(k)}_c\oplus\Gamma^{(k)}_u,\;\;
H_{\pi}=\Gamma^{(k)}_s\oplus\Gamma^{(k)}_u,
\;\;\Gamma^{(k)}_0=\Gamma^{(k)}_s\oplus\Gamma^{(k)}_c\oplus(\Gamma^{(k)}_u\cap\Gamma^{(k)}_0)
\end{equation}
for $k\geq 0$ is a family of decomposition invariant with respect
to the renormalization operators $S(k,k')$ for $0\leq k<k'$.

\begin{comment}
As $\xi=Id_{\Sigma(\pi)}$, by Lemma~\ref{invariophi}, for every
$\varphi\in\ls(\sqcup_{\alpha\in \mathcal{A}} I^{(k)}_{\alpha})$
and $l\geq k$ we have\marginpar{change both references}
\begin{equation}\label{niezmo}
S(k,k')\varphi\in\ls(\sqcup_{\alpha\in \mathcal{A}}
I^{(k)}_{\alpha})\text{ and }
\mathcal{O}(S(k,k')\varphi)=\mathcal{O}(\varphi)\text{ for each
}\mathcal{O}\in\Sigma(\pi).
\end{equation}
\end{comment}

\begin{proposition}\label{thmcorrecgener}
Assume that  $T$ is of periodic type. There exist $C_1,C_2>0$ such
that for every $\varphi\in\overline{\ls}_0(\sqcup_{\alpha\in
\mathcal{A}} I^{(0)}_{\alpha})$ if
$\widehat{\varphi}+\Gamma^{(0)}_{cs}= P^{(0)}\varphi$ then
$\widehat{\varphi}-\varphi\in\Gamma_0^{(0)}$ and for any  $k\geq
1$ we have
\[\frac{1}{|I^{(k)}|}\|S(k)(\widehat{\varphi})\|_{L^1(I^{(k)})}\leq
C_1\lv(\varphi)k^{M} +C_2\|\widehat{\varphi}\|_{L^1(I^{(0)})}/|I^{(0)}|k^{M-1}
%\left( C_1 \lv(\varphi)+C_2\frac{1}{|I^{(0)}|}\|\widehat{\varphi}\|_{L^1(I^{(0)})}\right) k^{M-1}
.\]
If additionally  $T$  is of hyperbolic periodic
type and $\varphi\in\ls_0(\sqcup_{\alpha\in \mathcal{A}}
I^{(0)}_{\alpha})$ then for any $k\geq 0$
\[\frac{1}{|I^{(k)}|}\|S(k)(\widehat{\varphi})\|_{L^1(I^{(k)})}\leq
C_1 \lv
(\varphi)+C_2\frac{1}{|I^{(0)}|}\|\widehat{\varphi}\|_{L^1(I^{(0)})}.\]
\end{proposition}

\begin{proof}
\emph{Non-hyperbolic case.} Let us first show that
$\widehat{\varphi}-\varphi\in \Gamma_0^{(0)}$. Since
$U^{(0)}\widehat{\varphi}=
\widehat{\varphi}+\Gamma^{(0)}_{cs}=P^{(0)}\varphi$,
\[U^{(0)}\widehat{\varphi}= U^{(0)}\circ P_0^{(0)}\varphi-\Delta P^{(0)}\varphi=
U^{(0)}\varphi-U^{(0)}\circ C^{(0)}\varphi-\Delta P^{(0)}\varphi,\]
we have $\varphi-\widehat{\varphi}\in U^{(0)}\circ C^{(0)}\varphi+\Delta P^{(0)}\varphi \subset \Gamma_0^{(0)}$.
In view of (\ref{splatanies}) and (\ref{przem}),
\[U^{(k)}\circ
S(k)\widehat{\varphi}=S_u(k)\circ U^{(0)}\widehat{\varphi}
=S_u(k)\circ P^{(0)}\varphi=P^{(k)}\circ S(k)\varphi.\] Therefore,
from (\ref{szak}) and (\ref{nave}), we have
\begin{align*}
\|U^{(k)}\circ
S(k)\widehat{\varphi}\|_{L^1(I^{(k)})/\Gamma^{(k)}_{cs}}
&=\|P^{(k)}(S(k)\varphi)\|_{L^1(I^{(k)})/\Gamma^{(k)}_{cs}}\leq(8+K)C
|I^{(k)}| \lv (\varphi).
\end{align*}
%Moreover, by (\ref{nave}), $\|U^{(k)}\circ
%S(k)\widehat{\varphi}\|_{L^1(I^{(k)})/\Gamma^{(k)}_{cs}}\leq(8+K)C
%|I^{(k)}| \lv (\varphi)$. If additionally
%$\varphi\in\ls_0(\sqcup_{\alpha\in \mathcal{A}} I^{(0)}_{\alpha})$
%then, by (\ref{nave}),
%\[\|U^{(k)}\circ
%S(k)\widehat{\varphi}\|_{L^1(I^{(k)})/\Gamma^{(k)}_{cs}}\leq(8+K)C
%|I^{(k)}| \lv (\varphi).\]
It follows from the definition of $\| \cdot
\|_{L^1(I^{(k)})/\Gamma^{(k)}_{cs}}$ on the quotient space that
for every $k\geq 0$ there exists
$\varphi_k\in\overline{\ls}_0(\sqcup_{\alpha\in \mathcal{A}}
I^{(k)}_{\alpha})$ and $h_k\in\Gamma^{(k)}_{cs}$ such that
\begin{equation}\label{szcsk1}
S(k)\widehat{\varphi}=\varphi_k+h_k\text{ and
}\|\varphi_k\|_{L^1(I^{(k)})}\leq(8+K)C |I^{(k)}|\lv (\varphi).
\end{equation}
%Moreover, if  $\varphi\in\ls_0(\sqcup_{\alpha\in \mathcal{A}}
%I^{(0)}_{\alpha})$ then $\varphi_k\in\ls_0(\sqcup_{\alpha\in
%\mathcal{A}} I^{(k)}_{\alpha})$ can be chosen so that
%\begin{equation}\label{szcsk3}
%\|\varphi_k\|_{L^1(I^{(k)})}\leq(8+K)C |I^{(k)}|\lv (\varphi).
%\end{equation}
Next note that
\begin{equation}\label{roznren}
\varphi_{k+1}+h_{k+1}=S(k+1)\widehat{\varphi}=S(k,k+1)S(k)\widehat{\varphi}=S(k,k+1)\varphi_k+A^th_k,
\end{equation}
so setting $\Delta h_{k+1}=h_{k+1}-A^th_k$ ($\Delta h_0=h_0$) we
have $\Delta h_{k+1}=-\varphi_{k+1}+S(k,k+1)\varphi_k$. Moreover,
by (\ref{nase}) and (\ref{szcsk1}), for $k\geq 1$,
\begin{align*}
\|&\Delta
 h_{k}   \|_{L^1(I^{(k)})}=\|\varphi_{k} + S(k-1,k)\varphi_{k-1}\|_{L^1(I^{(k)})}\\
&\leq
\|\varphi_{k}\|_{L^1(I^{(k)})}+\|S(k-1,k)\varphi_{k-1}\|_{L^1(I^{(k)})}
\leq\|\varphi_{k}\|_{L^1(I^{(k)})}+\|\varphi_{k-1}\|_{L^1(I^{(k-1)})}\\
&\leq
\left(1+\frac{|I^{(k-1)}|}{|I^{(k)}|}\right)(8+K)C|I^{(k)}|\lv
(\varphi) \leq (1+\|A\|)(8+K)C|I^{(k)}|\lv (\varphi).
\end{align*}
It follows from (\ref{compnorm}) that $ \|\Delta h_{k}\|\leq d\nu(A)(1+\|A\|)(8+K)C\lv
(\varphi)$ for $k\geq 1$ and
\begin{align*}
\|\Delta h_{0}\|&\leq \frac{d\nu(A)}{|I^{(0)}|}\|
h_{0}\|_{L^1(I^{(0)})}
=\frac{d\nu(A)}{|I^{(0)}|}\|\widehat{\varphi}-\varphi_0\|_{L^1(I^{(0)})}\\
&\leq
d\nu(A)\left(\|\widehat{\varphi}\|_{L^1(I^{(0)})}/{|I^{(0)}|}
+(8+K)C\lv(\varphi)\right).
\end{align*}
Since $h_k=\sum_{0\leq l\leq k}(A^t)^l\Delta h_{k-l}$ and $\Delta
h_l\in\Gamma^{(k')}_{cs}$, by (\ref{censtab}),
\begin{eqnarray*}
\|h_k\|&\leq&\sum_{0\leq l\leq k}\|(A^t)^l\Delta h_{k-l}\|\leq\sum_{0\leq l\leq k}Cl^{M-1}\|\Delta h_{k-l}\|\\
&\leq& C'_1\lv(\varphi)k^{M}
+C_2\|\widehat{\varphi}\|_{L^1(I^{(0)})}/|I^{(0)}|k^{M-1}
%\left(C'_1\lv(\varphi)+C_2\|\widehat{\varphi}\|_{L^1(I^{(0)})}/|I^{(0)}|\right)k^{M-1}.
\end{eqnarray*}
for some $C'_1, C_2>0$.  Setting $C_1:=C'_1+(8+k)C$, in view of
(\ref{szcsk1}), it follows that for $k\geq 1$,
\begin{align*}\|S(k)\widehat{\varphi}\|_{L^1(I^{(k)})}&
\leq\|{\varphi}_k\|_{L^1(I^{(k)})}+|I^{(k)}|\|h_k\|\\
&\leq|I^{(k)}|\left(
C_1\lv(\varphi)k^{M} +C_2\|\widehat{\varphi}\|_{L^1(I^{(0)})}/|I^{(0)}|k^{M-1}\right).
\end{align*}

\subsubsection*{Hyperbolic case} Let us now prove the second part,
assuming that $T$ is of hyperbolic periodic type and
$\varphi\in\ls_0(\sqcup_{\alpha\in \mathcal{A}}
I^{(0)}_{\alpha})$. Then, as shown just before
Proposition~\ref{thmcorrecgener},
$\Gamma^{(k)}_{cs}=\Gamma^{(k)}_{c}\oplus\Gamma^{(k)}_{s}$ and
$H_\pi=\Gamma^{(k)}_s\oplus\Gamma^{(k)}_u$ are invariant direct
sum decompositions. Let $h_k=h^s_k+h^c_k$, where
$h^c_k\in\Gamma^{(k)}_{c}$ and $h^s_k\in \Gamma^{(k)}_{s}\subset
H_\pi$. By Remark~\ref{remiso}, $\Lambda^\pi(h^s_k)=0$. In view of
Lemma~\ref{invariophi}, (\ref{szcsk1}) and Remark
\ref{Ohzeroreformulation}, it follows that $\mathcal{O}(h^s_k)=0$
and
\[\mathcal{O}(\widehat{\varphi})=\mathcal{O}(S(k)\widehat{\varphi})=
\mathcal{O}(\varphi_k)+\mathcal{O}(h^c_k)
\text{ for every }\mathcal{O}\in\Sigma(\pi).\]
Suppose that
\begin{equation*}
\varphi(x)=-\sum_{\alpha\in\mathcal{A}}(C^+_\alpha\log(|I|\{(x-l_\alpha)/|I|\})+
C^-_\alpha\log(|I|\{(r_\alpha-x)/|I|\}))+g(x),
\end{equation*}
where $g\in\bv_*^1(\sqcup_{\alpha\in \mathcal{A}} I_{\alpha})$.
Then $\widehat{\varphi} = \varphi + h$ for some
$h\in\Gamma^{(0)}_0$. Thus $\bl(\widehat{\varphi})= \bl(\varphi)$
and since $\var(g+h) = \var(g)$ we have $\lv(\widehat{\varphi})=
\lv (\varphi)$.
%\begin{equation*}
%\widehat{\varphi}(x)=-\sum_{\alpha\in\mathcal{A}}(C^+_\alpha\log(|I|\{(x-l_\alpha)/|I|\})
%+C^-_\alpha\log(|I|\{(r_\alpha-x)/|I|\}))+g(x)+h
%\end{equation*}
 Thus, by Proposition~\ref{lemlogreno1},
$\lv(S(k)\widehat{\varphi})\leq
C\lv(\widehat{\varphi})=C\lv(\varphi)$. Similarly, since
$\varphi_k=S(k)\widehat{\varphi}-h_k$, it follows that
$\lv(\varphi_k)=\lv(S(k)\widehat{\varphi})\leq C\lv(\varphi)$.
Thus, by Lemma~\ref{lemoszo}, for every
$\mathcal{O}\in\Sigma(\pi)$ we can estimate
$\mathcal{O}(\varphi_k)$ and $\mathcal{O}(\widehat{\varphi})$
respectively by
\begin{align*}
|\mathcal{O}(\varphi_k)| &\leq
2d\nu(A)\frac{1}{|I^{(k)}|}\int_{I^{(k)}}|\varphi_k(x)|\,dx+2d\lv(\varphi_k)
\\ & \leq
2d\nu(A)\frac{1}{|I^{(k)}|}\|\varphi_k\|_{L^1(I^{(k)})}+2dC\lv(\varphi),\\
|\mathcal{O}(\widehat{\varphi})| & \leq
2d\nu(A)\frac{1}{|I^{(0)}|}\|\widehat{\varphi}\|_{L^1(I^{(0)})}+2d\lv(\varphi).
\end{align*}
Hence, by (\ref{szcsk1}), $|\mathcal{O}(\varphi_k)|\leq
2dC(\nu(A)(8+K)+1)\lv (\varphi)$. It follows that there exist
$K_1,K_2>0$ such that, for every $\mathcal{O}\in\Sigma(\pi)$,
\begin{align*}
|\mathcal{O}(h^c_k)|&\leq|\mathcal{O}(\varphi_k)|+
|\mathcal{O}(\widehat{\varphi})|= K_1\lv
(\varphi)+K_2\frac{1}{|I^{(0)}|}\|\widehat{\varphi}\|_{L^1(I^{(0)})},
\end{align*}
so, by Remark \ref{Ohzeroreformulation},
\[\|\Lambda^\pi(h^c_k)\|= \max_{\mathcal{O}\in \Sigma(\pi)}
|\mathcal{O}(h^c_k)| \leq K_1\lv
(\varphi)+K_2\|\widehat{\varphi}\|_{L^1(I^{(0)})}/{|I^{(0)}|}.\]
Since, by Remark~\ref{remiso}, $\Lambda^\pi:\Gamma^{(k)}_{c}\to
\R^{\Sigma_0(\pi)}$ is an isomorphism of linear spaces, there
exists $K'\geq 1$ such that $\|h\|\leq K'\|\Lambda^\pi h\|$ for
every $h\in\Gamma^{(k)}_{c}$. It follows that
\begin{equation}\label{estcentral}
\|h^c_k\|\leq K'\left(K_1 \lv
(\varphi)+K_2\frac{1}{|I^{(0)}|}\|\widehat{\varphi}\|_{L^1(I^{(0)})}\right).
\end{equation}
Let $\Delta h^s_{k+1}=h^s_{k+1}-A^th^s_k$  for $k\geq 0$ and
$\Delta h^s_0=h^s_0$. Then from (\ref{roznren}), we have \[\Delta
h^s_{k+1}=-\varphi_{k+1}+S(k,k+1)\varphi_k-h^c_{k+1}+A^th^c_k=-\varphi_{k+1}+S(k,k+1)\varphi_k-h^c_{k+1}+h^c_k.\]
Therefore, by (\ref{nase}), (\ref{compnorm}), (\ref{szcsk1}) and
(\ref{estcentral}), for all $k\geq 1$,
\begin{align*}
\|&\Delta
h^s_{k}\|_{L^1(I^{(k)})}\leq\|\varphi_{k}+ h^c_{k} \|_{L^1(I^{(k)})}+
\|S(k-1,k)(\varphi_{k-1}+h^c_{k-1}\|_{L^1(I^{(k)})}
%+\|\|_{L^1(I^{(k)})}+\|h^c_{k-1}\|_{L^1(I^{(k)})}
\\ & \leq\|\varphi_{k}\|_{L^1(I^{(k)})}++\|h^c_{k}\|_{L^1(I^{(k)})} +
\|\varphi_{k-1}\|_{L^1(I^{(k-1)})}+\|h^c_{k-1}\|_{L^1(I^{(k-1)})}\\
&\leq|I^{(k)}|\left(1+\|A\|\right)\left((8+K)C+K'K_1)\lv
(\varphi)+K'K_2\|\widehat{\varphi}\|_{L^1(I^{(0)})}/|I^{(0)}|\right).
\end{align*}
It follows from (\ref{compnorm}) that there exist constants
$K_1',K_2'>0$ such that for $k\geq 1$
\[
\|\Delta
h^s_{k}\|\leq K'_1\lv
(\varphi)+K'_2\|\widehat{\varphi}\|_{L^1(I^{(0)})}/|I^{(0)}|,
\]
while for $k=0 $ we have
\begin{align*}\|\Delta
h^s_{0}\|&=\|
h^s_{0}\| =\|\widehat{\varphi}-\varphi_0-h^c_0\|\leq
\|h^c_0\|+\frac{d\nu(A)}{|I^{(0)}|}(\|\widehat{\varphi}\|_{L^1(I^{(0)})}+\|{\varphi}_0\|_{L^1(I^{(0)})})\\
&\leq K'_1\lv
(\varphi)+K'_2\|\widehat{\varphi}\|_{L^1(I^{(0)})}/|I^{(0)}|.
\end{align*}
Since $h^s_k=\sum_{0\leq l\leq k}(A^t)^l\Delta h^s_{k-l}$ and
$\Delta h^s_l\in\Gamma^{(l)}_{s}$,  it follows from (\ref{stab}) that
\begin{equation}\label{exp}
\begin{split}
\|h^s_k\| & \leq \sum_{0\leq l\leq k}\|(A^t)^l\Delta h^s_{k-l}\|\leq\sum_{0\leq l\leq k}C\exp(-l\theta_-)\|\Delta h^s_{k-l}\|\\
& \leq
%C\sum_{l\in \N} \exp(-l\theta_-) \left(K'_1\lv (\varphi)+K'_2\|\widehat{\varphi}\|_{L^1(I^{(0)})}/|I^{(0)}|\right) =
\frac{K'_1\lv
(\varphi)+K'_2\|\widehat{\varphi}\|_{L^1(I^{(0)})}/|I^{(0)}}{1-\exp(-\theta_-)} .
\end{split}
\end{equation}
Combining  (\ref{szcsk1}), (\ref{estcentral}) and (\ref{exp}), we
find that for some $C_1, C_2>0$
\begin{align*}
\frac{1}{|I^{(k)}|}\|S(k)\widehat{\varphi}\|_{L^1(I^{(k)})}
&\leq\frac{1}{|I^{(k)}|}\|{\varphi}_k\|_{L^1(I^{(k)})}+\|h^c_k\|+\|h^s_k\|\\
&\leq C_1\lv
(\varphi)+C_2\|\widehat{\varphi}\|_{L^1(I^{(0)})}/|I^{(0)}|.
\end{align*}
\end{proof}
%\subsection{Correction operator}
\begin{proofof}{Theorem}{operatorcorrection}
Let us first show that for every
$\varphi\in\overline{\ls}_0(\sqcup_{\alpha\in \mathcal{A}}
I_{\alpha})$ there exists  a unique $h\in
\Gamma^{(0)}_u\cap\Gamma^{(0)}_0$ such that $\varphi-h\in
P^{(0)}\varphi$, where $P^{(0)}$ is the operator in  Definition
\ref{wzorp}. Since $\widehat{\varphi}-\varphi\in
\Gamma^{(0)}_0=(\Gamma^{(0)}_u\cap\Gamma^{(0)}_0)\oplus\Gamma^{(0)}_{cs}$,
there exist $h\in(\Gamma^{(0)}_u\cap\Gamma^{(0)}_0)$ and
$h'\in\Gamma^{(0)}_{cs}$ such that
$\varphi-h=\widehat{\varphi}+h'$. As $\widehat{\varphi}\in
P^{(0)}\varphi$, it follows that
\[\varphi-h\in\widehat{\varphi}+\Gamma^{(0)}_{cs}=P^{(0)}\varphi.\]
Suppose that $h_1,h_2\in\Gamma^{(0)}_u\cap\Gamma^{(0)}_0$ are vectors such
that
\[\varphi-h_1+\Gamma^{(0)}_{cs}=\varphi-h_2+\Gamma^{(0)}_{cs}=
P^{(0)}\varphi.\] Then
$\|S(k)(\varphi-h_1)\|_{L^1(I^{(k)})}/|I^{(k)}|$ and
$\|S(k)(\varphi-h_2)\|_{L^1(I^{(k)})}/|I^{(k)}|$ grow polynomially
in $k$ by the first part of Proposition~\ref{thmcorrecgener}.
Thus,
$\|(A^t)^k(h_1-h_2)\|\leq\|S(k)(h_1-h_2)\|_{L^1(I^{(k)})}/|I^{(k)}|$
grows polynomially as well, so $h_1-h_2\in\Gamma^{(0)}_{cs}$.
Since $h_1-h_2\in\Gamma^{(0)}_u$ and $\Gamma^{(0)}_{cs} \cap
\Gamma^{(0)}_u = \{ {0}\}$, it follows that $h_1=h_2$. Thus, there
exists a unique linear operator
$\mathfrak{h}:\overline{\ls}_0(\sqcup_{\alpha\in \mathcal{A}}
I_{\alpha})\to \Gamma^{(0)}_{u}\cap\Gamma^{(0)}_{0}$, called the
{\em correction operator}, such that
\[\varphi-\mathfrak{h}(\varphi)+\Gamma^{(0)}_{cs}= P^{(0)}(\varphi).\]
Note that, by Remark~\ref{Pzero}, ${P}^{(0)}(h)=0$ for each
$h\in\Gamma^{(0)}_0$, so
\begin{equation}\label{opernagama}
{\mathfrak{h}}(h)=h\;\text{ if }\;h\in\Gamma^{(0)}_u\cap
\Gamma^{(0)}_0\;\;\;\text{ and }\;\;\;{\mathfrak{h}}(h)=0\;\text{ if
}\;h\in\Gamma^{(0)}_{cs}.\end{equation}
In particular, the image of $\mathfrak{h}$ is $\Gamma^{(0)}_u\cap
\Gamma^{(0)}_0$ which has dimension $g-1$.

 In view of (\ref{szak}) the
operator $P^{(0)}:\overline{\ls}_0(\sqcup_{\alpha\in \mathcal{A}}
I_{\alpha})\to  \overline{\ls}_0(\sqcup_{\alpha\in \mathcal{A}}
I_{\alpha})/\Gamma^{(0)}_{cs}$ is bounded with respect  to the
norm $\| \cdot \|_{L_1(I)/\Gamma^{(0)}_{cs}}$. Therefore, by the
closed graph theorem, the operator ${\mathfrak{h}}$ is also
bounded. Indeed, if $\varphi_n\to \varphi$ in ${\overline{\ls}}_0$
and ${\mathfrak{h}}(\varphi_n)\to h$ in
$\Gamma^{(0)}_u\cap\Gamma^{(0)}_0$ then have both
\begin{align*}
& {P}^{(0)}\varphi_n\to
{P}^{(0)}\varphi=\varphi-{\mathfrak{h}}(\varphi)+\Gamma^{(0)}_{cs},\\
& {P}^{(0)}\varphi_n = \varphi_n-{\mathfrak{h}}(\varphi_n)+\Gamma^{(0)}_{cs} \to \varphi-h+\Gamma^{(0)}_{cs},
\end{align*}
so from one hand ${\mathfrak{h}}(\varphi)-h \in
\Gamma^{(0)}_u\cap\Gamma^{(0)}_0$ and at the same time
${\mathfrak{h}}(\varphi)-h\in\Gamma^{(0)}_{cs}$, so
$h={\mathfrak{h}}(\varphi)$. Since the  vector norm and the
$L^1$-norm are equivalent on $\Gamma^{(0)}$ by (\ref{compnorm}),
we get that the operator is bounded.
%, i.e.~there exists $C>0$ such that
%\[\|{\mathfrak{h}}(\varphi)\|\leq C(\bl(\varphi)+
%\var(g_\varphi)+\|g_\varphi\|_{\sup})\text{ for each }
%\varphi\in{\ls}_0(\sqcup_{\alpha\in \mathcal{A}}
%I_{\alpha}).\]
Suppose now that $\mathfrak{h}(\varphi)=0$. Then
\[
\varphi=\varphi-\mathfrak{h}(\varphi)\in\varphi-
\mathfrak{h}(\varphi)+\Gamma^{(0)}_{cs}=P^{(0)}(\varphi).
\]
Now parts (1) and (2) of the Theorem follows directly from
Proposition~\ref{thmcorrecgener}. This concludes the proof.
\end{proofof}

The following Lemma will be used several times in \S\ref{final:sec}.
\begin{lemma}\label{theorem:invop}
%If $\varphi\in\ls_0(\sqcup_{\alpha\in \mathcal{A}} I_{\alpha})$ is
%a coboundary then $\mathfrak{h}(\varphi)=0$. Moreover,
If $\varphi\in\bv_0(\sqcup_{\alpha\in \mathcal{A}} I_{\alpha})$ is
a measurable coboundary then ${\mathfrak{h}}(\varphi)=0$.
\end{lemma}

\begin{proof}
\begin{comment}
Suppose that $\varphi\in\ls_0(\sqcup_{\alpha\in \mathcal{A}}
I_{\alpha})$ is a coboundary. In view of the previous lemma,
$\varphi=\xi-\xi\circ T$ for a bounded measurable function
$\xi:I\to\R$. Then, by the definition of the operator $S(k)$,
$S(k)\varphi=\xi-\xi\circ T^{k}$ for each $k\geq 0$. Hence
\[\frac{1}{|I^{(k)}|}\|S(k)\varphi\|_{L^1(I^{(k)})}\leq 2\|\xi\|_{\sup}.\]
By Proposition~\ref{thmcorrecgener}, there exists $C>0$ such that
\[\frac{1}{|I^{(k)}|}\|S(k)\left(\varphi-\mathfrak{h}(\varphi)\right)\|_{L^1(I^{(k)})}\leq C.\]
Therefore, the sequence
$\left(\|S(k)(\mathfrak{h}(\varphi))\|_{L^1(I^{(k)})}/|I^{(k)}|\right)_{k\geq
0}$ is bounded. In view of (\ref{compnorm}), it follows that
$\left(\|(A^t)(\mathfrak{h}(\varphi))\|\right)_{k\geq 0}$ is
bounded, so
$\mathfrak{h}(\varphi)\in\Gamma^{(0)}_{cs}\cap\Gamma^{(0)}_{u}=\{0\}$ which
gives $\mathfrak{h}(\varphi)=0$.
\end{comment}
Suppose that $\varphi\in\bv_0(\sqcup_{\alpha\in \mathcal{A}}
I_{\alpha})$ and  $\varphi=\xi-\xi\circ T$ for a measurable
function $\xi:I\to\R$. Set $h:={\mathfrak{h}}(\varphi)$. Since
$\varphi-h\in P^{(0)}\varphi$ and the operator $P^{(0)}$ is an
extension of the operator $P^{(0)}$ defined in \cite{Co-Fr}, by
Theorem~C.6 in \cite{Co-Fr}, there exists constants $C, M>0$ such
that $\|\varphi^{(n)}-h^{(n)}\|_{\sup}\leq C\log^Mn$.  Moreover,
as shown in Lemma 4.1 in
 \cite{Co-Fr}, there exists $\delta>0$ such that
for each $\alpha\in\mathcal{A}$ and $k>0$ there exists a
measurable set $C^{(k)}_\alpha\subset I$ such that
$Leb(C^{(k)}_\alpha)\geq \delta>0$ and
$h^{(Q_\alpha(k))}(x)=((A^t)^kh)_\alpha$ for all $x\in
C^{(k)}_\alpha$. Since $\varphi$ is a coboundary, by Lusin's
theorem, there exist $K>0$ and a sequence $(B_k)_{k\geq 0}$ of
measurable sets with $Leb(B_k)>1-\delta$ such that
$|\varphi^{(k)}(x)|\leq K$ for all $x\in B_k$ and $k \geq 0$. Then
taking  $x\in C^{(k)}_\alpha\cap B_{Q_\alpha(k)}\neq\emptyset$, for all $\alpha\in\mathcal{A}$ we
get
\[|((A^t)^kh)_\alpha|= |h^{(Q_\alpha(k))}(x)|\leq
|\varphi^{(Q_\alpha(k))}(x)|+C\log^MQ_\alpha(k)\leq K+Ck^M\log^M\|A\|
.\] Therefore
$\|(A^t)^kh\|\leq K+Ck^M\log^M\|A\|$ for $k\geq 1$, so
$h\in\Gamma^{(0)}_{cs}\cap\Gamma^{(0)}_{u}=\{0\}$.
\end{proof}

\section{Ergodicity}\label{ergodicity:sec}
In this section  we prove ergodicity for the corrected cocycle
over IETs (Theorem~\ref{mainIETs}).  Let $\mathfrak{h}$ be the
correction operator defined in Section~\ref{correction:sec}.
\begin{theorem}\label{theorem:ergmain}
Let $T:I\to I$ be an IET of hyperbolic periodic type and
$\varphi\in\ls_0(\sqcup_{\alpha\in\mathcal{A}}I_{\alpha})$ such
that $\mathfrak{h}(\varphi)=0$. If $\bl(\varphi)\neq 0$
(i.e.~not all constants $C_\alpha^\pm$ are zero)
%$\bl(\varphi)\neq 0$,
then the
skew product $T_{{\varphi}}$ is
ergodic. %. Then for every $\widehat{\varphi}\in
%P^{(0)}\varphi$ the skew product $T_{\widehat{\varphi}}$ is
%ergodic. In particular, if $\bl(\varphi)\neq 0$ and
%$\mathfrak{h}(\varphi)=0$
\end{theorem}
The proof is given at the end of \S\ref{tightness:sec}. Theorem
\ref{theorem:ergmain} implies Theorem~\ref{mainIETs}:
\begin{proofof}{Theorem}{mainIETs}
Given $\varphi\in\ls_0(\sqcup_{\alpha\in\mathcal{A}}I_{\alpha})$
such that $\bl(\varphi)\neq 0$, let $\chi =
\mathfrak{h}(\varphi)$. By Theorem~\ref{operatorcorrection},
$\chi$ is constant on each $I_{\alpha}$, belongs to a $g-1$
dimensional subspace of $\Gamma^{(0)}$ and since
$\mathfrak{h}(\varphi-\chi)=0$, the skew product $T_{\varphi -
\chi}$ is ergodic by Theorem~\ref{theorem:ergmain}.
\end{proofof}

For the rest of this section, assume that $T:I\to I$ is an IET is
of hyperbolic periodic type, $|I|=1$  and $\varphi$ is a cocycle
in $\ls_0(\sqcup_{\alpha\in \mathcal{A}} I_{\alpha})$ such that
%$\mathfrak{h}(\varphi)=0$ and
$\bl(\varphi)\neq 0$. To prove Theorem~\ref{theorem:ergmain}, we
will use the ergodicity criterion given by
Proposition~\ref{ergodicity:criterium} in Section
\ref{essentialvalues:sec}. In \S\ref{oscillationsec} we will
construct the rigidity sets for
Proposition~\ref{ergodicity:criterium} and prove some preliminary
Lemmas, while in \S\ref{tightness:sec} we will verify that they
satisfy the assumptions of Proposition~\ref{ergodicity:criterium}.

\subsection{Rigidity sets with large oscillations of Birkhoff sums}\label{oscillationsec}
%\subsection{Rigidity sets}\label{rigiditysec}
Katok proved in \cite{Ka} that for \emph{any} interval exchange
transformation there exists  a sequence of Borel sets $(\Xi_n)$
and an increasing sequence of numbers $(q_n)$ and $\delta>0$ such
that
\begin{equation}\label{rigidity}
Leb(\Xi_n)\geq \delta, \;\;Leb(\Xi_n\triangle T^{-1}\Xi_n)\to
0\;\;\mbox{ and } \sup_{x\in \Xi_n}d(x,T^{q_n}x)\to 0.
\end{equation}
We call sequences $(\Xi_n)$ and $(q_n)$ with the above property
\emph{rigidity sets} and \emph{rigidity times} respectively.  We
present here below a particular variation on the construction of
Katok, using Rauzy-Veech induction (Definition \ref{beta}),  which
allows us to obtain further properties (in particular
Lemma~\ref{arithmetic progressions}) needed in the following
sections.\footnote{A different variant of Katok's construction was
also used by the second author in \cite{Ul:wea, Ul:abs}. We remark
that the second property in (\ref{rigidity}) is not always
required in the definition of rigidity sets (for example, it is
not assumed in  \cite{Ul:wea, Scheg, Ul:abs}), but it is important
for us for the proof of ergodicity.}

\begin{notation}\label{alpha}
Let $\overline{\alpha}\in\mathcal{A}$ be such that
$\pi_0(\overline{\alpha})=1$, i.e.~$I_{\overline{\alpha}}$ is the
first of the intervals exchanged by $T$. Notice that  for each
$n\geq 0$ we have $\pi_0^{(n)}(\overline{\alpha}) = 1$.
\end{notation}

%Since $\bl(\varphi)\neq 0$, not all constants $C_\alpha^\pm $ are
%zero. We will present the construction of the rigidity sets
%assuming that there exists ${\beta_0} \in\mathcal{A}$ with
%$C^-_{\beta_0}\neq 0$. The case where $C_\alpha^-=0$ for all
%$\alpha \in \mathcal{A}$ requires a very similar symmetric
%construction. We comment on the differences at the end of this
%\S\ref{oscillationsec}.

%The following property will be crucial to have large
%Birkhoff sums of derivatives (see Lemma~\ref{largederivatives}).
% and hence in the proof of ergodicity in \S\ref{ergodicity:sec}.
%Let $T$ be an IET of periodic type with period $p$. Consider the
%interval exchange transformation obtained inducing $T$ on  $
%I^{(pn)}_{\overline{\alpha}}$. In general, the induced map of an
%IET on a subinterval is an IET of at most $d+2$ subintervals. In
%this case, since  the interval $I^{(pn)}_{\overline{\alpha}}$ is
%\emph{admissible} in the sense of Veech (see \cite{Ve2} \S8) the
%induced map on is an IET of $d$ intervals and moreover it is an
%interval exchange which belong to the Rauzy-Veech orbit of $T$,
%i.e.  $I^{(pn)}_{\overline{\alpha}} = I^{(m)}$ for some
%$pn<m<p(n+1)$ and the induced map is exactly by $T^{(m)}$ (see
%Veech \cite{Ve2}, \S8). This in particular implies that the
%subintervals $ I^{(m)}_{ \beta} $, $\beta\in\mathcal{A}$ are
%balanced in the sense of (\ref{balancedintervals}) and in
%particular

\begin{lemma}\label{remark:wybj}
For every
$\varphi\in\overline{\logs}(\sqcup_{\alpha\in\mathcal{A}}I_{\alpha})$
with $\bl(\varphi)\neq 0$ there exists $\beta_0\in A$ such that
for every integer $n\geq 2$ there exists  $\beta_n\in\mathcal{A}$
and $j_n\in\N$ so that at least one of the following  two cases
holds:
%\begin{equation*}C^{-}_\beta_0\neq 0, \  r_{\beta_0}=\widehat{T}^{j_n} r^{(n)}_{\beta_n}
%\qquad (\text{Case}\ (R)) \quad\text{or} \quad  C^{+}_\beta_0\neq 0, \
% l_{\beta_0}={T}^{j_n} l^{(n)}_{\beta_n} \qquad (\text{Case}\  (L)),  \end{equation*}

\begin{itemize}
\item[-] {\it Case (R)}: $C^{-}_{\beta_0}\neq 0$ and
 $ r_{\beta_0}=\widehat{T}^{j_n} r^{(n)}_{\beta_n}$,
%, where $Q_{\overline{\alpha}}(n-2)\leq j_{n} <Q_{\beta_n}(n)$;
\item[-] Case (L): $C^{+}_{\beta_0}\neq 0$ and $ l_{\beta_0}={T}^{j_n} l^{(n)}_{\beta_n}$,
%, where $Q_{\overline{\alpha}}(n-2)\leq j_{n} <Q_{\beta_n}(n)$;
\end{itemize}
where  in both cases, one has
\begin{equation}\label{controlledheight}
Q_{\overline{\alpha}}(n-2)\leq j_{n} <Q_{\beta_n}(n).
\end{equation}
Moreover, in both cases the closures of the intervals  $T^{i}
I^{(n)}_{\beta_n}$ for $Q_{\beta_n}(n) \leq i \leq Q_{\beta_n}(n)
+ Q_{\overline{\alpha}}(n-2)$ do not contain any point of
$End(T)=\{ r_\alpha, l_\alpha, \alpha \in \mathcal{A}\}$.
\end{lemma}
\begin{proof}
Since $\bl(\varphi)\neq 0$, not all constants $C_\alpha^\pm $ are
zero. If there exists at least one $\beta$ such that $C_\beta^-
\neq 0$, pick as $\beta_0$ one of these $\beta$. In this case let
$\chi $ be the permutation given by Lemma~\ref{comparingconstants}
applied to $k=0$ and $k'=n$ and let $\beta_n:=\chi^{-1}(\beta_0)$.
Then by Lemma~\ref{comparingsingularities} there exists $0\leq j_n
<Q_{\beta_n}(n)$ such that
$(\widehat{T})^{j_n}r^{(n)}_{\beta_n}=r_{\beta_0}$, i.e.~we have
Case (R). Consider now the case in which $C_\alpha^-=0$ for all
$\alpha \in \mathcal{A}$. Since $\varphi$ has singularities of
geometric type, at least one among  $C_{\pi_0^{-1}(1)}^+$ and
$C_{\pi_1^{-1}(1)}^+$ is zero. Thus, since
$\varphi\in\overline{\logs}$ satisfy the symmetry condition
(\ref{zerosymweak}),  there must exists $\beta_0$ such that
$C_{\beta_0}^+\neq 0$ and $\beta_0 \notin \{ \pi_0^{-1}(1) ,
\pi_1^{-1}(1) \}$. In this case set $\beta_n=\beta_0$ for all $n$.
%Notice  that  $\beta_n \notin \{ \pi_n^{-1}(1) , \pi_n^{-1}(1) \}=\{ \pi_0^{-1}(1) , \pi_0^{-1}(1) \}$ for any $n$.
By Lemma~\ref{comparingsingularities} there exists $0\leq j_n <Q_{\beta_n}(n)$
such that $({T})^{j_n}l^{(n)}_{\beta_n}=l_{\beta_0}$, i.e.~we have Case (L).

%Remark that under Rauzy-Veech induction, since the first interval
%$I^{n}_{{{\pi^n}_0}^{-1}}$ is never changed,  $(\pi^n_0)^{-1}(1)=
%\pi_0^{-1}(1)$  and $(\pi^n_1)^{-1}(1)=\pi_1^{-1}(1)$ for any $n
%\in \N$. Thus, for any $n \in \N$,  $\beta_n \notin \{
%(\pi_0^{n})^{-1}(1) , (\pi_1^n)^{-1}(1) \}$.

Remark that $I^{(n-1)}\subset I^{(n-2)}_{\overline{\alpha}}$,
because, since $Z(n-2,n-1)=A$ is a positive matrix, each $x\in
I^{(n-1)}$ has to visit $I^{(n-2)}_{\overline{\alpha}}$ before its
first return time to $I^{(n-1)}$. Repeating the argument one more
time, we see that $I^{(n)}$ is strictly contained in
$I^{(n-2)}_{\overline{\alpha}}$ (since $I^{(n)}$ and
$I^{(n-2)}_{\overline{\alpha}}$ share $0$ as left endpoint, this
means that the right endpoint of $I^{(n)}$ is in the interior of
$I^{(n-2)}_{\overline{\alpha}}$). Remark that the interiors of the
intervals $T^j  I^{(n-2)}_{\overline{\alpha}}$ for $0\leq j<
Q_{\overline{\alpha}}(n-2)$ do not contain any point of $End(T)$.
This  remark implies that, since in Case (L) we have $\beta_n \neq
(\pi_0^{(n)})^{-1}(1)$ (i.e.~$l^{(n)}_{\beta_n}\neq 0$),  in both
Cases one has $j_n\geq Q_{\overline{\alpha}_n}(n-2)$ and concludes
the proof that (\ref{controlledheight}) hold in all Cases.  Since
$T^{Q_{\beta_n}(n)} I^{(n)}_{\beta_n}\subset I^{(n)} \subsetneq
I^{(n-2)}_{\overline{\alpha}}$ and, in Case (L), we also  have
$\beta_n \neq (\pi_1^{(n)})^{-1}(1)$ (i.e.~$T^{Q_{\beta_n}(n)}
l^{(n)}_{\beta_n} \neq 0$), this remark also shows that the last
part of the Lemma holds.
\end{proof}

\begin{definition}[Class of rigidity sets]\label{beta}
For each $n \in \N$, let $\beta_0$, $\beta_n$ and  $j_n$ be given
by Lemma \ref{remark:wybj}, so that we have $C_{\beta_0}^-\neq 0$
and $\widehat{T}^{j_n} r_{\beta_n} = r_{\beta_0}$ where
$Q_{\overline{\alpha}}(n-2) \leq j_n<Q_{\beta_n}(n)$  (Case (R)),
or $C_{\beta_0}^+\neq 0$ and $ l_{\beta_0}={T}^{j_n}
l^{(n)}_{\beta_n}$ where $Q_{\overline{\alpha}}(n-2) \leq
j_n<Q_{\beta_n}(n)$  (Case (L)). Set $q_n:=Q_{\beta_n}(n)$ and
$p_n:=Q_{\overline{\alpha}}(n-2)$.

%Let us assume that either
%\begin{enumerate}\item
%the left endpoint of $I^{(n)}_{{\overline{\alpha}}, \beta_n}$,
%that we call $a_0$,  is such that $T^{j_n} a_0 = l_{\beta_0}$ for
%some  $j_n$ with $Q_{\overline{\alpha}}(0,n) \leq j_n < q_n$ and
%$\beta_0 \in \mathcal{A}$ such that $C_{\beta_0}^+\neq 0$. $
%I^{(n)}_{\alpha, \beta_0} $ is the largest of the exchanged
%subintervals, so that   $| I^{(n)}_{\alpha, \beta(\alpha)}| >|
%I^{(n)}_\alpha |/d $.

Let  $J^{(n)}_0\subset  I^{(n)}_{ {\beta_n}} $ be any subinterval
such that $|J^{(n)}_0|\geq c | I^{(n)}_{ {\beta_n}}| $ for some
$c$ independent on $n$.  For each $0\leq k < p_n$ set
$J_k^{(n)}:=T^k J_0^{(n)}$ and let
%When $n$ is fixed and there is no ambiguity we will omit the
%dependence on $n$ and write $J_k:= J_k^{(n)}$.
\begin{equation}\label{Cndef}
\Xi_n := \bigcup_{k=0}^{p_n-1}  J_k^{(n)} .
\end{equation}
\end{definition}
\begin{lemma}
For any choice of $J_k^{(n)}$ as in Definition \ref{beta}, the
sets $(\Xi_n)$  defined by (\ref{Cndef})  are rigidity sets with
rigidity times the $(q_n)$.
\end{lemma}
\begin{proof}
From (\ref{balancedintervals}), (\ref{areas}) and  from
$Q_{\overline{\alpha}}(n)\leq \|A\|^2 Q_{\overline{\alpha}}(n-2)$
it follows that
\begin{align}\label{dodelta}
%\begin{aligned}
|\Xi_n| & = \sum_{k=0}^{p_n-1}  |J_k^{(n)}| \geq c\,
Q_{\overline{\alpha}}(n-2) | I^{(n)}_{\beta_n}
|
%\geq\frac{c}{\nu(A)\|A\|^2}\, Q_{\overline{\alpha}}(n) |I^{(n)}_{\overline{\alpha}} |\\&
\geq\frac{c}{d \nu(A)^2\|A\|^2 |I^{(0)}|^2}.
%\end{aligned}
\end{align}
It is easy to check that for all $x \in \Xi_n$, $d(T^{q_n} x, x
)\leq |I^{(n)}|$ (we refer to \cite{Ul:wea} for details) and that
since $\Xi_n$ is a tower over a subset of $I^{(n)}_{\beta_n}$,
$|\Xi_n \Delta T^{-1}\Xi_n |\leq |I^{(n)}| $, which tends to zero
by minimality of $T$. Thus the conditions in (\ref{rigidity})
hold.
\end{proof}

We will now choose  $J_0^{(n)}\subset I^{(n)}_{\beta_n} $ so that
if we set $J_k^{(n)}=T^kJ_0^{(n)}$, then
 for each $x \in {J_k^{(n)}}=T^kJ_0^{(n)}$, $0\leq k < p_n$, the Birkhoff
sums $ (\varphi^{(q_n)})'' (x) $ are large, in the precise sense
of Lemma~\ref{largederivatives} below.  The rigidity sets
$(\Xi_n)$ used in the proof of ergodicity (in
\S\ref{tightness:sec}) will be the ones obtained by  Definition
\ref{beta} from these subintervals $J_k^{(n)}$.  We will also show
that for each $0\leq k < p_n$ we can choose a subinterval
$\widetilde{J}_k^{(n)}\subset J_k^{(n)}$  so that
$(\varphi^{(q_n)})' (x) $ is also large for $x\in
\widetilde{J}_k^{(n)}$ in the sense of
Corollary~\ref{largederivative} below. Since the construction is
basically symmetric in Case $(R)$ and Case $(L)$, we will give all
the details in Case $(R)$ and only the definitions in Case $(L)$.

\begin{definition}\label{Jkdef}
Set $[a_k,b_k):= T^k I^{(n)}_{\beta_n}$   for $0\leq k < p_n$,
where  $\beta_n,   p_n$ are as in Definition \ref{beta}. Recall
that ${\lambda^{(n)}_{\beta_n}}=|I^{(n)}_{\beta_n}|$. Fix $0\leq
\overline{c}<1/2$ and set
\begin{equation}\label{defJk}
\begin{split}
J_k^{(n)}: & = \left( b_k - \overline{c}
{\lambda^{(n)}_{\beta_n}},b_k - \frac{\overline{c}
{\lambda^{(n)}_{\beta_n}} }{2}\right)  \text{ in Case (R) }, \\
J_k^{(n)}: & = \left( a_k +  \frac{\overline{c}
{\lambda^{(n)}_{\beta_n}} }{2}, a_k+  \overline{c}
{\lambda^{(n)}_{\beta_n}}\right)  \text{ in Case (L)}.
\end{split}
\end{equation}
\end{definition}
%The reason for this definition of $\overline{c}$ will be clear in
%the proof of Lemma~\ref{largederivatives} below.
Notice that since $0<
\overline{c} < 1/2$ we have the inclusions
\begin{equation}\label{halfinterval}
{J_k^{(n)}}\subset \left( a_k +
\frac{{\lambda^{(n)}_{\beta_n}}}{2},b_k \right)  \text{ in\ Case
(R)}, \ {J_k^{(n)}}\subset \left(  a_k, a_k +
\frac{{\lambda^{(n)}_{\beta_n}}}{2} \right)  \text{ in\ Case (L)}.
\end{equation}

\begin{lemma}\label{arithmetic progressions}
In Case (R), if $x \in J_k^{(n)}$, for each $0\leq j <  q_n$ we have
%of the points in the orbit segment $\{ T^j x, 0\leq j <  q_n\} $ has the following properties:
\begin{itemize}
\item[(i)] $\{ T^j x - l_\alpha  \}\geq {\lambda^{(n)}_{\beta_n}} /2 $ for all
$\alpha \in \mathcal{A}$;
%it has distance at least ${\lambda^{(n)}_{\beta_n}} /2 $ from each $r_\alpha$, $\alpha \in \mathcal{A}$;
\item[(ii)] $ \{ r_\alpha - T^j x \}\geq {\lambda^{(n)}_{\beta_n}} /\nu(A)  $
for all $\alpha$ such that $C_{\alpha}^-\neq 0$ and $\alpha \neq \beta_0$;
% ,\alpha_{0}\}$, where $\alpha_0=\pi_0(\alpha_0)=d$;
%it has distance at least  ${\lambda^{(n)}_{\beta_n}} /\nu(A) $  from each $l_\alpha$ with $\alpha \neq \beta_0$;
\item[(iii)] $\{ r_{\beta_0} - T^j x \}   \geq {\lambda^{(n)}_{\beta_n}} /\nu(A)  $
%it has distance at least ${\lambda^{(n)}_{\beta_n}} /\nu(A) $ from $l_{\beta_0} $
 with the only exception of $j={j_n-k}$, for which
% singularity of $T$ with the only exception of $T^i x$, for which
$\overline{c} {\lambda^{(n)}_{\beta_n}}/2\leq\{ r_{\beta_0}  -
T^{j_n-k} x \} \leq \overline{c} {\lambda^{(n)}_{\beta_n}}$;
\end{itemize}
Moreover, for all $x \in J_k^{(n)}$,
\begin{itemize}
\item[(iv)] the  minimum spacing of points
in  $\{ T^j x, \, 0\leq j <q_n\}$, i.e.~$\min \{ |T^i x - T^j x|,$
for $ 0\leq i\neq j <q_n\}$, is greater than
${\lambda^{(n)}_{\beta_n}}$.
\end{itemize}
\end{lemma}
\begin{remark}\label{analogous}
In Case (L), one can state and prove a Lemma analogous
\footnote{In the version for Case (L) the statement and the proof
is actually simpler, since there is not need to assume anything as
$\alpha$ such that $C_\alpha^- \neq 0 $ in Part (2).} to Lemma
\ref{arithmetic progressions}, in which the role of $\{ r_\alpha -
T^j x \}$ and $\{  T^j x - l_\alpha  \}$ is reversed.
\end{remark}
\begin{proof}
%If $x \in J_k^{(n)}$, by construction $|x-b_0| \geq {\lambda^{(n)}_{\beta_n}}/2$ and $x$
Recall that $J_0^{(n)}$ is contained in $I^{(n)}_{\beta_n}$ which is a
continuity interval for $T^{q_n}$ and $T^{q_n}
I^{(n)}_{\beta_n}\subset I^{(n)}$ is contained in
$I^{(n-2)}_{\overline{\alpha}}$ which is a continuity interval for
each $T^k$ with $0\leq k < Q_{\overline{\alpha}}(n-2)$. This
implies that, for each $0\leq k < p_n=Q_{\overline{\alpha}}(n-2)$,
the images $T^jT^kI^{(n)}_{\beta_n}$ for $ j =0, \dots, q_n -1 $
do not contain any $l_\alpha$ or $r_\alpha$ in their interiors.

%Moreover, $T^j $ act as an isometry on $I^{(n)}_{\beta_n}$.
Thus, since $J_k^{(n)}\subset T^kI^{(n)}_{\beta_n}$, for each
$x\in J_k^{(n)}$,  $j=0, \dots, q_n-1$ and $\alpha \in
\mathcal{A}$ we have that $\{  T^j x-l_\alpha \}$ is at least the
distance of $x$ from the left endpoint of $T^kI^{(n)}_{\beta_n}$.
By (\ref{halfinterval}) this gives that $\{  T^j x-l_\alpha \}\geq
{\lambda^{(n)}_{\beta_n}}/2$, i.e.~proves (i).

For any $0\leq k < p_n$, by Definition~\ref{beta}, since
$b_0=r^{(n)}_{\beta_n}$, we have $\widehat{T}^{j_n-k }b_k =
\widehat{T}^{j_n}b_0 = r_{\beta_0}$ and $j_n-k\geq 0$. If $x \in
J_k^{(n)}$, by (\ref{halfinterval}), $\overline{c}
{\lambda^{(n)}_{\beta_n}}/2\leq b_k-x\leq \overline{c}
{\lambda^{(n)}_{\beta_n}}$ and since $\widehat{T}^{j_n-k}$ is an
isometry on the interval $[x,b_k]$, this gives  $\overline{c}
{\lambda^{(n)}_{\beta_n}}/2\leq r_{\beta_0}-T^{j_n-k} x \leq
\overline{c} {\lambda^{(n)}_{\beta_n}} $, which gives
$\overline{c} {\lambda^{(n)}_{\beta_n}}/2\leq \{
r_{\beta_0}-T^{j_n-k} x \} \leq \overline{c}
{\lambda^{(n)}_{\beta_n}}$ in (iii).

Let us complete the proof of (iii) and prove (ii). Let $x \in
J_k^{(n)}$ and let us first consider the case $0\leq j < q_n -k$.
Remark that the images $\widehat{T}^l \widehat{I}^{(n)}_{\beta}$
for $0\leq l < Q_{\beta}(n)$ and $\beta \in \mathcal{A}$ are
disjoint and give a partition of $\widehat{I}$, denoted by
$\mathcal{P}_n$. By
Lemma~\ref{comparingsingularities}, %there exists a permutation
%$\chi:\mathcal{A}\to\mathcal{A}$ %($\chi:=\chi_\upsilon(0,m)$)
% such that %if $\chi(\alpha)\neq\alpha_\upsilon$ then
$\{ r_\alpha, \alpha \in \mathcal{A}\}$ are contained in the
orbits of the right endpoints of the intervals $I^{(n)}_{\beta}$,
$\beta \in \mathcal{A}$. Moreover, there exists a unique $\beta'$
%$\chi(\alpha')=\alpha_{1-\upsilon}$
 such that the tower $\widehat{T}^l
\widehat{I}^{(n)}_{\beta'}$, $0\leq l < Q_{\beta'}(n)$ contains
both $r_{\alpha_1}$ and $r_{\alpha_0}=\widehat{T}r_{\alpha_1}$.

By the Keane condition, since the $\widehat{T}$-orbit of
$b_0=r^{(n)}_{\beta_n}$  contains $r_{\beta_0}$ (recall that by
definition $\chi(\beta_n)=\beta_0$), it does not contain any other
$r_\alpha$ but $r_{\beta_0}$, unless either $r_\alpha$ (which
belongs to the orbit) or $r_{\beta_0}$ are equal to $|I|$. In the
latter case, the $\widehat{T}$-orbit of $b_0=r^{(n)}_{\beta_n}$
contains $r_{\alpha_\upsilon}$ (recall that $\alpha_\upsilon\in \{
\pi_0^{-1}(d), \pi_1^{-1}(d)\}$) and, again by Keane's condition,
no other $r_\alpha$. Indeed, one either has  $\alpha_\upsilon=
\pi_1^{-1}(d)$ and $\hat{T}(r_{\alpha_\upsilon})=|I|=r_{\beta_0}$
or $\alpha_\upsilon= \pi_0^{-1}(d)$ and $\hat{T} r_{\beta_0} =
r_{\alpha_{\upsilon}}=|I|$ with $\beta_0=\pi_1^{-1}(d)$. Notice
that in this case, though,  $C^-_{\alpha_\upsilon}=0$. Thus, if $x
\in J_k^{(n)}$, for all $0\leq j < q_n -k$ with the exception of
$j=j_n-k$ and all $\alpha$ for which $C_\alpha^-\neq 0$, we have
that $\{r_{\alpha}- T^j x \}$ is at least the minimum length of an
element of the partition $\mathcal{P}_n$, which, by balance
(\ref{balancedintervals}) of the $I^{(n)}_{\beta}$,
$\beta\in\mathcal{A}$, is at least
$\lambda^{(n)}_{\beta_n}/\nu(A)$.

Let us now consider  $ q_n -k\leq j < q_n$. By the definition of
return time $q_n$, $\widehat{T}^{q_n} \widehat{I}^{(n)}_{\beta_n}
\subset
\widehat{I}^{(n)}\varsubsetneq\widehat{I}^{(n-2)}_{\overline{\alpha}}$.
Thus, for all   $ q_n -k\leq j < q_n$, $T^{j} J_k^{(n)}$ is contained in
the Rohlin tower
$\widehat{T}^l\widehat{I}^{(n-2)}_{\overline{\alpha}}$, $0\leq
l<p_n=Q_{\overline{\alpha}}(n-2)$, which does not contain any
$r_{\alpha}$, $\alpha\in\mathcal{A}$ (see
Lemma~\ref{remark:wybj}). Therefore if $x\in J_k^{(n)}$ then $T^jx$
belongs to an interval of the partition $\mathcal{P}_n$ whose
right endpoint is not of the form $r_{\alpha}$,
$\alpha\in\mathcal{A}$. It follows that $\{r_{\alpha}- T^j x \}$
is at least the minimum length of an element of the partition
$\mathcal{P}_n$, which is at least
$\lambda^{(n)}_{\beta_n}/\nu(A)$.
 This concludes the
proof of (ii) and (iii).

%By  the Keane's condition, the orbit of $a_0$ does not contain any
%other discontinuity of $T$ but $l_{\beta_0}$. Thus, for each
%$j\neq j_n -k$, since the intervals $T^j J_k^{(n)}$ are disjoint,
%we have $T^j x$ for each $$ Thus, thee first three properties
%follows from the properties of the orbits of endpoints of
%$I^{(n)}_{\alpha, \beta}$ together with the definition of
%$J_k^{(n)}$. , since points in ${J_k^{(n)}}$ have by definition
%distance at least $\overline{c}/2{\lambda^{(n)}_{\beta_n}} $ from
%the endpoints of $I^{(n)}_{\alpha, \beta}$.

Property (iv) follows from the fact already remarked that for each
$0\leq k <p_n$  the intervals $T^i T^kI^{(n)}_{\beta_n}$ for
$0\leq i < q_n$ are disjoint and $T^i$ is an isometry on
$T^kI^{(n)}_{\beta_n}$.
\end{proof}

\begin{lemma}\label{lemma:roznqn}
Let $\varphi\in\ls_0(\sqcup_{\alpha\in \mathcal{A}}I_{\alpha})$.
Then for each $x\in J_0^{(n)}$ and $0\leq m<p_n$ we have
\[|\varphi^{(q_n)}(x)-\varphi^{(q_n)}(T^mx)|\leq C_2:=d\nu(A)
(4d\max(1/\overline{c},\nu(A))+M)\bl(\varphi),\] where $M>0$ is
the constant in Corollary~\ref{cancellations_periodic} and
$\overline{c}$ the one in  Definition \ref{Jkdef}.
\end{lemma}

\begin{proof}
Assume without loss of generality that $|I|=1$. Consider the Case
(R). First note that, if $[x,T^{q_n}x]$ denotes the interval with
endpoints $x$ and $T^{q_n}x$, we have
\[|\varphi^{(q_n)}(x)-\varphi^{(q_n)}(T^mx)|=
|\varphi^{(m)}(x)-\varphi^{(m)}(T^{q_n}x)|\leq\int_{[x,T^{q_n}x]}|(\varphi^{(m)})'(y)|\,d
y.\] Fix $y\in [x,T^{q_n}x]\subset I^{(n)}$. As we mentioned
before, the images  $T^jI^{(n)}$ for $ 0\leq j<p_n $ do not
contain any $l_\alpha$ or $r_\alpha$ in their interiors.
Therefore, for every $0\leq j<m$
\[\{T^jy-l_\alpha\}\geq
\min(\{T^jx-l_\alpha\},\{T^jT^{q_n}x-l_\alpha\}),\]
\[\{r_\alpha-T^jy\}\geq
\min(\{r_\alpha-T^jx\},\{r_\alpha-T^jT^{q_n}x\})\] for each
$\alpha\in\mathcal{A}$. Since $T^jT^{q_n}x=T^{q_n-1}(T^{j+1}x)$
with $0<j+1\leq m<p_n$, in view of Lemma~\ref{arithmetic
progressions}, applied to $x\in J_0^{(n)}$ and $T^{j+1}x\in J_{j+1}$, we
have $\{T^jy-l_\alpha\}\geq{\lambda^{(n)}_{\beta_n}}/2$ for all
$\alpha\in\mathcal{A}$ and $\{r_\alpha-T^jy\}\geq
\underline{c}{\lambda^{(n)}_{\beta_n}}/2$ if $C^-_\alpha\neq 0$,
where $\underline{c}=\min(\overline{c},1/\nu(A))$. Therefore,
\[y^l_\alpha=\min_{0\leq j<m}(T^jy-l_\alpha)^+\geq{\lambda^{(n)}_{\beta_n}}/2\
\text{ for all }\alpha\in\mathcal{A},\]
\[y^r_\alpha=\min_{0\leq
j<m}(r_\alpha-T^jy)^+\geq\underline{c}{\lambda^{(n)}_{\beta_n}}/2\
\text{ if }\ C^-_\alpha\neq 0.\] In view of
Corollary~\ref{cancellations_periodic} applied to $k=0$ and $k'=m$
and since $\underline{c}\leq 1$, it follows that
\[|(\varphi^{(m)})'(y)|\leq \sum_{\alpha \in
\mathcal{A}} \frac{|C_\alpha^+|}{{y}_\alpha^l} + \sum_{\alpha \in
\mathcal{A}} \frac{|C_\alpha^-|}{{y}_\alpha^r}  + M \bl(\varphi)
m\leq
\left(\frac{4d}{\underline{c}{\lambda^{(n)}_{\beta_n}}}+Mq_n\right)\bl(\varphi).\]
Therefore
\begin{align*}
|\varphi^{(q_n)}&(x)-\varphi^{(q_n)}(T^mx)|\leq
|x-T^{q_n}x|\left(\frac{4d}{\underline{c}{\lambda^{(n)}_{\beta_n}}}+Mq_n\right)\bl(\varphi)\\
&\leq
|I^{(n)}|\left(\frac{4d}{\underline{c}{\lambda^{(n)}_{\beta_n}}}+Mq_n\right)\bl(\varphi)
\leq
d\nu(A)|I^{(n)}_{\beta_n}|\left(\frac{4d}{\underline{c}{\lambda^{(n)}_{\beta_n}}}+
Mq_n\right)\bl(\varphi)\\ &\leq
d\nu(A)(4d/\underline{c}+M)\bl(\varphi),
\end{align*}
since ${\lambda^{(n)}_{\beta_n}}=|I^{(n)}_{\beta_n}|$ and
$|I^{(n)}_{\beta_n}|q_n=|I^{(n)}_{\beta_n}|Q_{\beta_n}(n)\leq
1$.
The proof of Case (L) is similar.
\end{proof}

For the next Lemma \ref{largederivatives} and its Corollary
\ref{largederivative}, we will consider cocycles $\psi \in
\ls_0(\sqcup_{\alpha\in \mathcal{A}} I_{\alpha})$, with an
additional assumption.  We will consider  $\psi$ of the usual
form, that, for $|I|=1$,  is
\begin{equation}\label{formpsi}
\psi(x)=-\sum_{\alpha\in\mathcal{A}}C^+_\alpha\log\{x-l_\alpha\} -
\sum_{\alpha\in\mathcal{A}}
C^-_\alpha\log\{r_\alpha-x\}+g_\psi(x),
\end{equation}
but in addition we will assume that $g'_\psi \in BV^1 $. This
allows  us to consider $\psi''$.
\begin{lemma}\label{largederivatives}
Let $\psi \in  \ls_0(\sqcup_{\alpha\in \mathcal{A}} I_{\alpha})$
be such that $g'_\psi \in BV^1 $. Consider the intervals
${J_k^{(n)}}$ defined in (\ref{defJk}) with
\begin{equation}\label{cdef}
\overline{c}:= {\left( |C^{\pm}_{\beta_0}|/( \pi^2 \nu(A)^2
\bl(\varphi) + \|g_\psi''\|_{\sup} )\right)}^{1/2}.
\end{equation}
Then for each $x \in J_k^{(n)}$ we have  $| (\psi'')^{(q_n)}(x)| \geq c_1
/(\lambda^{(n)}_{\beta_n})^2$ where the constant $c_1>0$ is
explicitly given by $c_1:= \pi^2\nu(A)^2 \bl(\psi) /3 $.
\end{lemma}
\begin{proof}
Since $g'_\psi \in BV^1 $, we can differentiate (\ref{formpsi})
twice and get
\begin{equation*}
\psi''(x) =  -\sum_{\alpha\in\mathcal{A}} \frac{C^+_\alpha}{\{x-l_\alpha\}^2} -
\sum_{\alpha\in\mathcal{A}} \frac{ C^-_\alpha}{\{ r_\alpha-x\}^2} +g_\psi''(x).
\end{equation*}
Assume that  Case (R) holds and take  $x \in {J_k^{(n)}}$.  By Lemma~\ref{arithmetic
progressions}, the minimum of $\{T^j x -l_\alpha\} $ for $\alpha
\in \mathcal{A}$  and $ 0\leq j < q_n$ is largest than
${\lambda^{(n)}_{\beta_n}}/2$ and the  points  $\{ T^j x, \, 0\leq
j < q_n \}$ are at least ${\lambda^{(n)}_{\beta_n}}$-spaced, so we
have the following upper bound:
%$\alpha \neq \beta_0$, since by Lemma~\ref{arithmetic progressions} to $r_\alpha$ and $l_{\alpha}$
\begin{equation*}
\left| \sum_{0\leq j < q_n} \frac{C^+_\alpha}{\{T^j x
-l_\alpha\}^2} \right| \leq \sum_{j=1}^{q_n} \frac{|C^+_\alpha| }{
j^2 ({\lambda^{(n)}_{\beta_n}}/2)^2 } \leq \frac{4 \pi^2}{6}
\frac{|C^+_\alpha|} {(\lambda^{(n)}_{\beta_n})^2}  .
\end{equation*}
Reasoning in the same way, from (ii) in Lemma~\ref{arithmetic
progressions}, for each $r_\alpha$ such that $C_\alpha^-\neq 0$
and $\alpha \neq \beta_0$ we get an analogous estimate for
\[\left| \sum_{0\leq j < q_n} \frac{C^-_\alpha}{{\{r_\alpha-T^j x
\}}^2} \right|\leq \frac{\pi^2\nu(A)^2}{6} \frac{|C^-_\alpha|} {
(\lambda^{(n)}_{\beta_n})^2} .\] Clearly, the estimate holds
trivially also if $C_\alpha^-= 0$, so it holds for all $\alpha
\neq \beta_0$. Again by (iii) in Lemma~\ref{arithmetic
progressions}, we have that $\{r_{\beta_0}- T^{j_n-k} x  \} \leq
\overline{c} {\lambda^{(n)}_{\beta_n}}$, so that
\[\left|\frac{C^-_{\beta_0}}{\{r_{\beta_0}- T^{j_n-k} x  \}^2}\right| \geq
\frac{|C^-_{\beta_0}|}{ \overline{c}^2
(\lambda^{(n)}_{\beta_n})^2}.\]
%If $j\neq $ $\{ T^{j_n-k}  x  -l_{\beta_0}\} \leq \overline{c} {\lambda^{(n)}_{\beta_n}}$.
If we exclude $T^{j_n-k} x $, for the other points in the orbit
$\{ T^j x, \,  0\leq j < q_n , j\neq j_n-k \}$  we can reason as
above using the lower bound of (iii) in Lemma~\ref{arithmetic
progressions} on the minimal value of $\{r_{\beta_0}-T^j x \}$ and
the lower bound  on the spacing in (iv) to get

%the minimum of  $| T^{j_n x } -l_{\beta_0}|$  $r_{\beta_0}$  in
%the orbit  $\{T^j x, 0\leq j < q_n \}$ (where $x \in
%{J_k^{(n)}}$), that we will denote by $T^{j_n} x$, by definition
%of ${J_k^{(n)}}$  is such that $| T^{j_n x } -l_{\beta_0}| \geq
%{\lambda^{(n)}_{\beta_n}}/2$, so that

\begin{gather*}
\left| \sum_{0\leq j < q_n} \frac{C^-_{\beta_0}}{\{r_{\beta_0}-T^j
x \}^2 } - \frac{C^-_{\beta_0}}{\{r_{\beta_0}- T^{j_n-k} x  \}^2}
\right| \qquad\qquad \\\leq \sum_{j=1}^{q_n-1}
\frac{|C^-_{\beta_0}| }{ j^2 ({\lambda^{(n)}_{\beta_n}}/\nu(A))^2
} \leq \frac{\pi^2\nu(A)^2}{6} \frac{|C^-_{\beta_0}|}{
{(\lambda^{(n)}_{\beta_n})^2}}  .
\end{gather*}
Remark that, since $g_\psi' \in BV^1$,  $| (g_\psi'')^{(q_n)}(y)| \leq q_n
\|g_\psi''\|_{\sup} \leq \|g_\psi'' \|_{\sup}
/{(\lambda^{(n)}_{\beta_n})}^2 $ for each $y\in I$ because
${\lambda^{(n)}_{\beta_n}} q_n
=|I^{(n)}_{\beta_n}|Q_{\beta_n}(n)\leq 1$ and
$1/{\lambda^{(n)}_{\beta_n}} \leq
1/{(\lambda^{(n)}_{\beta_n})}^2$. Combining all the above
estimates and recalling  that $\bl(\psi) = \sum_{\alpha}
(|C_\alpha^+| + |C_\alpha^-|)$, we get
\begin{align*}
|( \psi'')^{(q_n)}(x) | & \geq \left| \left|
\frac{C^-_{\beta_0}}{\{r_{\beta_0}-T^{j_n-k}x  \}^2}  \right| -
\left|( \psi'')^{(q_n)}(x) -
\frac{C^-_{\beta_0}}{\{r_{\beta_0}-T^{j_n-k}x  \}^2}   \right|
\right|  \\ & \geq \frac{|C^-_{\beta_0}|}{ \overline{c}^2
{{\lambda^{(n)}_{\beta_n}}}^2} - \frac{2 \pi^2\nu(A)^2
\bl(\psi)}{3{(\lambda^{(n)}_{\beta_n})}^2} -
\frac{\|g_\psi''\|_{\sup}}{{(\lambda^{(n)}_{\beta_n})}^2}.
\end{align*}
Recalling the  definition (\ref{cdef}) of $\overline{c}$, this
gives $|( \psi'')^{(q_n)}(x) | \geq  \pi^2\nu(A)^2
\bl(\psi)/3{(\lambda^{(n)}_{\beta_n})}^2$ and concludes the proof
of the lemma for the Case (R). The Case (L) is similar.
\end{proof}

\begin{corollary}\label{largederivative}
If $g'_\psi \in BV^1$ then  for every $0\leq k<p_n$ there exists a
subinterval $\widetilde{J}_k^{(n)} \subset {J_k^{(n)}}$ such that
$|\widetilde{J}_k^{(n)}|\geq |{J_k^{(n)}}|/3$ and for each $x \in
\widetilde{J}_k^{(n)}$ we have
\begin{equation*}
|( \psi^{(q_n)})'(x)| \geq c' q_n , \quad \text{where} \quad
c'=\pi^2\nu(A) ^2 \overline{c}\bl(\psi)/36.
\end{equation*}
\end{corollary}
\begin{proof}
%Let us subdivide each ${J_k^{(n)}}$ into three subintervals of
%equal size
%\begin{equation*}
%{J_k^{(n)}}^1 =\left( a_k + \frac{\overline{c} {\lambda^{(n)}_{\beta_n}} }{2},
%a_k + \frac{4\overline{c} {\lambda^{(n)}_{\beta_n}}}{6} \right), \, {J_k^{(n)}}^2 =
%\left( a_k + \frac{4\overline{c} {\lambda^{(n)}_{\beta_n}}}{6} ,
% a_k + \frac{5\overline{c} {\lambda^{(n)}_{\beta_n}}}{6} \right) \, {J_k^{(n)}}^2 =
% \left( a_k + \frac{5\overline{c} {\lambda^{(n)}_{\beta_n}}}{6}  ,
% a_k + \overline{c} {\lambda^{(n)}_{\beta_n}} \right) .
%\end{equation*}
%Since $\psi' T^i$ is continuous on ${J_k^{(n)}}$ for $i=0, \dots,
%q_n$,  by mean value theorem we have that, for each $l=1,2,3$, the
%difference of values  of  $|S_{q_n }\psi'|$ at the endpoints of
%${J_k^{(n)}}^l$ is given by $|S_{q_n }\psi'(\xi^l)|
%|{J_k^{(n)}}^l|$  for some $\xi^l \in {J_k^{(n)}}^l$. Thus,  by
%Lemma~\ref{largederivatives}, the difference  is at least $(c /
%{(\lambda^{(n)}_{\beta_n})}^2) |{J_k^{(n)}}^l| =c/
%3{\lambda^{(n)}_{\beta_n}} $, where $c$ is the constant in
%Lemma~\ref{largederivatives}.
By Lemma~\ref{largederivatives}, the sign of $( \psi^{(q_n)})''$
is constant on ${J_k^{(n)}}$, so assume without loss of generality
that $( \psi^{(q_n)})''>0$, so that $( \psi^{(q_n)})'$ is
increasing on $J_k^{(n)}$. Assume we are in Case (R). Consider the
value of $( \psi^{(q_n)})'$ at the middle point
$b_k-{3\overline{c} {\lambda^{(n)}_{\beta_n}}}/{4}$ of
${J_k^{(n)}}$. If $( \psi^{(q_n)})'(b_k-{3\overline{c}
{\lambda^{(n)}_{\beta_n}}}/{4})\geq 0$, let
$\widetilde{J}_k^{(n)}$ be the right third subinterval of
${J_k^{(n)}}$, i.e. $\widetilde{J}_k^{(n)}:=\left[ b_k -
{2\overline{c} {\lambda^{(n)}_{\beta_n}}}/{3} , b_k -
{\overline{c} {\lambda^{(n)}_{\beta_n}}}/{2}  \right]$. Since
$\psi' \circ T^i$ is continuous on $J_k^{(n)}$ for $0\leq i< q_n$,
by mean value theorem and by monotonicity, there exists $\xi \in
(b_k - {3\overline{c} {\lambda^{(n)}_{\beta_n}}}/{4}  , b_k -
{2\overline{c} {\lambda^{(n)}_{\beta_n}}}/{3} )$ such that for
each $x \in \widetilde{J}_k^{(n)}$
\begin{align*}
( \psi^{(q_n)})'(x)&\geq ( \psi^{(q_n)})'\left( b_k -
\frac{2\overline{c} {\lambda^{(n)}_{\beta_n}}}{3} \right)\\ &= (
\psi^{(q_n)})'\left(b_k-\frac{3\overline{c}
{\lambda^{(n)}_{\beta_n}}}{4} \right) + ( \psi^{(q_n)})''(\xi)
\frac{{\lambda^{(n)}_{\beta_n}}\overline{c}}{12} \geq
\frac{\overline{c}c_1}{12{\lambda^{(n)}_{\beta_n}}} ,
\end{align*}
where the latter inequality follows from positivity of $(
\psi^{(q_n)})'(b_k-{3\overline{c}
{\lambda^{(n)}_{\beta_n}}}/{4} )$ and the lower bound $(
\psi^{(q_n)})''(\xi)\geq c_1/{(\lambda^{(n)}_{\beta_n})}^2$
given by Lemma~\ref{largederivatives}.
% this implies that $S_{q_n }\psi'\geq $ on ${J_k^{(n)}}$.

Similarly, if $( \psi^{(q_n)})'(b_k-{3\overline{c}
{\lambda^{(n)}_{\beta_n}}}/{4})\leq 0$, we can  let
$\widetilde{J}_k^{(n)}$ be the left third subinterval of
$\widetilde{J}_k^{(n)}$, i.e. $\widetilde{J}_k^{(n)}:=  \left[b_k
- \overline{c} {\lambda^{(n)}_{\beta_n}} , b_k - {5\overline{c}
{\lambda^{(n)}_{\beta_n}}}/{6} \right]$ and reasoning as above we
get $( \psi^{(q_n)})'(x)\leq -
\frac{c_1\overline{c}}{12{\lambda^{(n)}_{\beta_n}}}$ for all $x\in
\widetilde{J}_k^{(n)}$. Recalling that ${\lambda^{(n)}_{\beta_n}}
q_n\leq 1$ and the definition  of $c_1$, this concludes the proof
in Case (R). Case (L) is completely symmetric.
\end{proof}
%The conclusion follows recalling the definition of $c$ in
%Lemma~\ref{largederivatives} and remarking that $q_n
%{\lambda^{(n)}_{\beta_n}} \leq 1$, since $ I^{(n)}_{\alpha ,
%\beta}$ has lenght ${\lambda^{(n)}_{\beta_n}} $ and $q_n$ disjoint
%images. the difference of values  of  $|S_{q_n }\varphi'|$ at the
%endpoints of ${J_k^{(n)}}^l$ is given by $|S_{q_n
%}\varphi'(\xi^l)| |{J_k^{(n)}}^l|$  for some $\xi^l \in
%{J_k^{(n)}}^l$. Thus,  by Lemma~\ref{largederivatives}, the
%difference  is at least $(c / {(\lambda^{(n)}_{\beta_n})}^2)
%|{J_k^{(n)}}^l| =c/ 3{\lambda^{(n)}_{\beta_n}} $, where $c$ is the
%constant in Lemma~\ref{largederivatives}.  Morover, $S\varphi'$
%his concludes the construction  under the assumption that there
%exists $\beta_0 \in \mathcal{A}$ such that $C_\beta^-\neq 0$. Let
%us comment on how to construct intervals $J_k^{(n)}'$ with the
%same properties in the case that $\bl(\varphi)\neq 0$ but
%$C_\alpha^-=0$ for all $\alpha\in \mathcal{A}$. The construction
%is in the same spirit, so we leave the details to the reader.
%Since $\psi$ has singularities of geometric type, at least one
%among  $C_{\pi_0^{-1}(1)}^+$ and  $C_{\pi_1^{-1}(1)}^+$ is zero.
%Thus, since $\bl(\varphi)\neq 0$,  there must exists $\beta_1$
%such that $C_{\beta_1}^+\neq 0$ and $\beta_1 \notin \{
%\pi_0^{-1}(1) , \pi_1^{-1}(1) \}$. Remark that under Rauzy-Veech
%induction, since the first interval $I^{n}_{{{\pi^n}_0}^{-1}}$ is
%never changed,  $(\pi^n_0)^{-1}(1)= \pi_0^{-1}(1)$  and
%$(\pi^n_1)^{-1}(1)=\pi_1^{-1}(1)$ for any $n \in \N$. Thus, for
%any $n \in \N$, there exists $\beta_n$ such that

\subsection{Tightness and ergodicity}\label{tightness:sec}
In this subsection we conclude the proof of Theorem
\ref{theorem:ergmain}.
%Let $T:I\to I$ be an IET of hyperbolic periodic type and
%$\varphi\in\ls_0(\sqcup_{\alpha\in\mathcal{A}}I_{\alpha})$ such
%that $\mathfrak{h}(\varphi)=0$ and $\bl(\varphi)\neq 0$.
%Since $g_\varphi\in\bv^1(\sqcup_{\alpha\in \mathcal{A}}
%I_{\alpha})$, by Theorem~\ref{cohcon}, $g_\varphi$ is
%cohomologous via a continuous transfer function to a piecewise linear function.
% Thus, there exists a continuous $h:I\to\R$ such that
%$\varphi = \psi + h  \cdot T - h$ and $g_\psi$ is piecewise
%linear. In particular, $g_\psi' \in BV^1$, so we can apply
%Corollary \ref{largederivative} to $\psi$. Let $J_k^{(n)}'$ be the
%intervals given by Corollary \ref{largederivative} and let
%$(\Xi_n)$ and let $(q_n)$ be the sequences of rigidity sets and
%times constructed as in Definition~\ref{beta} staring from the
%intervals ${\widetilde{J}_k^{(n)}}$.
We will verify that the assumptions the ergodicity criterion in
Proposition~\ref{ergodicity:criterium} hold for the rigidity sets
$(\Xi_n)$ and rigidity times $(q_n)$ constructed  in the previous
\S\ref{oscillationsec}. We first prove the following.

\begin{proposition}\label{tightness}
Let $T:I\to I$ be an IET of periodic type. For every cocycle
$\varphi\in\ls_0(\sqcup_{\alpha\in\mathcal{A}}I_{\alpha})$  with
$\mathfrak{h}(\varphi)=0$ and $\bl(\varphi)\neq 0$\footnote{We
remark that the assumption $\bl(\varphi)\neq 0$ is used only to
define the sets $(\Xi_n)$ as in Definition (\ref{beta}), but does
not play any essential role in this Proposition. The same
conclusion holds more in general for similar rigidity sets also
when $\bl(\varphi)= 0$. On the other hand the assumption
$\mathfrak{h}(\varphi)=0$ is crucial in this Proposition, while
the assumption $\bl(\varphi)\neq 0$ is crucial in
Proposition~\ref{oscillations}.} any rigidity sets $(\Xi_n)$ and
rigidity times $(q_n)$ as in Definition \ref{beta} there exists
$C>0$ such that
\begin{equation}\label{bounded}
\int_{\Xi_n}|\varphi^{(q_n)}(x)|\ud x \leq C\quad\text{ for all
}\quad n\geq 1.
\end{equation}
\end{proposition}
\begin{proof}
Let $(\Xi_n)$ and $(q_n)$ by any rigidity sets and times as in
Definition \ref{beta}. Let us first prove that there exists a
constant $C_1>0$ such that for any $n\in \mathbb{N}$ and for any
subinterval  $J \subset I^{(n)}_{\beta_n}$
\begin{equation}\label{controloverJ}
\int_J | \varphi^{(q_n)}(x)| \, \ud x \leq  C_1 | I^{(n)}|.
\end{equation}
Recall that for $x\in I^{(n)}_{\beta_n}$ we have
$S(n)(\varphi)(x)=\varphi^{(Q_{\beta_n}(n))}(x)=\varphi^{(q_n)}(x)$.
Hence
\[\int_J | \varphi^{(q_n)}(x)| \; \ud x=\int_J | S(n)(\varphi)| \; \ud x
\leq \| S(n)(\varphi)\|_{L^1(I^{(n)})}\] Thus,
(\ref{controloverJ}) follows from
Theorem~\ref{operatorcorrection}.

Let us now fix any $0\leq k < p_n $. Given $x \in J_k^{(n)}$, let
$x=T^{k}y$ for some $y \in J_0^{(n)}$. By Lemma~\ref{lemma:roznqn},
$|\varphi^{(q_n)}(y)-\varphi^{(q_n)}(T^ky)|\leq C_2$, so
\[|\varphi^{(q_n)}(x)|\leq |\varphi^{(q_n)}(T^{-k}x)|+C_2\text{ for each }x\in J_k^{(n)}.\]
Thus, by (\ref{controloverJ}), it follows that
\begin{align*}\int_{J_k^{(n)}}|\varphi^{(q_n)}(x)|\ud x&\leq
\int_{J_k^{(n)}}|\varphi^{(q_n)}(T^{-k}x)|\ud x+C_2|J_k^{(n)}|\\
&=\int_{J_0^{(n)}}|\varphi^{(q_n)}(x)|\ud x+C_2|J_k^{(n)}|\leq
(C_1+C_2)|I^{(n)}|.
\end{align*}
Consequently,
\begin{align*}
\int_{\Xi_n} | \varphi^{(q_n)}| \, \ud x &= \sum_{k=0}^{p_n-1}
\int_{J_k^{(n)}} | \varphi^{(q_n)}| \, \ud x\leq
(C_1+C_2)p_n|I^{(n)}|\\&\leq
(C_1+C_2)|I^{(n-2)}_{\overline{\alpha}}|Q_{\overline{\alpha}}(n-2)\leq
C_1+C_2,
\end{align*}
which concludes the proof.
\end{proof}

%\subsection{Oscillations and ergodicity}

\begin{proposition}\label{oscillations}
Let $T:I\to I$ be an IET of periodic type. For each
$\varphi\in\ls_0(\sqcup_{\alpha\in\mathcal{A}}I_{\alpha})$ such
that $\bl(\varphi)\neq 0$  there exists rigidity sets $(\Xi_n)$
and rigidity times $(q_n)$  with
$\lim_{n\to\infty}Leb(\Xi_n)=\delta>0$ and  $c>0$ such that for
all $s$ large enough we have
\begin{equation}\label{wm}
\limsup_{n\to\infty}\left|\int_{\Xi_n}e^{2\pi i s
\varphi^{(q_n)}(x)}\, \ud x\right| \leq c< \delta.
\end{equation}
\end{proposition}
\begin{proof}
Since $g_\varphi\in\bv^1(\sqcup_{\alpha\in \mathcal{A}}
I_{\alpha})$, by Corollary~\ref{lempl}, $g_\varphi$ is
cohomologous via a continuous transfer function to a piecewise
linear function.
 Thus, there exists a continuous $h:I\to\R$ such that
$\varphi = \psi + h \cdot T - h$ and $g_\psi$ is piecewise linear.
In particular, $g_\psi' \in BV^1$, so we can apply Corollary
\ref{largederivative} to $\psi$. Let $(\Xi_n)$ and let $(q_n)$ be
the sequences of rigidity sets and times  as in
Definitions~\ref{beta} and \ref{Jkdef}, where the constant
$\overline{c}$ is given by (\ref{cdef}).  In view of
(\ref{dodelta}), passing to a subsequence if necessary, we can
assume that $\lim_{n\to\infty}Leb(\Xi_n)=\delta>0$.

Since $h$ is continuous and  by the properties of rigidity sets
$d(T^{q_n}x,x)\to 0$,
we have
\begin{equation}\label{remark:przecoh}
\lim_{n\to\infty}\left|\int_{\Xi_n}e^{2\pi i s
(\psi^{(q_n)}(x)+h(T^{q_n}x)-h(x))}\,\ud x-
\int_{\Xi_n}e^{2\pi i s \psi^{(q_n)}(x)}\,\ud x\right|= 0.
\end{equation}
In view of (\ref{remark:przecoh}), since $\varphi^{(q_n)} =
\psi^{(q_n)}+h\circ T^{q_n}-h$, it is enough to prove (\ref{wm})
for $\psi$. Since $\Xi_n$ is the union of the intervals
${J_k^{(n)}}$ for $k=0, \dots, p_n-1$, we will estimate the
integral over each  ${J_k^{(n)}}:= [a_k,b_k]$.  Let
${\widetilde{J}_k^{(n)}}:= [\widetilde{a}_k,\widetilde{b}_k]
\subset J_k^{(n)}$,  for $k=0, \dots, p_n-1$, be the subintervals
given by Corollary \ref{largederivative}. We will first control
the integral over each $\widetilde{J}_k^{(n)}$.
%\begin{equation}\label{Ckintegralsum}
% \int_{\Xi_n} e^{ i s \BS{ \varphi}{q_n}(x)}\ud x  = \sum_{k=0}^{p_n-1}%%%^{q_n-1}
%   \int_{J_{k}} e^{ i s \varphi^{(q_n)}(x)}\ud x  .%= \sum_{i=0}^{q_n-1}
% \int_{J_k^{(n)}} \\varphirac{ \\varphirac{\ud x}{\ud x x}( e^{2\pi i {\overline{\alpha}}
%\varphi^{(q_n)}(x)}) }{2\pi i {\overline{\alpha}} {\varphi'}^{(q_n)}(x)}\ud x x.
%\end{equation}
Since a.e. $ \frac{ d }{ d x} ({ \psi}^{(q_n)})={\psi'}^{(q_n)} $
and $|{\psi'}^{(q_n)}|\geq c'q_n>0$ on each
$\widetilde{J}_k^{(n)}$ (Corollary~\ref{largederivative}), using
integration by parts we get
%and, since
%\begin{equation*}
% e^{2\pi i {\overline{\alpha}} \BS{f}{q_n}(x)} = \frac{ \frac{\ud x}{\ud x x}(
%e^{2\pi i {\overline{\alpha}} \BS{f}{q_n}(x)}) }{2\pi i {\overline{\alpha}} \BS{f'}{q_n}(x)},
%\end{equation*}
%Integrating by parts each of the integrals on $J_k^{(n)} = (\widetilde{a}_k,\widetilde{b}_k)$, we get
\begin{align} \label{parts} %\label{parts}
%\begin{equation}\label{parts}
\begin{split}
&\left| \int_{\widetilde{J}_k^{(n)}} e^{ i s  \psi^{(q_n)}(x)}\;d
x \right|=
%= \sum_{i=0}^{q_n-1}  \int_{\widetilde{J}_k^{(n)}} \frac{ \frac{\ud x}{\ud x x}(
%e^{2\pi i {\overline{\alpha}} \BS{f}{q_n}(x)}) }{2\pi i {\overline{\alpha}} \BS{f'}{q_n}(x)}\ud x x
\left|  \int_{\widetilde{a}_k}^{\widetilde{b}_k} \frac{ \frac{ d
}{ d x}( e^{ i s
 \psi^{(q_n)}(x)}) } { i s \psi'^{(q_n)}(x)}\,d x
%\int_{\widetilde{a}_k}^{\widetilde{b}_k}  e^{2\pi i {\overline{\alpha}} \BS{f}{q_n}(x)}\ud x x
\right|  \\
&\;\;\;\;= \frac{1}{|s|}
%\frac{
\left| {\left[ \frac{ e^{ i s  \psi^{(q_n)}(x)}
}{\psi'^{(q_n)}(x)} \right]}^{\widetilde{b}_k}_{\widetilde{a}_k} -
\int_{\widetilde{a}_k}^{\widetilde{b}_k}
  e^{ i s  \psi^{(q_n)}(x)} \frac{ d }{ d x} \left( \frac{1}{ \psi'^{(q_n)}(x)}\right) \,dx  \right| .
%}{2\pi|{\overline{\alpha}}|}.
\end{split}
\end{align}
%\end{equation}
Let us estimate each of the two terms in (\ref{parts}) separately.
By Corollary~\ref{largederivative},
\begin{equation}\label{firstpart}
\left| \left[ \frac{ e^{ i s  \psi^{(q_n)}(x)} }{\psi'^{(q_n)}(x)}
\right]^{\widetilde{b}_k}_{\widetilde{a}_k} \right| \leq
\frac{2}{\min_{z\in {\widetilde{J}_k^{(n)}} }|\psi'^{(q_n)}(z)|
}\leq \frac{2}{c' q_n}.
\end{equation}
Recall that for every $\mathscr{C}^{1}$-function $f:J\to\R$  we
have $\Var{{f}}{J} =  \int_J |{f}'|\,\ud x$ and that if $|{f}|>0$
then $\Var{1/{f}}{J} \leq \Var{{f}}{J}/ (\min_J {f})^{2}$. Since
$\psi^{(q_n)}$ is $\mathscr{C}^1$ on ${\widetilde{J}_k^{(n)}}$,
using again Corollary~\ref{largederivative} we estimate the second
term by
\begin{align*}
\left|  \int_{\widetilde{a}_k}^{\widetilde{b}_k}  e^{ i s
{\psi}^{(q_n)}(x)} \frac{ d}{ d x} \left( \frac{1}{
{\psi'}^{(q_n)}(x)}\right) dx \right| &\leq \Var{  \frac{1}{
{\psi'}^{(q_n)}}}{{\widetilde{J}_k^{(n)}}} \leq \frac{1}{c'^2
q_n^2} \Var{{ {\psi'}^{(q_n)}}}{{\widetilde{J}_k^{(n)}}}.
\end{align*}
%Moreover,  since $T^j  {\widetilde{J}_k^{(n)}} \subset T^j
%(\Delta^{n_k}_{j_0})_{l_0} $ which, by Lemma~\ref{disjoint} are
%all disjoint for $0\leq j< q_n$ and, since the minimum distance of
%each  $T^j  {\widetilde{J}_k^{(n)}}$ from  $0$ and $1$ is at least
%by the distance of  $T^j  {\widetilde{J}_k^{(n)}}$ from the
%endpoints of of $ T^j (\Delta^{n_k}_{j_0})_{l_0} $, i.e.
%$\beta_n{\lambda^{(n)}_{\beta_n}}$ where $\beta_n=1/2 - \beta_1$,
%so that \subset [{\lambda^{(n)}_{\beta_n}} , 1- \lamdba]$
We can write $\Var{{ {\psi'}^{(q_n)}}}{ {\widetilde{J}_k^{(n)}}} =
\Var{ \sum_{j=0}^{q_n-1}\psi'\cdot T^j}{ {\widetilde{J}_k^{(n)}}}
\leq \sum_{j=0}^{q_n-1}\Var{\psi'}{T^j {\widetilde{J}_k^{(n)}}} $.
Assume without loss of generality that $|I|=1$. Thus
\begin{equation*}
\psi'(x) =  \sum_{\alpha\in\mathcal{A}}
\frac{C^+_\alpha}{\{x-l_\alpha\}} - \sum_{\alpha\in\mathcal{A}}
\frac{ C^-_\alpha}{\{ r_\alpha-x\}} +g_\psi'(x),
\end{equation*}
where $g_\psi'$ is of bounded variation. By Lemma~\ref{arithmetic
progressions}, if we are in the Case (R)  of Definition
\ref{Jkdef} or by Remark \ref{analogous}, if we are in the Case
(L), the minimum distance of each $T^j  {\widetilde{J}_k^{(n)}}$
from each $l_\alpha$, $\alpha \in \mathcal{A}$ and $r_\alpha$, for
all $\alpha \in \mathcal{A}$ such that $C^-_\alpha\neq 0$, is at
least $\underline{c} {\lambda^{(n)}_{\beta_n}}/2$, where
$\underline{c} := \min ( \overline{c} ,\nu(A)^{-1}) $ and
${\lambda^{(n)}_{\beta_n}} = |I^{(n)}_{\beta_n}|$. Since the
intervals $T^j{\widetilde{J}_k^{(n)}}$, $0\leq j<q_n$ are pairwise
disjoint, it follows that
\[\sum_{j=0}^{q_n-1}\Var{\frac{C^+_\alpha}{\{x-l_\alpha\}}}{T^j {\widetilde{J}_k^{(n)}}}
\leq \Var{\frac{C^+_\alpha}{\{x\}}}{[\underline{c}
{\lambda^{(n)}_{\beta_n}}/2, 1]}\leq \frac{2|C^+_\alpha|}{\underline{c} {\lambda^{(n)}_{\beta_n}}},\]
\[\sum_{j=0}^{q_n-1}\Var{\frac{C^-_\alpha}{\{r_\alpha-x\}}}{T^j {\widetilde{J}_k^{(n)}}}
\leq \Var{\frac{C^-_\alpha}{\{1-x\}}}{[0,1-\underline{c}
{\lambda^{(n)}_{\beta_n}}/2]}\leq \frac{2|C^-_\alpha|}{\underline{c} {\lambda^{(n)}_{\beta_n}}}.\]
Moreover,
\[\sum_{j=0}^{q_n-1}\Var{g'_\psi}{T^j {\widetilde{J}_k^{(n)}}}\leq \Var{g'_\psi}{I}.\]
Therefore
\begin{equation}\label{secondpart}
\Var{{ {\psi'}^{(q_n)}}}{ {\widetilde{J}_k^{(n)}}}
=\sum_{j=0}^{q_n-1}\Var{\psi'}{ T^j {\widetilde{J}_k^{(n)}}}  \leq
\frac{2 \bl(\psi)}{\underline{c}{\lambda^{(n)}_{\beta_n}}} +
\Var{g_\psi'}{I}.
\end{equation}
Using the estimates (\ref{firstpart}) and (\ref{secondpart}) in
(\ref{parts}),  for each $k=0,\ldots, p_n-1$ we get
\[\left| \int_{\widetilde{J}_k^{(n)}} e^{ i s  \psi^{(q_n)}(x)}\;d
x \right|\leq \frac{1}{|s|}\left(\frac{2}{c' q_n}+\frac{1}{c'^2
q_n^2}\left(\frac{2
\bl(\psi)}{\underline{c}{\lambda^{(n)}_{\beta_n}}} +
\Var{g_\psi'}{I}\right)\right)\leq \frac{C}{p_n|s|}, \] where
$C:=2/c'+1/c'^2\left(2 d\nu(A)^2 |I^{(0)}|\bl(\psi)/\underline{c}+
\Var{g_\psi'}{I}\right)$, since $p_n \leq q_n$ and
${\lambda^{(n)}_{\beta_n}}
q_n=|I^{(n)}_{\beta_n}|Q_{\beta_n}(n)\geq 1/d\nu^2(A)|I^{(0)}|$,
by (\ref{areas}).

As $|\widetilde{J}_k^{(n)}|\geq |J_k^{(n)}|/3$ for all $0\leq k
<p_n$, we have
$Leb(\Xi_n\setminus\bigcup_{k=0}^{p_n-1}\widetilde{J}_k^{(n)})\leq
\frac{2}{3}Leb(\Xi_n)$, and hence
\begin{align*}\label{tendstozero}
\left| \int_{\Xi_n} e^{ i s {\psi}^{(q_n)}(x)}\ud x \right| \leq
\frac{2}{3}Leb(\Xi_n)+\sum_{k=0}^{p_n-1}\left|
\int_{\widetilde{J}_k^{(n)}} e^{ i s  \psi^{(q_n)}(x)}\;\ud x
\right|\leq \frac{2}{3}Leb(\Xi_n)+\frac{C}{|s|}.
\end{align*}
Consequently, whenever $|s|\geq 12C/\delta$,
\[\limsup_{n\to\infty}\left| \int_{\Xi_n} e^{ i s {\psi}^{(q_n)}(x)}\ud x
\right|\leq  \frac{2}{3}\delta+\frac{C}{|s|}
\leq\frac{3}{4}\delta.\]
\end{proof}

\begin{corollary}\label{nocoboudary}
For every IET $T:I\to I$ of periodic type if
$\varphi\in\ls_0(\sqcup_{\alpha\in\mathcal{A}}I_{\alpha})$ is a
cocycle with $\bl(\varphi)\neq 0$ then $\varphi$ is not a
coboundary.
\end{corollary}

\begin{proof}
Assume by contradiction that $\varphi=h-h\circ T$ for some
measurable $h:I\to\R$,  so for any $n \in \N$ we have
$\varphi^{(q_n)}= h\circ T^{q_n}-h$. Since   by Lusin's theorem we
can approximate $h$ by a uniformly continuous function on a set of
measure tending to one and by the properties of rigidity sets
$d(T^{q_n}x,x)\to 0$,  for every real $s$ we have
% for every measurable function $h:I\to\R$
%we have
%\begin{equation}\label{remark:przecoh}
%\lim_{n\to\infty}\left|\int_{\Xi_n}e^{2\pi i s
%(\varphi^{(q_n)}(x)+h(T^{q_n}x)-h(x))}\,dLeb(x)-
%\int_{\Xi_n}e^{2\pi i s \varphi^{(q_n)}(x)}\,dLeb(x)\right|= 0,
%\end{equation}
%% since   by Lusin's theorem we can approximate$h$
%by a uniformly continuous function on a set of measure tending
%to one and since
%Hhen, by(\ref{remark:przecoh}),

\begin{align*}
 \limsup_{n\to\infty}\left|\int_{\Xi_n}e^{2\pi i s
\varphi^{(q_n)}(x)}\, \ud x\right|&
=\limsup_{n\to\infty}\left|\int_{\Xi_n}e^{2\pi i
s(h(T^{q_n}x)-h(x))}\, \ud x\right|\\&
=\lim_{n\to\infty}Leb(\Xi_n)=\delta,
\end{align*}
which contradicts Proposition~\ref{oscillations}.  Thus, $\varphi$
cannot be a coboundary.
\end{proof}

%\begin{theorem}\label{theorem:ergmain}
%Let $T:I\to I$ be an IET of periodic type and
%$\varphi\in\ls_0(\sqcup_{\alpha\in\mathcal{A}}I_{\alpha})$ such
%that $\bl(\varphi)\neq 0$. Then for every $\varphi\in
%P^{(0)}\varphi$ the skew product $T_{\varphi}$ is
%ergodic. In particular, if $\bl(\varphi)\neq 0$ and
%$\mathfrak{h}(\varphi)=0$ then the skew product $T_{{\varphi}}$ is
%\end{theorem}

\begin{proofof}{Theorem}{theorem:ergmain}
Consider the rigidity sets and times $(\Xi_n)$, $(q_n)$, given by
Proposition \ref{oscillations}. Since they belong to the class in
Definition \ref{beta}, they also satisfy Proposition
\ref{tightness}.  Ergodicity of the skew product
$T_\varphi:I\times\R\to I\times\R$ now follows from
Proposition~\ref{tightness} and Proposition~\ref{oscillations} by
the criterion in Proposition~\ref{ergodicity:criterium}.
\end{proofof}

%\begin{proofof}{Theorem}{mainIETs}
%To complete
%\end{proofof}

\section{Reduction of locally Hamiltonian flows to skew products}\label{reduction:sec}
In this section we prove Theorem~\ref{reductionthm} (all details
are placed in Appendix~\ref{singularities}) and
Theorem~\ref{mainHamiltonian} (see \S\ref{final:sec}). Let us
first recall how to represent a locally Hamiltonian flow
$(\phi_t)_{t\in\mathbb{R}}$ as a special flow over an IET and set
up the notation that we use in the rest of this section.

\subsection{Special flow representation of locally Hamiltonian flows}
Let $(\phi_t)_{t\in\mathbb{R}}$ be a locally Hamiltonian flow
 determined by a closed $1$-form $\eta$ on a symplectic surface
$(\Surf,\omega)$. Recall that we assume that there are no saddle
connections and that the local Hamiltonian is a Morse function, so
all zeros (elements of $\Sigma$) are simple saddles. Let
$(\mathscr{F},\nu_{\mathscr{F}})$ be the measured foliation given
by $(\phi_t)_{t\in\mathbb{R}}$ (see the Introduction). By a
theorem of Calabi~\cite{Calabi} and  Katok \cite{Ka0}, there
exists an Abelian differential $\alpha$ on $\Surf$ such that the
vertical measured foliation of $\alpha$ coincides with the
measured foliation $(\mathscr{F},\nu_{\mathscr{F}})$. Moreover, at
each point $z\in\Sigma$ the Abelian differential $\alpha$ has zero
with multiplicity $1$. Denote by $X_\alpha:\Surf\setminus\Sigma\to T\Surf$
the
 vertical vector field, i.e.\ $\alpha(X_\alpha)=i$,
and let $(F^v_t)_{t\in\R}$ stand for the corresponding vertical
flow on $\Surf\setminus\Sigma$. The vertical flow $(F^v_t)_{t\in\R}$
preserves the $2$-form $\omega_\alpha=\frac{i}{2}\alpha \wedge
\overline{\alpha}$ on $\Surf$ which vanishes on $\Sigma$. It follows
that there exists a non-negative smooth function $W:\Surf\to\R$ with zeros
 at $\Sigma$, and such that $\omega_\alpha=W\omega$.
Therefore, $X=WX_\alpha$ on $\Surf\setminus\Sigma$. It follows that
there exists a smooth time change function $h:\R\times \Surf\to\R$
such that $\phi_tx=F^v_{h(t,x)}x$, or equivalently   $W(\phi_t
x)=\frac{\partial h}{\partial t}(t,x)$ with $h(0,x)=x$.

We will consider so called  regular adapted coordinates on
$\Surf\setminus\Sigma$, this is coordinates $\zeta$ relatively to
which $\alpha_\zeta=d\zeta$. If $p\in \Sigma$ is a singular point
 then we consider singular adapted
coordinates around $p$, this is coordinates $\zeta$ relatively to
which $\alpha_\zeta=id\frac{\zeta^{2}}{2}=i\zeta\,d\zeta$. Then
all changes of regular coordinates are given by translations. If
$\zeta'$ is a regular adapted coordinate and $\zeta$ is a singular
adapted coordinate, then $\zeta'=i\zeta^{2}/2+c$. Then for a
regular adapted coordinate $\zeta$ we have
$\omega_\alpha=d\Re\zeta\wedge d\Im\zeta$, $X_\alpha(\zeta)=i$ and
$F^v_t\zeta=\zeta+it$. Moreover, for a singular adapted coordinate
$\zeta$ we have $\omega_\alpha=|\zeta|^2d\Re\zeta\wedge
d\Im\zeta$, $\zeta X_\alpha(\zeta)=1$, and hence
$X_\alpha(\zeta)=\frac{\overline{\zeta}}{|\zeta|^{2}}$. It follows
that for a singular adapted coordinate $\zeta=x+iy$ we have
$W(\zeta)=|\zeta|^2V(\zeta)$, where $V$ is a smooth positive
function. Hence, $X(\zeta)=V(\zeta)\overline{\zeta}=V(x,y)(x,-y)$.

Let $J\subset \Surf\setminus\Sigma$ be a transversal smooth curve
for $(\phi_t)_{t\in\R}$ such that the boundary of $J$ consists of
two points situated  on an incoming and an outgoing separatrix
respectively, and the segment of each separatrix between the
corresponding saddle point and the corresponding boundary point of
$J$ contains no intersection with the interior of $J$. Let
$\gamma:[0,a]\to J$ stand for the induced parametrization, i.e.\
$\nu_{\mathscr{F}}(\gamma|_{[0,t]})=t$ for any $t\in[0,a]$, such
that  $\gamma(0)$ lies on an incoming separatrix and $\gamma(a)$
lies on an outgoing separatrix. From  now on we will identify the
curve $J \subset \Surf$ with the interval $[0,a)$ and, by abusing
the notation, we will denote by $I$ both the interval
$[0,a)\subset \R$  and  the curve $J$ on $\Surf$.

Denote by $T:\J\to \J$ the first-return map induced on $\J$.  In
the induced parametrization, $T:I\to I$ is an interval exchange
transformation and it preserves the measure induced by the
restriction of $\nu_{\mathscr{F}}$ to $I$, which coincides with
the Lebesgue measure $Leb$ on $I$. Moreover,
$T=T_{(\pi,\lambda)}$, where $\pi\in\mathcal{S}^0_{\mathcal{A}}$
for some finite set $\mathcal{A}$ and
$(\pi,\lambda)\in\mathcal{S}^0_{\mathcal{A}}\times\R_+^{\mathcal{A}}$
satisfies the Keane condition, because by assumption
$(\phi_t)_{t\in\R}$ has no saddle connections. Recall that
$l_\alpha$, $\alpha\in\mathcal{A}$ stand for the left end points
of the exchanged intervals.

\begin{lemma}\label{periodicimpliesperiodic}
If  $(\phi_t)_{t\in\mathbb{R}}$ is of hyperbolic periodic type
then  the IET $T$ can be chosen to be of hyperbolic periodic type.
\end{lemma}
\begin{proof}
Let $\Psi:\Surf\to \Surf$ be the diffeomorphism that fixes the
flow foliation $\mathscr{F}$ and rescales by $\rho<1$ the
transversal measure $\nu_\mathscr{F}$. Since $\Psi$ fixes $\Sigma$
(as a set) and sends leaves to leaves, replacing $\Psi$ by one of
its powers, we can assume that there exists a point $z_0\in\Sigma$
such that $\Psi(z_0)=z_0$ and all separatrixes emanating from
$z_0$ are fixed. Consider a transversal $\gamma: [0,a]\to \Surf$
such that $\gamma(0)=z_0$ and the endpoint $\gamma(a)$ is on an
outgoing separatrix. Up to modification of $\Psi$ by an isotopy
which leaves $(\mathscr{F}, \nu_{\mathscr{F}})$ invariant, one can
also assume that $\Psi(\gamma)\subset\gamma$ (see for example \S9
in \cite{FLP}). The first return map on $\gamma$ in the induced
parametrization, as seen above, gives an IET
$T=T_{(\pi,\lambda)}:I\to I$ with $I=[0,a)$. Moreover, as
$\Psi(\nu_\mathscr{F})=\rho\,\nu_\mathscr{F}$, we have
$\Psi(\gamma(x))=\gamma(\rho x)$ for every $x\in[0,a]$. Since
$\Psi(\mathscr{F})=\mathscr{F}$ and $\Psi(\gamma)\subset \gamma$,
$\gamma(\rho a)$ still belongs to an outgoing separatrix and
$[0,\rho a)$ is  \emph{admissible} in the sense defined by  Veech
in \S3 in \cite{Ve0}. This, as shown by Veech in \cite{Ve0},
implies that $[0,\rho a) = I^k $ for some $k\geq 1$ (recall that
$I^k$ is the $k^{th}$ inducing interval of Rauzy-Veech induction)
and that the first return map on $I^k=[0,\rho a)$ is
$\mathcal{R}^{k}(T)$.

Every discontinuity $l_\alpha$ of $T$ is such that
$\gamma(l_\alpha)$ is the first backward intersection of one of
the incoming separatrix with the interior of $\gamma$. Since
$\gamma(\rho\, l_\alpha) = \Psi (\gamma (l_\alpha)) $ and
$\Psi(\gamma)\subset \gamma$,   also $\gamma(\rho\, l_\alpha)$ is
the first backward intersection of an incoming separatrix with the
interior of $\Psi(\gamma)$.
%Furthermore, since $\Psi$ sends $\mathscr{F}$ to $\mathscr{F}$, $\Psi(\gamma) \subset \gamma$ and $$
%Since $\Psi$ rescales $\nu_\mathscr{F}$ and we are using the parametrization induced by ,,
This shows that the IET  induced by $T$ on $I^k=[0,\rho\,a)$
has datas $(\pi, \rho\,\lambda )$,  hence $\mathcal{R}^{k}(T) = T_{(\pi,
\rho\lambda)}$. This shows that $\Theta(\mathcal{R}^{k} T) =
\Theta( T)$  and thus $\Theta(\mathcal{R}^{n+k} T) =
\Theta(\mathcal{R}^{n}T)$ for $n\geq 0$. Let
$A=\Theta(\mathcal{R}^k T)$ be the period matrix. Since the orbit
of $T$ under $\mathcal{R}$ is obviously infinite, $A^m$ is a
positive matrix for some $m\geq 1$, by  Lemma in \S1.2.4 in
\cite{Ma-Mo-Yo}. It follows that replacing $\Psi$ by its $m$-th
iteration, we can assume $A$ is a positive matrix. Therefore $T$
is of period type.

Moreover, the action induced by $\Psi$ on $H_1(\Surf, \mathbb{R})$
is isomorphic to the action of $A$ on
$\R^{\mathcal{A}}/\ker\Omega_\pi$, and hence to the action of
$(A^t)^{-1}$ on $H_{\pi}$ (see \S2 and \S7 in \cite{ViB}). Thus,
the assumption that $(\phi_t)_{t\in\R}$ is of hyperbolic periodic
type is equivalent to $T$ being of hyperbolic periodic type.

Finally we want to choose a transversal $\gamma$ as in the
construction before Lemma \ref{periodicimpliesperiodic}, i.e.~such
that $\gamma([0,a])\subset \Surf\setminus\Sigma$ and $\gamma(0)$
lies on an incoming separatrix and $\gamma(a)$ lies on an outgoing
separatrix. One can obtain such a transversal by homotoping
$\gamma$ slightly along the leaves of $\mathscr{F}$ to a new
$\gamma'$ so that $\gamma'(0)$ now belongs to an incoming
separatrix for $z_0$. If the homotopy is small enough so that
$\Sigma$ is not hit, the first return on $\gamma'$ is still given
by the same IET $T$.
\end{proof}

Set $\underline{\alpha}=\pi_1^{-1}(1)\in\mathcal{A}$. Denote by
$\tau:I\to \R_+$ the first-return time map of the flow
$(\phi_t)_{t\in\R}$ to $\J$. This map is well defined and smooth on
the interior of  each interval $I_\alpha$, $\alpha\in\mathcal{A}$,
and $\tau$ has a singularity of logarithmic type at each point
$l_\alpha$, $\alpha\in\mathcal{A}$ (see \cite{Ko}) except for the
right-side of $l_{\underline{\alpha}}$; here the one-sided limit
of $\tau$ from the left exists\footnote{We remark that this is due
to our convention of choosing $\gamma(0)$ on an incoming separatrix
and $\gamma(a)$ on an outgoing one. If we had chosen $\gamma(0)$ on
 an outgoing separatrix and $\gamma(a)$ on an incoming one, the
 finite one-sided limit from the right would be at
 $l_{\overline{\alpha}}$ where
 $\overline{\alpha}=\pi_0^{-1}(1)\in\mathcal{A}$.}.
 The precise nature of these singularities is
analyzed in Theorem~\ref{theorem:formphi} below.

The considerations so far show that the flow $(\phi_t)_{t\in\R}$
on $(\Surf,\nu)$ is measure-theore\-tically isomorphic to the
special flow $T^\tau$. An isomorphism is established by the map
$\Gamma:I^\tau\to{ \Surf},\;\;\Gamma(x,s)=\phi_s\gamma(x)$.

\subsection{Extensions as special flows}\label{extensionsspecialflows}
Let us now consider an extension $(\Phi^f_t)_{t\in\R}$ given by a
$\mathscr{C}^{2+\epsilon}$-function $f:\Surf\to\R$.  Let us consider
its transversal submanifold $\J\times\R\subset \Surf\times\R$. Note
that every point $(\gamma(x),y)\in \gamma(\Int I_\alpha)\times\R$
returns to $\J\times\R$ and the return time is
$\widehat{\tau}(x,y)=\tau(x)$. Denote by
$\varphi_f:\bigcup_{\alpha\in\mathcal{A}}\Int I_\alpha\to\R$ the
$\mathscr{C}^{2+\epsilon}$-function
\begin{equation}\label{defvarphi}
\varphi_f(x)=F(\tau(x),\gamma(x))=\int_0^{\tau(x)}f(\phi_s\gamma(x))ds,\;\;\text{
for }\;\;x\in\bigcup_{\alpha\in\mathcal{A}}\Int
I_\alpha.
\end{equation}
Notice that
\begin{equation}\label{mean}
Leb(f)= \int_I\varphi_f(x)\,dx=\int_{ \Surf}f\,d\nu= \nu(f).
\end{equation}

Let us consider the skew product
$T_{\varphi_f}:(I\times\R,Leb\times
 \mR)\to (I\times\R,Leb\times  \mR)$,
$T_{\varphi_f}(x,y)=(Tx,y+\varphi_f(x))$ and the special flow
$(T_{\varphi_f})^{\widehat{\tau}}$ built over $T_{\varphi_f}$ and
under the roof function $\widehat{\tau}:I\times\R\to\R_+$ given by
$\widehat{\tau}(x,y)=\tau(x)$. Thus, by standard arguments, this
show the following.

\begin{lemma}\label{specrep}
The special flow $(T_{\varphi_f})^{\widehat{\tau}}$ is
measure-theoretically isomorphic to the flow $(\Phi^f_t)$ on $(
\Surf\times\R,\nu\times \mR)$.\bez
\end{lemma}
Recall that $\varphi_f$ is  $\mathscr{C}^{2+\epsilon}$ in the
interior of each interval $I_\alpha$, $\alpha\in\mathcal{A}$. The
following Proposition provides further properties of  the
singularities of $\varphi_f$ at the endpoints of $I_\alpha$,
$\alpha\in\mathcal{A}$ and their symmetry properties. Recall that
$\underline{\alpha}=\pi_1^{-1}(1)$ and set
$\overline{\alpha}=\pi_0^{-1}(d)$.

\begin{theorem}\label{theorem:formphi}
For every $\mathscr{C}^{2+\epsilon}$-function $f:\Surf\to\R$ there
exist $C_\alpha^{\pm}$,  $\alpha\in\mathcal{A}$, with
$C^+_{\underline{\alpha}}=C^-_{\overline{\alpha}}=0$,
 and
 $g\in \ac(\sqcup_{\alpha\in \mathcal{A}}
I_{\alpha})$ such that
\begin{align*}\varphi_f(x)=-\sum_{\alpha\in\mathcal{A}}
\left(C^+_\alpha\log\left(|I|\{(x-l_\alpha)/|I|\}\right)
+C^-_\alpha\log\left(|I|\{(r_\alpha-x)/|I|\}\right)\right)+g(x).
\end{align*}
Moreover, $\varphi_f\in\overline{\ls}(\sqcup_{\alpha\in
\mathcal{A}} I_{\alpha})$ and $g=g_1+g_2$ with $g_1,g_2\in
\ac(\sqcup_{\alpha\in \mathcal{A}} I_{\alpha})$ satisfying
$g_1'\in \ls(\sqcup_{\alpha\in \mathcal{A}} I_{\alpha})$
 and
$g_2'\in \ac(\sqcup_{\alpha\in \mathcal{A}} I_{\alpha})$. There
exists a constant $C>0$  such that
\begin{equation}\label{controlC}
C^{-1}\sum_{z\in\Sigma}|f(z)|\leq\bl(\varphi_f)\leq
C\sum_{z\in\Sigma}|f(z)|\quad\text{ and  }\quad\|g\|_{\bv}\leq
C\|f\|_{\mathscr{C}^{2}}
\end{equation} for every
$f\in\mathscr{C}^{2+\epsilon}(\Surf)$. In particular, the linear
operator
\[\mathscr{C}^{2+\epsilon}(\Surf)\ni f\mapsto \varphi_f\in \overline{\ls}(\sqcup_{\alpha\in
\mathcal{A}} I_{\alpha})\] is bounded.
\end{theorem}
The proof of this Theorem is presented in
Appendix~\ref{singularities}. In Appendix~\ref{singularities} we
also prove the following Proposition:
\begin{proposition}\label{vanishO:prop}
If $f(z)=0$ for each $z\in\Sigma$ then
$\varphi_f\in\ac(\sqcup_{\alpha\in \mathcal{A}} I_{\alpha})$ and
$\mathcal{O}(\varphi_f)=0$ for every $\mathcal{O}\in\Sigma(\pi)$.
\end{proposition}

The following proposition is also needed  to complete the proof of
Theorem~\ref{reductionthm} and will be used also in the proof of
Theorem~\ref{mainHamiltonian}.

\begin{proposition}\label{theorem:coblog}\label{corcohred}
Assume that  $T$ is of periodic type. Then every
$\varphi\in\ac_0(\sqcup_{\alpha\in \mathcal{A}} I_{\alpha})$ with
$\varphi'\in\overline{\ls}(\sqcup_{\alpha\in \mathcal{A}}
I_{\alpha})$  is cohomologous (via a continuous transfer function)
to a cocycle  $\psi\in\pl_0(\sqcup_{\alpha\in \mathcal{A}}
I_{\alpha})$ with $s(\psi)=s(\varphi)$. In particular, if
additionally  $s(\varphi)=0$ then $\varphi$ is cohomologous (via a
continuous transfer function) to $h\in \Gamma^{(0)}_0$.
\end{proposition}
The proof of this Proposition is given in  Appendix
\ref{cohreduction:sec}. Collecting together all these statements,
we get the proof of Theorem  \ref{reductionthm}.
\begin{proofof}{Theorem}{reductionthm}
The first part  of the Theorem~\ref{reductionthm} follows by
combining Lemma~\ref{specrep} and Theorem~\ref{theorem:formphi}
and  the second part using also
Lemma~\ref{periodicimpliesperiodic} and
Proposition~\ref{theorem:coblog} and recalling that special flows
with cohomologous roof functions are measure-theoretically
isomorphic.
\end{proofof}

\begin{comment}
Following \cite{Co-Fr} for $\varphi\in\bv(\sqcup_{\alpha\in
\mathcal{A}} I_{\alpha})$ and $\mathcal{O}\in\Sigma(\pi)$ let
\[\mathcal{O}(\varphi)=\sum_{\alpha\in\mathcal{A},\pi_0(\alpha)\in\mathcal{O}}
\varphi_-(r_{\alpha})-
\sum_{\alpha\in\mathcal{A},\pi_0(\alpha)-1\in\mathcal{O}}\varphi_+(l_{\alpha}).\]
Recall that if $h\in\Gamma$ then $h\in H_\pi$ if and only if
$\mathcal{O}(h)=0$ for every $\mathcal{O}\in\Sigma(\pi)$.
Moreover, for each $\varphi\in\ac(\sqcup_{\alpha\in \mathcal{A}}
I_{\alpha})$ and $k\geq 1$ we have
\begin{equation}\label{zbibigo}
\mathcal{O}(S(k)\varphi)=\mathcal{O}(\varphi)\text{ and
}\|\mathcal{O}(\varphi)\|\leq 2d\|\varphi\|_{\sup}.
\end{equation}
\end{comment}

%\begin{theorem}\label{theorem:flowreduc}
%Let $(\phi_t)_{t\in\R}$ multivalued Hamiltonian flow of hyperbolic periodic
%type on $M$ ($g>1$). If $f:M\to\R$ is a
%$\mathscr{C}^2$-function such that $\mathfrak{H}(f)=0$ and
%$\mathfrak{L}(f)= 0$ then the flow $(\Phi^f_t)_{t\in\R}$ on
%$M\times\R$ is reducible.
%\end{theorem}

\subsection{The dichotomy for extensions.}\label{final:sec}
In this section we prove Theorem~\ref{mainHamiltonian}. We will
use the following Lemma which exploits the special flow
representation in \S\ref{extensionsspecialflows}.
\begin{lemma}\label{corollary:ergcob}
The flow $(\Phi^f_t)_{t\in\R}$ is ergodic if and only if the skew
product $T_{\varphi_f}$ is ergodic. For every\footnote{This Lemma
holds more generally for any $f\in\mathscr{C}^{1}(\Surf,\Sigma)$,
even if we need it only for
$f\in\mathscr{C}^{2+\epsilon}(\Surf,\Sigma)$.}
$f\in\mathscr{C}^{2+\epsilon}(\Surf,\Sigma)$ the flow
$(\Phi^f_t)_{t\in\R}$ is reducible if and only if $\varphi_f$ is a
coboundary with a continuous transfer function.
\end{lemma}
The proof is standard apart from the continuity of the transfer
function. We include it for completeness in  Appendix
\ref{reduction}.

\begin{proofof}{Theorem}{mainHamiltonian}
Let $(\phi_t)_{t\in\R}$ be a locally Hamiltonian flow of
hyperbolic periodic type on $\Surf$.
% ($g>1$)\marginpar{do we need $g>1$?}.
%Denote by $\mathscr{C}^{2+\epsilon}(\Surf, \R)$ (respectively
%$\mathscr{C}^{2+\epsilon}(\Surf,\Sigma,\R)$) the subspace of
%$\mathscr{C}^{2+\epsilon}$-functions $f:\Surf\to\R$ (respectively,
%vanishing on $\Sigma$).
%Recall that, for each $f\in \mathscr{C}^{2+\epsilon}(\Surf)$,  the
%extension $(\phi_t)_{t\in\R}$ is measure-theoretically isomorphic
%to a special flow $T^{\varphi_f}$ over an IET $T$ is of hyperbolic
%periodic type by Lemma~\ref{specialflow} and Lemma~\ref{periodic}.
%Moreover, in view of Theorem~\ref{theorem:formphi}
%\marginpar{add $LG_0$ in Prop} and
%Proposition~\ref{theorem:coblog},
%$\varphi_f\in\overline{\ls}(\sqcup_{\alpha\in \mathcal{A}}
%I_{\alpha})$ and it is cohomologous to a function
%$\widetilde{\varphi}_f\in{\ls}(\sqcup_{\alpha\in \mathcal{A}}
%I_{\alpha})$.
Let us split the proof in several steps.
\subsubsection*{Definition of the space $K$}
Let us first define a bounded linear  operator on
$\mathscr{C}^{2+\epsilon}(\Surf)$, and then use it to define $K$
as its kernel. Let $\nu(f):=\int_{\Surf}f\,d\nu$ and
$f_0:=f-\nu(f)$. By Theorem~\ref{theorem:formphi} the extension
$(\Phi^f_t)_{t\in\R}$ is measure-theoretically isomorphic  to a
special flow  built over the skew product $T_{\varphi_f}$ with
$\varphi_f \in \overline{\ls}(\sqcup_{\alpha\in \mathcal{A}}
I_{\alpha})$.  In view of (\ref{mean}), $Leb(\varphi_{f_0})
=\nu(f_0)=0$, so
$\varphi_{f_0}\in\overline{\ls}_0(\sqcup_{\alpha\in \mathcal{A}}
I_{\alpha})$.
% and $\widetilde{\varphi}_{f_0}\in{\ls}_0(\sqcup_{\alpha\in \mathcal{A}} I_{\alpha})$.
Consider the operator
$\mathfrak{h}:\overline{\ls}_0(\sqcup_{\alpha\in \mathcal{A}}
I_{\alpha}) \rightarrow \Gamma $ given by Theorem
\ref{operatorcorrection}. Let $\kappa=\#\Sigma=2(g-1)$ and let
\begin{equation*}
\mathfrak{H}:\mathscr{C}^{2+\epsilon}(\Surf)\to\R\times  \Gamma
%(\Gamma_u\cap\Gamma_0)\simeq\R^g
\quad \text{and}\quad
\mathfrak{L}:\mathscr{C}^{2+\epsilon}(\Surf)\to \R^{\kappa}
\end{equation*}
stand for the operators
\[\mathfrak{H}(f)=(\nu(f),\mathfrak{h}(\varphi_{f_0})),
\qquad\mathfrak{L}(f)=(f_0(z))_{z\in\Sigma}.\]
Since the operators $f\mapsto\nu(f)$, $f\mapsto\varphi_f$
(by Theorem~\ref{theorem:formphi}) and $\mathfrak{h}$ (by Theorem
\ref{operatorcorrection}) are linear and bounded, $\mathfrak{H}$
is a bounded linear operator as well. This shows that the kernel
$K$ of $\mathfrak{H}$ is a closed space. Moreover, the image of
$\mathfrak{H}$ has dimension $g$ since by
Theorem~\ref{operatorcorrection} the image of $\mathfrak{h}$ has
dimension $g-1$. Thus, $K$ has codimension $g$.

\subsubsection*{Invariance of $K$.}
Let us show that the operator $\mathfrak{H}$ is
$(\phi_t)_{t\in\R}$-invariant, i.e.\
$\mathfrak{H}(f\circ\phi_t)=\mathfrak{H}(f)$ for every $t\in\R$.
Since $\phi_t$ preserves $\nu$, we get $\nu(f\circ\phi_t)=\nu(f)$,
so it suffices to prove that
${\mathfrak{h}}(\varphi_{f\circ\phi_t})={\mathfrak{h}}(\varphi_{f})$
for each $t\in\R$ and $f \in \overline{\ls}_0(\sqcup_{\alpha\in \mathcal{A}}
I_{\alpha})$. Note that
\begin{align*}
\varphi_{f\circ\phi_t}(x)&=\int_0^{\tau(x)}f(\phi_{t+s}\gamma(x))ds=
\int_t^{t+\tau(x)}f(\phi_{s}\gamma(x))ds\\&=
\int_0^{\tau(x)}f(\phi_{s}\gamma(x))ds-\int_0^{t}f(\phi_{s}\gamma(x))ds+
\int_{\tau(x)}^{t+\tau(x)}f(\phi_{s}\gamma(x))ds.
\end{align*}
Let us consider the $\mathscr{C}^2$-function $\xi:I\to\R$,
$\xi(x)=\int_0^{t}f(\phi_{s}\gamma(x))ds$ and observe that
\[\int_{\tau(x)}^{t+\tau(x)}f(\phi_{s}\gamma(x))ds=
\int_{0}^{t}f(\phi_{s}\circ\phi_{\tau(x)}\gamma(x))ds=
\int_{0}^{t}f(\phi_{s}\gamma(Tx))ds=\xi(Tx),\] so
$\varphi_{f\circ\phi_t}=\varphi_{f}+\xi\circ T-\xi$ and
$\varphi_{f-f\circ\phi_t}=\xi-\xi\circ T$. As
$(f\circ\phi_t-f)(z)=0$ for each $z\in\Sigma$, by Proposition \ref{vanishO:prop},
%Theorem~\ref{theorem:formphi},
$\varphi_{f-f\circ\phi_t}\in\ac_0(\sqcup_{\alpha\in \mathcal{A}}
I_{\alpha})$. Since we showed that $\varphi_{f\circ\phi_t-f}$ is a
coboundary, Lemma~\ref{theorem:invop} implies that
${\mathfrak{h}}(\varphi_{f\circ\phi_t-f})=0$. Thus, by linearity,
${\mathfrak{h}}(\varphi_{f\circ\phi_t})={\mathfrak{h}}(\varphi_{f})$,
which completes the proof of invariance of $\mathfrak{H}$.  In
particular, it follows that the kernel $K$ is
$(\phi_t)_{t\in\R}$-invariant.

%\begin{corollary}
%There exists a linear map
%$\mathfrak{U}:\R^g\to\mathscr{H}\subset\mathscr{C}^{2+\epsilon}(M,\Sigma,\R)$
%such that for every $f\in\mathscr{C}^{2+\epsilon}(M,\R)$ with
%$\mathfrak{L}(f)\neq 0$ the extended flow corresponding to the
%corrected function $f-\mathfrak{U}(\mathfrak{H}f)$ is ergodic.
%\end{corollary}

\subsubsection*{Step 3: Ergodicity.}
We need to prove that if $f\in K\subset \mathscr{C}^{2+\epsilon}$
and $\sum_{z\in\Sigma}|f_0(z)|\neq 0$, then the flow $(\Phi^f_t)_{t\in\R}$
on $\Surf\times\R$ is ergodic. Since $f\in K$, we know that
$\mathfrak{H}(f)=0$. In particular we have
$Leb(\varphi_f)=\nu(f)=0$, $\mathfrak{h}(\varphi_f)={0}$ and since
$f=f_0$,  $\| \mathfrak{L}(f) \|  = \sum_{z\in\Sigma}|f_0(z)| \neq 0$.
%Thus $\mathfrak{h}(\varphi_f)=\underline{0}$.
%As $\mathfrak{L}(f) \neq 0$ and $f_0=f$, we have $\sum_{z\in\Sigma}|f(z)|=\mathfrak{L}(f)\neq 0$.
By Lemma~\ref{corollary:ergcob}, it suffices to show the skew product
$T_{\varphi_f}:I\times\R\to I\times\R$ is ergodic.

In view of Theorem~\ref{theorem:formphi}, the function
${\varphi}_f\in\overline{{\ls}}_0(\sqcup_{\alpha\in \mathcal{A}}
I_{\alpha})$ can be decomposed as $(\varphi_f - g_1) + g_1$ where
we can choose $g_1 \in AC_0(\sqcup_{\alpha\in \mathcal{A}}
I_{\alpha}) $ and  $\varphi_f - g_1 \in  {\ls}_0(\sqcup_{\alpha\in
\mathcal{A}} I_{\alpha})$, while $g_1' \in
{{\ls}}(\sqcup_{\alpha\in \mathcal{A}} I_{\alpha})$. By
Proposition~\ref{theorem:coblog}, $g_1 $ is cohomologous via a
continuous transfer function to a function in
$\pl_0(\sqcup_{\alpha\in \mathcal{A}} I_{\alpha})$, which is in
particular $BV^1$. Thus,  $\varphi_f$ can be decomposed as
$\widetilde{\varphi}_f+g$ with
$\widetilde{\varphi}_f\in{\ls}_0(\sqcup_{\alpha\in \mathcal{A}}
I_{\alpha})$ and $g\in{\ac}_0(\sqcup_{\alpha\in \mathcal{A}}
I_{\alpha})$ is a coboundary.  Next, by Lemma~\ref{theorem:invop},
${\mathfrak{h}}(g)=0$, so
$\mathfrak{h}(\widetilde{\varphi}_f)=\mathfrak{h}({\varphi}_f)=0$.
Since by (\ref{controlC}) we have
$\bl(\widetilde{\varphi}_f)=\bl({\varphi}_f)\geq
\|\mathfrak{L}(f)\|/C >0$, the skew product
$T_{\widetilde{\varphi}_f}$ is ergodic  by
Theorem~\ref{theorem:ergmain}. Since $\widetilde{\varphi}_f$ and
$\varphi_f$ are cohomologous, $T_{\widetilde{\varphi}_f}$ and
$T_{{\varphi}_f}$ are metrically isomorphic, so also
$T_{{\varphi}_f}$ is ergodic.  This completes the proof of the
first  case  of the dichotomy.

\subsubsection*{Step 4: Reducibility.}
Let us now prove that if $f\in K$ and $\sum_{z \in \Sigma}
|f_0(z)|= 0$ then the flow $(\Phi^f_t)_{t\in\R}$ on
$\Surf\times\R$ is reducible.  Since $f \in K$, $\nu(f)=0$ and
$f=f_0$, so  from (\ref{mean}) we have $Leb(\varphi_f)=0$ and from
(\ref{controlC}) we have  $\bl(\varphi_f)= 0$. It follows from
Theorem~\ref{theorem:formphi} that $\varphi_f\in\ac_0$ and
$\varphi'_f\in\overline{\ls}$. Moreover,
 Proposition~\ref{vanishO:prop} also gives that
$\mathcal{O}(\varphi_f)=0$ for each $\mathcal{O}\in\Sigma(\pi)$.
Summing over $\mathcal{O}\in\Sigma(\pi)$, by \eqref{sprzezo}, this
shows that $s(\varphi_f)=0$. Moreover, since by assumption $f\in
K$, ${\mathfrak{h}}(\varphi_f)=0$. Let us show that this implies
that $\varphi_f$ is a coboundary with a continuous transfer
function.

By Proposition~\ref{corcohred}, there exist $h\in\Gamma_0$ such
that $\varphi_f-h$ is a coboundary with a continuous transfer
function, that is  $\varphi_f-h=g-g\circ T$ and $g:I\to\R$ is continuous.
Let us show that then $\mathcal{O}(\varphi_f-h)=0$ for every
$\mathcal{O}\in\Sigma(\pi)$.
It is proved in \cite{Co-Fr} that for
each $\varphi\in\ac(\sqcup_{\alpha\in \mathcal{A}} I_{\alpha})$
and $k\geq 1$ we have
%\begin{equation*}%\label{zbibigo}
$\mathcal{O}(S(k)\varphi)=\mathcal{O}(\varphi)$ %\text{ and}
 and $|\mathcal{O}(\varphi)|\leq 2d\|\varphi\|_{\sup}$. %\end{equation*}
Thus,
%Indeed, by (\ref{zbibigo}),
\[
\begin{split}
|\mathcal{O}(\varphi_f-h)|& =|\mathcal{O}(S(k)(\varphi_f-h)|\leq
2d \ \|S(k)(\varphi_f-h)\|_{\sup} \\ & \leq 2d \
\sup_{\alpha\in\mathcal{A}} \sup_{x \in I^{(k)}_\alpha}\{|g(x)-g(
T^{Q_\alpha (k)} x)\}  \leq 2d\sup_{x,x'\in
I^{(k)}}\{|g(x)-g(x')|\}
\end{split}
\]
and the latter supremum tends to zero as $k\to\infty$, hence
$\mathcal{O}(\varphi_f-h)=0$. It follows that
$\mathcal{O}(h)=\mathcal{O}(\varphi_f)=0$ for every
$\mathcal{O}\in\Sigma(\pi)$, and hence $h\in H_\pi$ by Remark
\ref{Ohzeroreformulation}. Moreover, since $\varphi_f-h$ is a
coboundary, by Lemma~\ref{theorem:invop},
${\mathfrak{h}}(\varphi_f-h)=0$ and since
$\mathfrak{h}(\varphi_f)=0$ (because $f\in K$), this gives by
linearity that also ${\mathfrak{h}}(h)=0$. By
Proposition~\ref{corcohred}, $h$ is a coboundary with a continuous
transfer function as well. Therefore $\varphi_f=(\varphi_f-h)+h$
is a sum of coboundaries with continuous transfer functions. By
Lemma~\ref{corollary:ergcob}, this implies that the reducibility
of $(\Phi_t^f)_{t\in\R}$.

\subsubsection*{Step 5: Decomposition.}
It was proved in \cite{Co-Fr}, for every $h\in H_\pi$ there exists
a function $f\in\mathscr{C}^{2+\epsilon}(M,\Sigma)$ with
$\varphi_f=h$ (see Lemma~7.4 in \cite{Co-Fr}). Since
${\mathfrak{h}}(h)=h$ for each $h\in\Gamma_u\cap\Gamma_0\subset
H_\pi$, it follows that for every $v\in\R$ and
$h\in\Gamma_u\cap\Gamma_0$ there exists
$f\in\mathscr{C}^{2+\epsilon}(M,\Sigma)$ such that $\nu(f)=v$
and ${\mathfrak{h}}(\varphi_{f_0})={\mathfrak{h}}(h)=h$, hence
$\mathfrak{H}(f)=(\nu(f),{\mathfrak{h}}(\varphi_{f_0}))=(v,h).$
Therefore, there exists a $g$-dimensional subspace
$\mathscr{H}\subset\mathscr{C}^{2+\epsilon}(M,\Sigma)$ such that
$\mathfrak{H}:\mathscr{H}\to\R^g$ is a linear isomorphism. Given
$f \in\mathscr{C}^{2+\epsilon}(M,\Sigma)$, let $f_{\Sigma}\in
\mathscr{H}$ be the preimage of $\mathfrak{H}(f)$ by this
isomorphism.  Then if we set $f_K := f-f_{\Sigma}$ then
$\mathfrak{H}(f_K)=\mathfrak{H}(f)-\mathfrak{H}(f_\Sigma)=0$,
i.e.~$f \in K$. This gives the claimed decomposition $f=f_K+
f_\Sigma$ and concludes the proof.
\end{proofof}

\appendix
\section{Proof of Proposition~\ref{lemlos}.}
\label{lemlosproof}
In this section we give the proof of
Proposition~\ref{lemlos}. For a compactly absolutely continuous
function $\varphi:I\setminus End(T)\to\R$, this is absolutely
continuous on each compact subset of its domain, set
\[
los(\varphi)=\esssup\left\{\min_{\bar{x}\in
End(T)}|\varphi'(x)(x-\bar{x})|:x\in I\setminus End(T)\right\}.
\]
Of course, every function $\varphi\in\ol(\sqcup_{\alpha\in
\mathcal{A}} I_{\alpha})$ is compactly absolutely continuous and
\begin{equation}\label{relphic}
 los(\varphi)\leq
\bl(\varphi)+|I|\|g_\varphi'\|_{\sup}\quad\text{ and
}\quad\bl(\varphi)\leq 2d\,los(\varphi).
\end{equation}

\begin{lemma}\label{lemlogosc1}
Let $f:(x_0,x_1]\to\R$ be a compactly absolutely continuous
function such that $|f'(x)(x-x_0)|\leq C$ for a.e.\
$x\in(x_0,x_1]$. For every $J=[a,b]\subset[x_0,x_1]$ we have
\[|m(f,J)-f(b)|\leq
2C\text{ and }\frac{|f(b)-f(a)|}{b-a}\leq \frac{C}{a-x_0}\text{ if
}a>x_0.\]
\end{lemma}
\begin{proof}
If $a>x_0$ then using integration by parts we get
\begin{equation*}
\int_a^b(f(x)-f(b))\,dx=(a-x_0)(f(b)-f(a))-\int_a^b(x-x_0)f'(x)\,dx.
\end{equation*}
Moreover, by assumption,
%\begin{equation*}
$\left|\int_a^b(x-x_0)f'(x)\,dx\right|\leq\int_a^b|(x-x_0)f'(x)|\,dx\leq
C|J|$.
%\end{equation*}
Furthermore,
\begin{align*}
|f(b)-f(a)|&=\left|\int_a^bf'(x)\,dx\right|\leq
\int_a^b\frac{C}{x-x_0}\,dx=C\log\frac{b-x_0}{a-x_0}\\&=
C\log\left(1+\frac{b-a}{a-x_0}\right)\leq
C\frac{b-a}{a-x_0}=\frac{C|J|}{a-x_0}.
\end{align*}
It follows that
\begin{equation*}
\left|\frac{1}{b-a}\int_a^bf(x)\,dx-f(b)\right|=
\frac{1}{|J|}\left|\int_a^b(f(x)-f(b))\,dx\right|
\leq 2C. %.
\end{equation*}
Letting $a\to x_0$, we also have $|m(f,J)-f(b)|\leq C$ if
$J=[x_0,b]$.
\end{proof}

\begin{lemma}\label{lemlos0}
Let $\varphi\in\ol(\sqcup_{\alpha\in \mathcal{A}} I_{\alpha})$ and
$J\subset I_\alpha$ for some $\alpha\in\mathcal{A}$. Then
\begin{eqnarray}\label{meanv01}
&&|m(\varphi,J)-m(\varphi,I_\alpha)|\leq
los(\varphi)\left(4+\frac{|I_\alpha|}{|J|}\right);
\\ && \label{meanv02}
\frac{1}{|J|}\int_J|\varphi(x)-m(\varphi,J)|\,dx\leq
8los(\varphi).
\end{eqnarray}
\end{lemma}

\begin{proof}
Let $I_\alpha=[x_0,x_2]$ and $x_1=(x_0+x_2)/2$. Suppose that
$J=[a,b]\subset[x_0,x_1]$. In view of Lemma~\ref{lemlogosc1},
\begin{equation}\label{rozme}
|m(\varphi,J)-\varphi(b)|\leq
2los(\varphi),\;\;|m(\varphi,[x_0,x_1])-\varphi(x_1)|\leq
2los(\varphi)
\end{equation} and
\[|\varphi(x_1)-\varphi(b)|\leq
los(\varphi)\frac{x_1-b}{b-x_0}\leq los(\varphi)\frac{x_1-x_0}{b-a}=
\frac{los(\varphi)}{2}\frac{|I_\alpha|}{|J|}.\] Applying
Lemma~\ref{lemlogosc1} to $\varphi:[x_1,x_2)\to\R$ we also have
\[|m(\varphi,[x_1,x_2])-\varphi(x_1)|\leq 2los(\varphi).\]
Since
$m(\varphi,[x_0,x_2])=(m(\varphi,[x_0,x_1])+m(\varphi,[x_1,x_2]))/2$,
it follows that
\[|m(\varphi,I_\alpha)-\varphi(x_1)|\leq 2los(\varphi).\]
Therefore
\begin{equation}\label{meanv3}
|m(\varphi,J)-m(\varphi,I_\alpha)|\leq
4los(\varphi)+\frac{los(\varphi)}{2}\frac{|I_\alpha|}{|J|}.
\end{equation}
Let us consider the function $\bar{\varphi}:(x_0,x_1]\to\R$,
$\bar{\varphi}(x)=|\varphi(x)-m(\varphi,J)|$. The function
$\bar{\varphi}$ is compactly absolutely continuous with
$|\bar{\varphi}'(x)|\leq|\varphi'(x)|$ almost everywhere, hence
$los(\bar{\varphi})\leq los(\varphi)$. Therefore, by
Lemma~\ref{lemlogosc1},
\begin{align*}
\frac{1}{|J|}&\int_J|\varphi(x)-m(\varphi,J)|\,dx=m(\bar{\varphi},J)\leq
|m(\bar{\varphi},J)-\bar{\varphi}(b)|+|\bar{\varphi}(b)|\\
&=|m(\bar{\varphi},J)-\bar{\varphi}(b)|+|\varphi(b)-m(\varphi,J)|\leq
2los(\bar{\varphi})+2los(\varphi),
\end{align*}
hence
\begin{equation}\label{meanv4}
\frac{1}{|J|}\int_J|\varphi(x)-m(\varphi,J)|\,dx\leq
4los(\varphi).
\end{equation}
By symmetric arguments, (\ref{meanv3}), (\ref{meanv4}) and
\begin{equation}\label{rozmea}
|m(\varphi,J)-\varphi(a)|\leq 2los(\varphi)
\end{equation} hold
when $J\subset[x_1,x_2]$. If $x_1\in(a,b)$ then we can split $J$
into two intervals $J_1=[a,x_1]$ and $J_2=[x_1,b]$ for which
(\ref{meanv3}) and (\ref{meanv4}) hold. Since
\begin{equation}\label{meanmean}
m(\varphi,J)=\frac{|J_1|}{|J|}m(\varphi,J_1)+\frac{|J_2|}{|J|}m(\varphi,J_2),
\end{equation}
it follows that
\begin{align*}
|m(\varphi,J)-m(\varphi,I_\alpha)|&\leq los(\varphi)
\left(\frac{|J_1|}{|J|}\left(4+\frac{|I_\alpha|}{2|J_1|}\right)+
\frac{|J_2|}{|J|}\left(4+\frac{|I_\alpha|}{2|J_2|}\right)\right)\\&=
los(\varphi)\left(4+\frac{|I_\alpha|}{|J|}\right).
\end{align*}
By (\ref{rozme}) and (\ref{rozmea}),
$|m(\varphi,J_1)-\varphi(x_1)|\leq 2los(\varphi)\text{ and
}|m(\varphi,J_2)-\varphi(x_1)|\leq 2los(\varphi)$. Moreover, by
(\ref{meanmean}), $|m(\varphi,J)-\varphi(x_1)|\leq 2los(\varphi)$,
hence
\[
|m(\varphi,J_1)-m(\varphi,J)|\leq 4los(\varphi)\text{ and
}|m(\varphi,J_2)-m(\varphi,J)|\leq 4los(\varphi).
\]
In view of (\ref{meanv4}) applied to $J_1$ and $J_2$, it follows
that
\[
\frac{1}{|J_1|}\int_{J_1}|\varphi(x)-m(\varphi,J)|\,dx\leq 8
los(\varphi)\text{ and
}\frac{1}{|J_2|}\int_{J_2}|\varphi(x)-m(\varphi,J)|\,dx\leq 8
los(\varphi),
\]
and hence $\frac{1}{|J|}\int_{J}|\varphi(x)-m(\varphi,J)|\,dx\leq
8 los(\varphi).$
\end{proof}

\begin{proofof}{Proposition}{lemlos}
First note that if $g\in \bv(\sqcup_{\alpha\in \mathcal{A}}
I_{\alpha})$ then
\begin{equation}\label{meanvar}
|g(x)-m(g,J)|\leq\var g\text{ for each }x\in
I_\alpha.
\end{equation}
Let $\varphi=\varphi_0+g_\varphi$ be the decomposition of the form
(\ref{fform}). Since $\bl(\varphi_0)=\bl(\varphi)$ and
$g_{\varphi_0}=0$, by (\ref{meanv01}), (\ref{meanv02}) and
(\ref{relphic}), we have
\[|m(\varphi_0,J)-m(\varphi_0,I_\alpha)|\leq
\bl(\varphi)\left(4+\frac{|I_\alpha|}{|J|}\right),\;\;
\frac{1}{|J|}\int_J|\varphi_0(x)-m(\varphi_0,J)|\,dx\leq
8\bl(\varphi).\] Moreover, in view of (\ref{meanvar}),
\[|m(g_\varphi,J)-m(g_\varphi,I_\alpha)|\leq
\var g_\varphi,\;\;
\frac{1}{|J|}\int_J|g_\varphi(x)-m(g_\varphi,J)|\,dx\leq\var
g_\varphi.\] Combining these inequalities completes the proof.
\end{proofof}

\section{Singularities of extensions}\label{singularities}
In this Appendix  we prove Theorem~\ref{theorem:formphi} and
Proposition \ref{vanishO:prop}. The following Lemma will be used
in the proof.
\begin{lemma}\label{lematprzej}
Let $g:[-1,1]\times[-1,1]\to\R$ be a
$\mathscr{C}^{2+\epsilon}$-function. Then the function
$\xi:=\xi^g:(0,1]\to\R$,
\[\xi^g(s)=
\int_{s}^{1}g\left(u,\frac{s}{u}\right)\frac{1}{u}du\] is of the
form
\[
\xi(s)=-g(0,0)\log s+\widetilde{\xi}(s)\text{ with
}\widetilde{\xi}(s)=-g_{xy}(0,0)s\log s+\xi_0(s),
\]
where $\xi_0:[0,1]\to\R$ is an absolutely continuous function
whose derivative is absolutely continuous and
$\|\widetilde{\xi}\|_{\bv}\leq C\|g\|_{\mathscr{C}^{2}}$. If
additionally $g(0,0)=0$, then
\begin{equation}
\label{granicaxi}
\lim_{s\to
0^+}\xi(s)=
\int_{0}^{1}\left(g\left(u,0\right)+g\left(0,u\right)\right)\frac{1}{u}du.
\end{equation}
\end{lemma}

\begin{proof}
First note that
\begin{equation}\label{formxi}
\xi(s)=\int_{\sqrt{s}}^{1}g\left(u,\frac{s}{u}\right)\frac{1}{u}du+
\int_{\sqrt{s}}^{1}g\left(\frac{s}{u},u\right)\frac{1}{u}du.
\end{equation}
Thus
\[\xi'(s)=\int_{\sqrt{s}}^{1}\frac{g_x\left(\frac{s}{u},u\right)+
g_y\left(u,\frac{s}{u}\right)}{u^2}du
-\frac{g(\sqrt{s},\sqrt{s})}{ s}\] and
\[\xi''(s)=\int_{\sqrt{s}}^{1}
\frac{g_{xx}\left(\frac{s}{u},u\right)+g_{yy}\left(u,\frac{s}{u}\right)}{u^3}du-
\frac{g_x(\sqrt{s},\sqrt{s})+g_y(\sqrt{s},\sqrt{s})}{s\sqrt{s}}
+\frac{g(\sqrt{s},\sqrt{s})}{s^2}.
\]
First suppose that $g(0,0)=0$, $g'(0,0)=0$ and $g''(0,0)=0$. Then
\begin{eqnarray*}
|g(x,y)|&\leq &
\min\left(\|g\|_{\mathscr{C}^2}(|x|^2+|y|^2),
\|g\|_{\mathscr{C}^{2+\epsilon}}(|x|^{2+\epsilon}+|y|^{2+\epsilon})\right),
\\
\|g'(x,y)\|&\leq&
\min\left(\|g\|_{\mathscr{C}^2}(|x|+|y|),
\|g\|_{\mathscr{C}^{2+\epsilon}}(|x|^{1+\epsilon}+|y|^{1+\epsilon})\right),
\\
\|g''(x,y)\| &\leq&
\|g\|_{\mathscr{C}^{2+\epsilon}}(|x|^{\epsilon}+|y|^{\epsilon}).
\end{eqnarray*}
It follows that
\[|\xi(s)|\leq 3\|g\|_{\mathscr{C}^2}, \;\;|\xi'(s)|\leq \|g\|_{\mathscr{C}^2}(3-2\log s)
\text{ and }|\xi''(s)|\leq
\frac{8\|g\|_{\mathscr{C}^{2+\epsilon}}}{s^{1-\epsilon/2}}.\]
Since $\xi'$ and $\xi''$ are integrable on $[0,1]$, $\xi$ and
$\xi'$ are absolutely continuous. Moreover,
\[\|\xi\|_{\bv}=\|\xi\|_{\sup}+\int_0^1|\xi'(s)|\,ds\leq 8\|g\|_{\mathscr{C}^2}.\]

For an arbitrary $g$  we use the following decomposition
\begin{align*}
g(x,y)=&\;g(0,0)+g_x(0,0)x+g_y(0,0)y\\&+\frac{1}{2}g_{xx}(0,0)x^2+
g_{xy}(0,0)xy+\frac{1}{2}g_{yy}(0,0)y^2+g_0(x,y).
\end{align*}
Then $g_0$ is a $\mathscr{C}^{2+\epsilon}$-function such that
$g_0$, $g_0'$ and $g''_0$ vanish at $(0,0)$ and
$\|g_0\|_{\mathscr{C}^2}\leq 5\|g\|_{\mathscr{C}^2}$. As we have
already proven, the function $\xi^{g_0}$ and its derivative are
absolutely continuous and $\|\xi^{g_0}\|_{\bv}\leq
8\|g_0\|_{\mathscr{C}^2}$. By straightforward computation, we also
have
\[
\xi^1(s)=-\log s,\;\xi^x(s)=\xi^y(s)=1-s,\;
\xi^{x^2}(s)=\xi^{y^2}(s)=\frac{1-s^2}{2},\;\xi^{xy}(s)=-s\log s.
\]
Hence
\begin{align*}
\xi(s)=&-g(0,0)\log
s+\left(g_x(0,0)+g_y(0,0)\right)(1-s)-g_{xy}(0,0)s\log s\\&+
\left(g_{xx}(0,0)+g_{yy}(0,0)\right)\frac{1-s^2}{4}+\xi^{g_0}(s).
\end{align*}
It follows that $\xi_0$ and its derivative are absolutely
continuous and
\[
\|\widetilde{\xi}\|_{\bv}\leq
2\|g\|_{\mathscr{C}^2}+\|\xi^{g_0}\|_{\bv}\leq
42\|g\|_{\mathscr{C}^2}.
\]

Assume additionally that $g(0,0)=0$. Since $g$ is Lipschitz
continuous with Lipschitz constant $\|g\|_{\mathscr{C}^1}$, we
have
\begin{eqnarray*}
\lefteqn{\left|\int_0^{1}g(u,0)\frac{1}{u}du-\int_{\sqrt{s}}^1g(u,s/u)\frac{1}{u}du\right|}\\\
&\leq&\int_0^{\sqrt{s}}\left|g(u,0)-g(0,0)\right|\frac{1}{u}du+
\int_{\sqrt{s}}^1\left|g(u,0)-g(u,s/u)\right|\frac{1}{u}du
\\&\leq&\|g\|_{\mathscr{C}^1}\left(\int_0^{\sqrt{s}}du+
\int_{\sqrt{s}}^1\frac{s}{u^2}du\right)=
\|g\|_{\mathscr{C}^1}\left(2{\sqrt{s}}-s\right)\to 0
\end{eqnarray*}
as $s\to 0$. The symmetric reasoning together with (\ref{formxi})
finally give (\ref{granicaxi}).
\end{proof}

\begin{proofof}{Theorem}{theorem:formphi}
For every $\delta>0$ and $z\in\Sigma$ denote by $B(z,\delta)$ the
closed ball  of radius $\delta$ and centered at $z$ in singular
adapted coordinates. Next choose  $\delta >0$ so that intervals
$[l_\alpha-\delta^2,l_\alpha+\delta^2]$, $\alpha\in\mathcal{A}$
are pairwise disjoint and $B(z,\delta)\cap \J=\emptyset$ for all
$z\in\Sigma$. For every $z\in\Sigma$ denote by $\mathcal{O}_z$ the
corresponding orbit in $\Sigma(\pi)$. For simplicity assume that
$|I|=1$.

We split the proof into several parts. In each of them we will
assume that $f$ is supported on a part of the surface $\Surf$. Then we
will collect together all parts to prove the theorem in full
generality.

\subsubsection*{Non-triviality on a neighborhood only one singularity.}
First fix $z\in\Sigma$ and assume that $f:\Surf\to\R$ is a
$\mathscr{C}^{2+\epsilon}$ function which vanishes on
$\Surf\setminus B(z,\delta)$. Recall that  each point $l_\alpha$,
$\alpha\neq\underline{\alpha}=\pi_1^{-1}(1)$ corresponds to the
first backward intersection with $\J$ of an incoming separatrix of
a fixed point, this fixed point will be denoted  by $
z_{l_{\alpha}}\in\Sigma$.

\subsubsection*{Regular case.} Now suppose that $z\neq z_{l_{\pi^{-1}_0(1)}}$.
Then there exist two distinct elements $\alpha_0$, $\alpha_1\in
\mathcal{A}$ such that $z=z_{l_{\alpha_0}}=z_{l_{\alpha_1}}$ and
$\mathcal{O}_z=\{\pi_0(\alpha_0)-1,\pi_0(\alpha_1)-1\}$. Let
$\zeta=x+iy$ be the singular adapted coordinate  around $z$. Then
there exists a positive $\mathscr{C}^\infty$-function
$V:[-\delta,\delta]\times[-\delta,\delta]\to\R$ such that
$X(\zeta)=V(x,y)(x,-y)$ and $\omega=\frac{dx\wedge\,dy}{V(x,y)}$
on $[-\delta,\delta]\times[-\delta,\delta]$. Moreover,
\[\gamma_{\pm}^v,\gamma_{\pm}^h:[-\delta^2,\delta^2]\to \Surf,\;\;\;
\gamma_{\pm}^h(s)=(\pm s/\delta,\pm\delta),\;\;\;
\gamma_{\pm}^v(s)=(\pm\delta,\pm s/\delta)\] establishes an
induced parameterization  of the boundary of the square
$[-\delta,\delta]\times[-\delta,\delta]$. Let us consider the
functions $\tau^{\pm}_\alpha:[-\delta^2,0)\cup(0,\delta^2]\to\R_+$
such that $\tau^{\pm}_\alpha(s)$ is the  exit time  of the point
$(\pm s/\delta,\pm\delta)$ for the flow $({\phi}_{t})$ from the
set $[-\delta,\delta]\times[-\delta,\delta]$. Since the positive
orbit of $l_{\alpha_\epsilon}$, $\epsilon=0,1$, hits the square
$[-\delta,\delta]\times[-\delta,\delta]$ at
$((-1)^\epsilon\delta,0)$ and $f$ vanishes on $\Surf\setminus
([-\delta,\delta]\times[-\delta,\delta])$, the function
$\varphi_f$ vanishes on
$I\setminus\left([l_{\alpha_0}-\delta^2,l_{\alpha_0}+
\delta^2]\cup[l_{\alpha_1}-\delta^2,l_{\alpha_1}+\delta^2]\right)$
and
\[\varphi_f(s+l_{\alpha_\epsilon})=
\int_{0}^{\tau^{{(-1)^\epsilon}}_{\alpha_\epsilon}(s)}
f(\phi_t({(-1)^\epsilon} s/\delta,{(-1)^\epsilon}\delta))\,dt
\;\;\text{ for }s\in[-\delta^2,\delta^2]\text{ and
}\epsilon=0,1.\] Fix $\epsilon\in\{0,1\}$ and let
$(x_t,y_t)=\phi_t({(-1)^\epsilon}
s/\delta,{(-1)^\epsilon}\delta)$. Then
\begin{equation*}
\left(\frac{d}{dt}x_t,\frac{d}{dt}y_t\right)=X(x_t,y_t)=V(x_t,y_t)(x_t,-y_t),
\end{equation*}
and hence
\[\frac{d}{dt}(x_t\cdot y_t)=y_t\frac{d}{dt}x_t+x_t\frac{d}{dt}y_t=0.\]
Therefore \[x_ty_t=x_0y_0=s.\] Since $s\neq 0$, it follows that
$x_t\neq 0$ for all $t\in\R$. By using the substitution $u = x_t$,
we obtain $du = \frac{d}{dt}x_t dt=V(x_t,s/x_t)x_t dt$ and
\begin{align*}
\varphi_f&(s+l_{\alpha_\epsilon})=\int_{0}^{\tau^{{(-1)^\epsilon}}_{\alpha_\epsilon}(s)}f
(x_t,y_t)dt=\int_{0}^{\tau^{{(-1)^\epsilon}}_{\alpha_\epsilon}(s)}f\left(x_t,\frac{s}{x_t}\right)dt
\\&=\int_{{(-1)^\epsilon}
s/\delta}^{{(-1)^\epsilon}\sgn(s)\delta}\frac{f\left(u,\frac{s}{u}\right)}{V\left(u,\frac{s}{u}\right)}\frac{du}{u}
=\int_{|s|/\delta^2}^1\frac{f}{V}\left({(-1)^\epsilon}\sgn(s)\delta
u,\frac{|s|/\delta^2}{{(-1)^\epsilon}\delta u}
\right)\frac{du}{u}.
\end{align*}
In view of Lemma~\ref{lematprzej},
\[
\varphi_f(s)=-C_{\alpha_\epsilon}\log|s-l_{\alpha_\epsilon}|+\widetilde{\xi}_\epsilon(s),\;\;\;
\widetilde{\xi}_\epsilon(s)=
-K_{\alpha_\epsilon}(s-l_{\alpha_\epsilon})\log|s-l_{\alpha_\epsilon}|+\xi_\epsilon(s)
\]
where
$\xi_\epsilon:[l_{\alpha_\epsilon}-\delta^2,
l_{\alpha_\epsilon}+\delta^2]\setminus\{l_{\alpha_\epsilon}\}\to\R$
is a function which is absolutely continuous  with  absolutely
continuous derivative,
\[
\var\widetilde{\xi}_\epsilon|_{[l_{\alpha_\epsilon}-\delta^2,l_{\alpha_\epsilon})}+
\var\widetilde{\xi}_\epsilon|_{(l_{\alpha_\epsilon},l_{\alpha_\epsilon}+\delta^2]}\leq
C_V\|f\|_{\mathscr{C}^2}\] and
\[C_{\alpha_\epsilon}=C_z:=\frac{f(0,0)}{V(0,0)},\;\;
K_{\alpha_\epsilon}=K_z:=\frac{\partial^2(f/V)}{\partial
x\,\partial y}(0,0).
\]
Therefore
\begin{align*}
\varphi_f(x)=&-C_{z}\sum_{\epsilon=0,1}
\left(\log\{x-l_{\alpha_\epsilon}\}+\log\{l_{\alpha_\epsilon}-x\}\right)+g(x),
\ \  \text{where}\\
g(x)=&-K_{z}\sum_{\epsilon=0,1}
\left(\{x-l_{\alpha_\epsilon}\}(\log\{x-l_{\alpha_\epsilon}\}-1)-
\{l_{\alpha_\epsilon}-x\}(\log\{l_{\alpha_\epsilon}-x\}-1)\right)\\&+g_0(x)
\end{align*}
and $g_0:I\to\R$ is absolutely continuous with absolutely
continuous derivative on
$I\setminus\{l_{\alpha_0},l_{\alpha_1}\}$, so $g_0,
g_0'\in\ac(\sqcup_{\alpha\in \mathcal{A}} I_{\alpha})$. Moreover,
$g\in\ac(\sqcup_{\alpha\in \mathcal{A}} I_{\alpha})$ and $g(x)$ is
equal to
\[C_{z}\sum_{\epsilon=0,1}\left(\log\{x-l_{\alpha_\epsilon}\}+
\log\{l_{\alpha_\epsilon}-x\}\right)\text{
if } x\in
I\setminus\bigcup_{\epsilon=0,1}[l_{\alpha_\epsilon}-
\delta^2,l_{\alpha_\epsilon}+\delta^2]
\]
\begin{align*}
&C_z\left(\log\{l_{\alpha_\epsilon}-x\}+
\log\{x-l_{\alpha_{1-\epsilon}}\}+\log\{l_{\alpha_{1-\epsilon}}-x\}\right)+
\widetilde{\xi}_\epsilon(x)\text{ if }
x\in[l_{\alpha_\epsilon},l_{\alpha_\epsilon}+\delta^2]\\ &
C_z\left(\log\{x-l_{\alpha_\epsilon}\}+
\log\{x-l_{\alpha_{1-\epsilon}}\}+\log\{l_{\alpha_{1-\epsilon}}-x\}\right)
+\widetilde{\xi}_\epsilon(x)\text{ if }
x\in[l_{\alpha_\epsilon}-\delta^2,l_{\alpha_\epsilon}].
\end{align*}
For $\epsilon=0,1$. It follows that
\begin{align*}\var g\leq&
4|C_z|\var(\log)|_{[\delta^2,1]}+\sum_{\epsilon=0,1}\left(
\var\widetilde{\xi}_\epsilon|_{[l_{\alpha_\epsilon}-\delta^2,l_{\alpha_\epsilon})}+
\var\widetilde{\xi}_\epsilon|_{(l_{\alpha_\epsilon},l_{\alpha_\epsilon}+\delta^2]}\right)
\\ \leq&4\frac{\|f\|_{\mathscr{C}^0}}{V(z)}\log\delta^{-2}+2C_V\|f\|_{\mathscr{C}^2}\leq
C_{\delta,V}\|f\|_{\mathscr{C}^2}.
\end{align*}
Finally note that $\varphi_f$ and $g$ can be represented as
follows
\[
\varphi_f(x)=-\sum_{\pi_0(\alpha)-1\in\mathcal{O}_z}C_{\alpha}^+\log\{x-l_{\alpha}\}-
\sum_{\pi_0(\alpha)\in\mathcal{O}_z}C_{\alpha}^-\log\{r_{\alpha}-x\}+g(x),
\]
where
\begin{align*}g(x)=&\;g_0(x)-
\sum_{\pi_0(\alpha)-1\in\mathcal{O}_z}K^+_{\alpha}\{x-l_{\alpha}\}(\log\{x-l_{\alpha}\}-1)\\&+
\sum_{\pi_0(\alpha)\in\mathcal{O}_z}
K^-_{\alpha}\{r_{\alpha}-x\}(\log\{r_{\alpha}-x\}-1)
\end{align*}
with $C_{\alpha}^+=C_z$, $K_{\alpha}^+=K_z$ if
$\pi_0(\alpha)-1\in\mathcal{O}_z$ and $C_{\alpha}^-=C_z$,
$K_{\alpha}^-=K_z$ if $\pi_0(\alpha)\in\mathcal{O}_z$. It follows
that (\ref{zerosym}) is valid  for $\mathcal{O}=\mathcal{O}_z$.
For $\mathcal{O}\neq\mathcal{O}_z$ the condition (\ref{zerosym})
holds trivially.

\subsubsection*{Exceptional  case.} Now assume that $z=z_{l_{\pi^{-1}_0(1)}}$.
Denote by $\alpha_0\neq \pi^{-1}_0(1)$ an element of the alphabet
for which $z=z_{\alpha_0}$.  Then $\mathcal{O}_z=\{0,
\pi_0(\alpha_0)-1,\pi_0(\underline{\alpha})-1\}$. Since
$l_{\pi^{-1}_0(1)}$ and $l_{\underline{\alpha}}$ lie on the same
incoming  separatrix of $z$, similar arguments to those used in
the regular case show that there exists $g_0,
\in\ac(\sqcup_{\alpha\in \mathcal{A}} I_{\alpha})$ with
$g_0'\in\ac(\sqcup_{\alpha\in \mathcal{A}} I_{\alpha})$ such that
\begin{align*}
\varphi_f(x)=&-C_z\left(\log\{x\}+\log\{l_{\underline{\alpha}}-x\}+
\log\{x-l_{\alpha_0}\}+\log\{l_{\alpha_0}-x\}\right)+g(x)\\
=&-\sum_{\pi_0(\alpha)-1\in\mathcal{O}_z}C_{\alpha}^+\log\{x-l_{\alpha}\}-
\sum_{\pi_0(\alpha)\in\mathcal{O}_z}C_{\alpha}^-\log\{r_{\alpha}-x\}+g(x),
\end{align*}
where
\begin{align*}
g(x)=&\;g_0(x)-K_z\left(\{x\}(\log\{x\}-1)-
\{l_{\underline{\alpha}}-x\}(\log\{l_{\underline{\alpha}}-x\}-1)\right.\\
&+\left. \{x-l_{\alpha_0}\}(\log\{x-l_{\alpha_0}\}-1)-
\{l_{\alpha_0}-x\}(\log\{r_{\alpha_0}-x\}-1)\right)\\
=&\;g_0(x)-\sum_{\pi_0(\alpha)-1\in\mathcal{O}_z}K^+_{\alpha}\{x-l_{\alpha}\}(\log\{x-l_{\alpha}\}-1)\\&+
\sum_{\pi_0(\alpha)\in\mathcal{O}_z}
K^-_{\alpha}\{r_{\alpha}-x\}(\log\{r_{\alpha}-x\}-1),
\end{align*}
with $C_{\alpha}^+=C_z$, $K_{\alpha}^+=K_z$ if $\alpha\neq
\underline{\alpha}$ and $\pi_0(\alpha)-1\in\mathcal{O}_z$;
$C_{\underline{\alpha}}^+=K_{\underline{\alpha}}^+=0$;
$C_{\alpha}^-=C_z$, $K_{\alpha}^-=K_z$ if
$\pi_0(\alpha)\in\mathcal{O}_z$; and $\var g\leq
C_{\delta,V}\|f\|_{\mathscr{C}^2}$.

\subsubsection*{Vanishing around singularities.} We will now deal  with
the case where $f$ vanishes on each ball $B(z,\delta/2)$,
$z\in\Sigma$. For every $\alpha\in\mathcal{A}$ denote by
$h_\alpha>0$ the first return time of points in $I_\alpha$ to $I$
for the vertical flow $(F^v_t)_{t\in\R}$ and set
$\bar{h}=(h_\alpha)_{\alpha\in\mathcal{A}}$. Since
$\phi_tx=F^v_{h(t,x)}x$ and $W(\phi_t x)=\frac{\partial
h}{\partial t}(t,x)$, we have $h(\tau(x),x)=h_\alpha$ for each
$x\in I_\alpha$. Then using the substitution $s=h(t,x)$, for each
$x\in I_\alpha$ we get
\[
\varphi_f(x)=\int_0^{\tau(x)}f(\phi_t(x))\,dt
=\int_0^{h_\alpha}\frac{f(F^v_s(x))}{W(F^v_s(x))}\,ds.
\]
The function $W:\Surf\to\R$ is positive $\mathscr{C}^\infty$ with
zeros only at $\Sigma$. Therefore $c_\delta:=\min\left\{W(x):x\in
\Surf\setminus \bigcup_{z\in\Sigma}B(z,\delta/2)\right\}>0$. Moreover,
$f/W:\Surf\to\R$ is a $\mathscr{C}^\infty$-function with
\[\|f/W\|_{\mathscr{C}^0}\leq c_\delta^{-1}\|f\|_{\mathscr{C}^0}\text{ and }
\|f/W\|_{\mathscr{C}^1}\leq
c_\delta^{-2}\|W\|_{\mathscr{C}^1}\|f\|_{\mathscr{C}^1}.\] It
follows that $\varphi_f$ can be extended to a
$\mathscr{C}^\infty$-function on each $\overline{I}_\alpha$,
$\alpha\in\mathcal{A}$,
\[\|\varphi_f\|_{\mathscr{C}^0}\leq \max\{h_\alpha:\alpha\in\mathcal{A}\}
\|f/W\|_{\mathscr{C}^0}\leq
\|\bar{h}\|c_\delta^{-1}\|f\|_{\mathscr{C}^0}\] and
\begin{align*}
\var\varphi_f&=\int_I|\varphi_f'(u)|\,du=
\sum_{\alpha\in\mathcal{A}}\int_{I_\alpha}
\left|\int_0^{h_\alpha}\frac{\partial}{\partial
y}(f/W)(F^v_s(x))\,ds\right|\,du\\&\leq
\langle\lambda,\bar{h}\rangle
c_\delta^{-2}\|W\|_{\mathscr{C}^1}\|f\|_{\mathscr{C}^1}.
\end{align*}
Hence $\varphi_f, \varphi'_f\in \ac(\sqcup_{\alpha\in \mathcal{A}}
I_{\alpha})$ and there exists a positive constant  $C_\ast$ such
that $\|\varphi_f\|_{\bv}\leq C_\ast\|f\|_{\mathscr{C}^1}$ for
each $f:\Surf\to\R$ vanishing on
$\bigcup_{z\in\Sigma}B(z,\delta/2)$. Since $\varphi$ has no
logarithmic singularities,  the condition (\ref{zerosym}) holds
trivially.

\subsubsection*{General case.} Let us consider a
$\mathscr{C}^\infty$-partition of unity
$\{\rho_z:z\in\Sigma\cup\{\ast\}\}$ of $\Surf$ such that $\rho_z$
vanishes on $\Surf\setminus B(z,\delta)$ for all $z\in\Sigma$ and
$\rho_\ast$ vanishes on $\bigcup_{z\in\Sigma}B(z,\delta/2)$. Since
the balls $B(z,\delta)$, $z\in\Sigma$ are pairwise disjoint,
$\rho_z\equiv 1$ on $B(z,\delta/2)$ for each $z\in\Sigma$. Let us
decompose $\varphi_f$ as follows
$\varphi_f=\sum_{z\in\Sigma}\varphi_{\rho_z\cdot
f}+\varphi_{\rho_\ast\cdot f}$. In view of all facts that have
been proved until now  for all $z\in\Sigma$ we get
\begin{align}\label{dosum1}
\varphi_{\rho_z\cdot
f}(x)=-\!\sum_{\pi_0(\alpha)-1\in\mathcal{O}_z}C_{\alpha}^+\log\{x-l_{\alpha}\}-\!
\sum_{\pi_0(\alpha)\in\mathcal{O}_z}C_{\alpha}^-\log\{r_{\alpha}-x\}+g_z(x),
\end{align}
where
\begin{align}
\begin{split}\label{dosum2}
g_z(x) =\;g_{z,0}(x)&-\sum_{\pi_0(\alpha)-1\in\mathcal{O}_z}
K^+_{\alpha}\{x-l_{\alpha}\}(\log\{x-l_{\alpha}\}-1)
\\ &+ \sum_{\pi_0(\alpha)\in\mathcal{O}_z}
K^-_{\alpha}\{r_{\alpha}-x\}(\log\{r_{\alpha}-x\}-1),
\end{split}
\end{align}
with $g_{z,0}, g'_{z,0}\in \ac(\sqcup_{\alpha\in
\mathcal{A}} I_{\alpha})$ and \[\|g_z\|_{\bv} \leq
C_{\delta,V}\|\rho_z\cdot f\|_{\mathscr{C}^2}\leq
C_{\delta,V}\|\rho_z\|_{\mathscr{C}^2} \|f\|_{\mathscr{C}^2}.\]
Moreover, $\varphi_{\rho_\ast\cdot f},
\varphi'_{\rho_\ast\cdot f}\in \ac(\sqcup_{\alpha\in
\mathcal{A}} I_{\alpha})$ and
\[\|\varphi_{\rho_\ast\cdot
f}\|_{\bv} \leq C_\ast\|\rho_\ast\cdot
f\|_{\mathscr{C}^2}\leq
C_\ast\|\rho_\ast\|_{\mathscr{C}^2}
\|f\|_{\mathscr{C}^2}.\] Let
\[g:=\sum_{z\in\Sigma}g_z+\varphi_{\rho_\ast\cdot
f},\text{
}g_2:=\sum_{z\in\Sigma}g_{z,0}+\varphi_{\rho_\ast\cdot
f},\;\;g_1=g-g_2\text{ and
}C^-_{\overline{\alpha}}=K^-_{\overline{\alpha}}=0.\] Then
$g_1,g_2,g_2'\in\ac(\sqcup_{\alpha\in \mathcal{A}} I_{\alpha})$
and
\[
\|g\|_{\bv}\leq
\left(\sum_{z\in\Sigma}C_{\delta,V}\|\rho_z\|_{\mathscr{C}^2}+
C_\ast\|\rho_\ast\|_{\mathscr{C}^2}\right)
\|f\|_{\mathscr{C}^2}.\] Since
\[\bigsqcup_{z\in\Sigma}\{\alpha:\pi_0(\alpha)-1\in\mathcal{O}_z\}
=\mathcal{A}\quad\text{ and
}\quad\bigsqcup_{z\in\Sigma}
\{\alpha:\pi_0(\alpha)\in\mathcal{O}_z\}=
\mathcal{A}\setminus\{\underline{\alpha}\},
\]
summing up (\ref{dosum1}) and (\ref{dosum2}) over $z \in \Sigma$, we get
\begin{align*}
\varphi(x)=&-\sum_{\alpha\in\mathcal{A}}
\left(C^+_\alpha\log\{x-l_\alpha\}+
C^-_\alpha\log\{r_\alpha-x\}\right)+g(x)\\
g'_1(x)=&-\sum_{\alpha\in\mathcal{A}}
\left(K^+_\alpha\{x-l_\alpha\}\log\{x-l_\alpha\}
+K^-_\alpha\{r_\alpha-x\}\log\{r_\alpha-x\}\right).
\end{align*}
Since the condition (\ref{zerosym})  holds for each function
$\varphi_{\rho_z\cdot f}$ and $\varphi_{\rho_\ast\cdot f}$ has
no logarithmic singularities, (\ref{zerosym}) is valid also for
$\varphi_f$. The same applies to $g'_1$. Moreover,
$C^+_{\underline{\alpha}}=C^-_{\overline{\alpha}}=0$ and
\[
C^+_\alpha=f/V(z)\text{ if }\alpha\neq\underline{\alpha},\;
\pi_0(\alpha)-1\in\mathcal{O}_z\text{ and }C^-_\alpha=f/V(z)\text{
if }\pi_0(\alpha)\in\mathcal{O}_z.
\]
Therefore,
\[
\bl(\varphi_f)=\sum_{\alpha\in\mathcal{A}}(|C^-_\alpha|+|C^+_\alpha|)
=4\sum_{z\in\Sigma}\frac{|f(z)|}{V(z)}.
\]
Since $V$ takes only positive values, it follows that
\[\frac{4}{\max\{V(z):z\in\Sigma\}}\sum_{z\in\Sigma}|f(z)|\leq \bl(\varphi_f)
\leq \frac{4}{\min\{V(z):z\in\Sigma\}}\sum_{z\in\Sigma}|f(z)|.\]
\end{proofof}

\begin{proofof}{Proposition}{vanishO:prop}
By Theorem~\ref{theorem:formphi},
$\varphi_f\in\ac(\sqcup_{\alpha\in \mathcal{A}} I_{\alpha})$. For
every two points $x_1,x_2\in \Surf$ such that $x_1=\phi_ux_0$ and
$x_2=\phi_vx_0$ for some $-\infty\leq u\leq v\leq+\infty$ and
$x_0\in S\setminus\Sigma$ let
$I(x_1,x_2)=\int_{u}^{v}f(\phi_sx_0)\,ds$. In view of
(\ref{granicaxi}), analysis similar to that in the proof of
Theorem~\ref{theorem:formphi} shows that
\[\lim_{s\to
l_\alpha^+}\varphi_f(s)=\left\{
\begin{array}{ccc}
I(l_\alpha,z_{l_\alpha})+I(z_{l_\alpha},Tl_\alpha)&\text{ if }&\pi_1(\alpha)\neq 1\\
I(l_\alpha,Tl_\alpha)&\text{ if }&\pi_1(\alpha)=1
\end{array}\right.
\]
\[\lim_{s\to
r_\alpha^-}\varphi_f(s)=\left\{
\begin{array}{ccc}I(r_\alpha,z_{r_\alpha})+I(z_{r_\alpha},\widehat{T}r_\alpha)
&\text{ if }&\pi_0(\alpha)\neq d\\
I(r_\alpha,\widehat{T}r_\alpha)&\text{ if
}&\pi_0(\alpha)=d.\end{array}\right.
\]
Therefore, for every $\alpha\in\mathcal{A}$ with
$\pi_1(\alpha)\neq 1$ and $\pi_0(\alpha)\neq 1,d$ we have
\[\lim_{s\to l_\alpha^+}\varphi_f(s)-\lim_{s\to l_\alpha^-}\varphi_f(s)=
I(z_{l_\alpha},Tl_\alpha)-I(z_{l_\alpha},\widehat{T}l_\alpha).\]
Take $\mathcal{O}=\mathcal{O}_z$ which does not contain $0$ and
$d$. Let $\alpha_0,\alpha_1$ be distinct elements of the alphabet
for which $z_{l_{\alpha_0}}=z_{l_{\alpha_1}}=z$. Then
$\mathcal{O}=\{\pi_0(\alpha_0)-1,\pi_0(\alpha_1)-1\}$ and
$Tl_{\alpha_{\epsilon}}=\widehat{T}l_{\alpha_{1-\epsilon}}$ for
$\epsilon=0,1$. In view of \eqref{odlabv}, it follows that
\[\mathcal{O}(\varphi_f)=
\sum_{\epsilon=0,1}\left(\lim_{s\to
l_{\alpha_\epsilon}^-}\varphi_f(s)-\lim_{s\to
l_{\alpha_\epsilon}^+}\varphi_f(s)\right)
=\sum_{\epsilon=0,1}\left(I(z,\widehat{T}l_{\alpha_\epsilon})-
I(z,Tl_{\alpha_\epsilon})\right)=0.\] Similar arguments to those
above show also that $\mathcal{O}(\varphi_f)= 0$ if
$0\in\mathcal{O}$ or $d\in\mathcal{O}$.
\end{proofof}

\section{Cohomological reduction}\label{cohreduction:sec}
In this Appendix we  prove Proposition \ref{corcohred}. Denote by
$\ac_0^s(\sqcup_{\alpha\in \mathcal{A}} I^{(0)}_{\alpha})$ the
subspace of all $\varphi\in\ac_0(\sqcup_{\alpha\in \mathcal{A}}
I^{(0)}_{\alpha})$ such that
$\varphi'\in\overline{\ls}_0(\sqcup_{\alpha\in \mathcal{A}}
I^{(0)}_{\alpha})$ and $\mathfrak{h}(\varphi')=0$. In view of
Theorem~\ref{operatorcorrection}, for every $\varphi\in
\ac_0^s(\sqcup_{\alpha\in \mathcal{A}} I^{(0)}_{\alpha})$ and
$k\geq 1$,
\begin{equation}\label{znowuvar}
\var(S(k)\varphi)\leq |I^{(k)}|k^{M}\left(C_1
\lv(\varphi')+C_2\var\varphi/|I^{(0)}|\right).
\end{equation}
Denote by
\[\widetilde{U}^{(k)}:\ac_0(\sqcup_{\alpha\in \mathcal{A}} I^{(k)}_{\alpha})\to
\ac_0(\sqcup_{\alpha\in \mathcal{A}}
I^{(k)}_{\alpha})/\Gamma^{(k)}_{s}\] the projection on the
quotient space. Since  $S(k,k')\Gamma^{(k)}_{s}=\Gamma^{(k')}_{s}$
we can define the quotient linear transformation of $S(k,k')$,
\[S_{\flat}(k,k'):\ac_0(\sqcup_{\alpha\in \mathcal{A}} I^{(k)}_{\alpha})/\Gamma^{(k)}_{s}
\to\ac_0(\sqcup_{\alpha\in \mathcal{A}}
I^{(k')}_{\alpha})/\Gamma^{(k')}_{s}.\] Then
\begin{equation}\label{splatanies1}
S_{\flat}(k,k')\circ
\widetilde{U}^{(k)}\varphi=\widetilde{U}^{(k')}\circ
S(k,k')\varphi\text{ for }\varphi\in\ac_0(\sqcup_{\alpha\in
\mathcal{A}} I^{(k)}_{\alpha}).
\end{equation}
Moreover, $S_{\flat}(k,k'):\Gamma^{(k)}/\Gamma^{(k)}_{s}
\to\Gamma^{(k')}/\Gamma^{(k')}_{s}$ is invertible. Since $A^t$ on
$\Gamma^{(0)}/\Gamma^{(0)}_{s}$ is isomorphic to $A^t$ on
$\Gamma_c^{(0)}\oplus\Gamma^{(0)}_{u}$, we get
\begin{equation}\label{szaflat}
\|(S_{\flat}(k,k'))^{-1}(h+\Gamma^{(k')}_{s})\|\leq
C(k'-k)^{M-1}\|h+\Gamma^{(k)}_{s}\|\quad\text{ if }\quad k'>k.
\end{equation}

\begin{lemma}\label{lemma:tildep}
 The operator $\Delta\widetilde{P}^{(k)}:
\ac_0^s(\sqcup_{\alpha\in \mathcal{A}} I^{(0)}_{\alpha})\to
\Gamma^{(k)}/\Gamma^{(k)}_{s}$,
\[\Delta\widetilde{P}^{(k)}=\sum_{r\geq k}(S_{\flat}(k,r+1))^{-1}\circ
\widetilde{U}^{(r+1)}\circ C^{(r+1)}\circ S(r,r+1)\circ
P_0^{(r)}\circ S(k,r)\] is well defined and $\|\Delta
\widetilde{P}^{(k)}\varphi\|\leq K\left(C_1
|I^{(k)}|\lv(\varphi')+C_2\var\varphi\right)$.
\end{lemma}

\begin{proof} In view of (\ref{znowuvar}), for $r\geq k$ we have
\begin{align*}\|P_0^{(r)}\circ S(k,r)(\varphi)\|_{\sup}&\leq
\var(S(k,r)(\varphi))\\&\leq (r-k+1)^{M}\left(|I^{(r)}|C_1
\lv(\varphi')+\frac{|I^{(r)}|}{|I^{(k)}|}C_2\var\varphi\right).
\end{align*}
Since $\|\widetilde{U}^{(r+1)}\|\leq 1$, $\|C^{(r+1)}\|\leq 1$,
$\|S(r,r+1)\|=\|A\|$ and $|I^{(r)}|=\rho_1^{-(r-k)}|I^{(k)}|$, by
(\ref{szaflat}),
\begin{multline*}
\|(S_{\flat}(k,r+1))^{-1}\circ \widetilde{U}^{(r+1)}\circ
C^{(r+1)}\circ S(r,r+1)\circ P_0^{(r)}\circ S(k,r)(\varphi)\|\\
\leq
(r+1-k)^{M-1}\rho_1^{-(r-k)}\|A\|(r-k+1)^{M}\left(C_1|I^{(k)}|
\lv(\varphi')+C_2\var\varphi\right).
\end{multline*}
It follows that $\Delta\widetilde{P}^{(k)}$ is well defined and
\[\|\Delta \widetilde{P}^{(k)}\varphi\|\leq K\left(C_1|I^{(k)}|
\lv(\varphi')+C_2\var\varphi\right),\] where $K=\sum_{j\geq
0}(j+1)^{2M}\rho_1^{-j}\|A\|$. This concludes the proof.
\end{proof}

Let $\widetilde{P}^{(k)}:\ac_0^s(\sqcup_{\alpha\in \mathcal{A}}
I^{(0)}_{\alpha})\to \ac^s_0(\sqcup_{\alpha\in \mathcal{A}}
I^{(0)}_{\alpha})/\Gamma^{(k)}_{s}$ be given by
\[\widetilde{P}^{(k)}=\widetilde{U}^{(k)}\circ P_0^{(k)}-\Delta \widetilde{P}^{(k)}.\]
Since $\|P_0^{(k)}\circ S(k)(\varphi)\|_{\sup}\leq
\var(S(k)(\varphi))\leq \var\varphi$ for every $\varphi\in
\bv(\sqcup_{\alpha\in \mathcal{A}} I^{(k)}_{\alpha})$, by
Lemma~\ref{lemma:tildep}, we get
\begin{equation}
\|\widetilde{P}^{(k)}\varphi\|_{\sup/\Gamma^{(k)}_{s}} \leq KC_1
|I^{(k)}|\lv(\varphi')+(KC_2+1)\var\varphi. \label{szapk1}
\end{equation}
Following the arguments in the proof of Lemma~\ref{commutes} for
all $0\leq k\leq k'$ and $\varphi\in \ac_0^s(\sqcup_{\alpha\in
\mathcal{A}} I^{(k)}_{\alpha})$ we get
\begin{equation}
S_{\flat}(k,k')\circ
\widetilde{P}^{(k)}\varphi=\widetilde{P}^{(k')}\circ
S(k,k')\varphi, \label{przem1}
\end{equation}

\begin{theorem}\label{thmcorrecver}
Assume that  $T$ is of periodic type. For every $\varphi\in
\ac_0^s(\sqcup_{\alpha\in \mathcal{A}} I^{(0)}_{\alpha})$ if
$\widehat{\varphi}+\Gamma^{(0)}_{s}= \widetilde{P}^{(0)}\varphi$
then $\widehat{\varphi}-\varphi\in\Gamma_0^{(0)}$ and there exist
$C'''_1,C'''_2,C'''_3>0$
\[\|S(k)\widehat{\varphi}\|_{\sup}
\leq
\exp(-k\theta_-)(C'''_1\lv(\varphi')+C'''_2\var\varphi+
C'''_3\|\widehat{\varphi}\|_{\sup}).\]
\end{theorem}
\begin{proof}
For simplicity, assume that $|I^{(0)}|=1$. Since
\[\widetilde{U}^{(0)}\widehat{\varphi}= \widetilde{P}^{(0)}\varphi=
\widetilde{U}^{(0)}\circ P_0^{(0)}\varphi-\Delta \widetilde{P}^{(0)}
\varphi=\widetilde{U}^{(0)}\varphi-\widetilde{U}^{(0)}\circ
C^{(0)}\varphi-\Delta \widetilde{P}^{(0)}\varphi,\] we have
$\varphi-\widehat{\varphi}\in \widetilde{U}^{(0)}\circ C^{(0)}\varphi+
\Delta \widetilde{P}^{(0)}\varphi \subset \Gamma_0^{(0)}.$
In view of (\ref{splatanies1}) and (\ref{przem1}),
\[\widetilde{U}^{(k)}\circ
S(k)\widehat{\varphi}=S_{\flat}(k)\circ
\widetilde{U}^{(0)}\widehat{\varphi} =S_{\flat}(k)\circ
\widetilde{P}^{(0)}\varphi=\widetilde{P}^{(k)}\circ S(k)\varphi.\]
Therefore, by (\ref{szapk1}),  (\ref{nave}) and (\ref{znowuvar}),
we have
\begin{align*}
\|\widetilde{U}^{(k)}&\circ
S(k)\widehat{\varphi}\|_{\sup/\Gamma^{(k)}_{s}}=
\|\widetilde{P}^{(k)}(S(k)\varphi)\|_{\sup/\Gamma^{(k)}_{s}}\\
&\leq KC_1 |I^{(k)}|\lv(S(k)(\varphi'))+(KC_2+1)\var(S(k)\varphi)\\
&\leq \max(1,k^M)|I^{(k)}|(C_1'\lv(\varphi')+C'_2\var(\varphi)).
\end{align*}
It follows that for every $k\geq 0$ there exists $\varphi_k\in
\ac_0^s (\sqcup_{\alpha\in \mathcal{A}} I^{(k)}_{\alpha})$ and
$h_k\in\Gamma^{(k)}_{s}$ such that
\begin{equation}\label{szcsk2}
S(k)\widehat{\varphi}=\varphi_k+h_k\text{,
}\|\varphi_k\|_{\sup}\leq \max(1,k^M)|I^{(k)}|
\left(C'_1\lv(\varphi')+C'_2\var\varphi\right).
\end{equation}
As
$\varphi_{k+1}+h_{k+1}=S(k+1)\widehat{\varphi}=
S(k,k+1)(S(k)\widehat{\varphi})=S(k,k+1)\varphi_k+A^th_k$,
setting $\Delta h_{k+1}=h_{k+1}-A^th_k$ ($\Delta h_0=h_0$) we have
$\Delta h_{k+1}=-\varphi_{k+1}+S(k,k+1)\varphi_k$. Moreover, by
(\ref{szcsk2}),
\begin{eqnarray*}
\|\Delta
h_{k+1}\|&=&\|\varphi_{k+1}-S(k,k+1)\varphi_k\|_{\sup}\leq\|\varphi_{k+1}\|_{\sup}+
\|S(k,k+1)\varphi_k\|_{\sup}\\&\leq&
(1+\|A\|)(k+1)^{M}|I^{(k+1)}|\left(
C'_1\lv(\varphi')+C'_2\var\varphi\right)
\end{eqnarray*}
and $\|\Delta h_{0}\|=\|\widehat{\varphi}-\varphi_0\|_{\sup}\leq
\|\widehat{\varphi}\|_{\sup}+(C'_1\lv(\varphi')+C'_2\var\varphi).$

Since $h_k=\sum_{0\leq l\leq k}(A^t)^{k-l}\Delta h_{l}$ and
$\Delta h_l\in\Gamma^{(k')}_{s}$, by (\ref{stab}),
\begin{align*}
\|h_k\|\leq&\sum_{0\leq l\leq k}\|(A^t)^{k-l}\Delta h_{l}\|\leq
\sum_{0\leq l\leq k}C\exp(-\theta_-(k-l))\|\Delta h_{l}\|\\
\leq&\;
C\exp(-\theta_-k)\left(\|\widehat{\varphi}\|_{\sup}+(C'_1\lv(\varphi')+
C'_2\var\varphi)\right)\\&+\sum_{1\leq l\leq
k}C\exp(-\theta_-(k-l)-\theta_1l)(1+\|A\|)l^M\left(C'_1\lv(\varphi')+
C'_2\var\varphi\right)\\
\leq
&\exp(-\theta_-k)(C_3''\|\widehat{\varphi}\|_{\sup}+C''_1\lv(\varphi')+
C''_2\var\varphi).
\end{align*}
In view of (\ref{szcsk2}), it follows that
\[\|S(k)\widehat{\varphi}\|_{\sup}\leq\|{\varphi}_k\|_{\sup}+\|h_k\|
\leq\;
\exp(-\theta_-k)(C'''_1\lv(\varphi')+C'''_2\var\varphi
+C'''_3\|\widehat{\varphi}\|_{\sup}).\]
\end{proof}
The following Proposition was proved in \cite{Ma-Mo-Yo}.
\begin{proposition}\label{propsn1}
For each bounded function $\varphi:I\to\R$, $x\in I$ and $n>0$ we
have
\begin{equation}\label{szacsn1}
|\varphi^{(n)}(x)|\leq
2\sum_{l\in \N}\|Z(l+1)\|\|S(l)\varphi\|_{\sup}.
\end{equation}
\end{proposition}

\begin{proofof}{Proposition}{theorem:coblog}
Since $\varphi'-Leb(\varphi')\in\overline{\ls}_0(\sqcup_{\alpha\in
\mathcal{A}} I_{\alpha})$, setting
$h:=\mathfrak{h}(\varphi'-Leb(\varphi'))\in\Gamma_0$, we have
$\mathfrak{h}(\varphi'-Leb(\varphi')-h)=0$. Choose $\varphi_1\in
\ac_0(\sqcup_{\alpha\in \mathcal{A}} I_{\alpha})$ so that
$\varphi_1'=\varphi'-Leb(\varphi')-h$. Then  $\varphi_1\in
\ac_0^s(\sqcup_{\alpha\in \mathcal{A}} I_{\alpha})$. In view of
Theorem~\ref{thmcorrecver}, there exist $h_1\in\Gamma_0$ and $C>0$
such that the function
$\varphi_2:=\varphi_1+h_1\in\ac_0(\sqcup_{\alpha\in \mathcal{A}}
I_{\alpha})$ satisfying
\[\|S(k)(\varphi_2)\|_{\sup}
\leq
C\exp(-\theta_-k)(\lv(\varphi_2')+\var\varphi_2+\|\varphi_2\|_{\sup}).\]
Therefore, by Proposition~\ref{propsn1}, for every $x\in I$ and
$n>0$,
\begin{align*}
|\varphi_2^{(n)}(x)|&\leq 2\sum_{l\geq
0}\|Z(l+1)\|\|S(l)\varphi_2\|_{\sup}\\&\leq
\frac{2\|A\|C}{1-\exp(-\theta_-)}(\lv(\varphi_2')+
\var\varphi_2+\|\varphi_2\|_{\sup}).
\end{align*}
In view of Proposition~\ref{ghmmy}, it follows that $\varphi_2$ is
a coboundary with a continuous transfer function. Let
$\psi:=\varphi-\varphi_2\in\ac_0(\sqcup_{\alpha\in \mathcal{A}}
I_{\alpha})$.
\[\psi'=\varphi'-\varphi'_1+(\varphi_1-\varphi_2)'=
\varphi'-(\varphi'-Leb(\varphi')-h)=Leb(\varphi')+h\in\Gamma.\] It
follows that $\psi\in\pl_0(\sqcup_{\alpha\in \mathcal{A}}
I_{\alpha})$. Since $h\in\Gamma_0$ and $\psi'=Leb(\varphi')+h$, we
also get $s(\psi)=Leb(\psi')=Leb(\varphi')=s(\varphi)$,   which
completes the proof.
\end{proofof}

\section{Reduction to skew product}\label{reduction}
In this Appendix we include for completeness the proof of Lemma
\ref{corollary:ergcob}.
\begin{proofof}{Lemma}{corollary:ergcob}
The first part is an  obvious consequence of Lemma~\ref{specrep},
since ergodicity is preserved by a measurable isomorphism and a
special flow is ergodic if and only if the base transformation is
ergodic.

Recall that the flow $(\Phi^f_t)_{t\in\R}$ is  reducible if it is
measure-theoretically isomorphic to the flow $(\Phi^0_t)_{t\in\R}$
via the map $\Surf\times\R\ni(x,y)\mapsto (x,y+G(x))\in
\Surf\times\R$, where $G:\Surf\to\R$ is a continuous function.
Reducibility is equivalent to the existence of a continuous
function $G:\Surf\to\R$ such that
\begin{equation}\label{redcoh}
F(t,x)=\int_{0}^tf(\phi_sx)\,ds=G(x)-G(\phi_tx)\text{ for all
}t\in\R\text{ and }x\in \Surf.
\end{equation}
Then for each $x\in I$ we have
\[\varphi_f(x)=F(\tau(x),\gamma(x))=G(\gamma(x))-G(\phi_{\tau(x)}\gamma(x))=
G\circ\gamma(x)-G\circ\gamma(Tx).\] It follows that $g:I\to\R$,
$g=G\circ\gamma$ is continuous and $\varphi=g-g\circ T$.

Suppose that $g:I\to\R$ is a continuous function such that
$\varphi_f=g-g\circ T$. Recall that for every $x\in
\Surf\setminus\Sigma$ the $(\phi_t)_{t\in\R}$ orbit of $x$ is dense
and intersects the cross section $I$. If $\phi_tx\in I$ for some
$t\in\R$ then set
\[G(x):=g(\phi_tx)+F(t,x)=g(\phi_tx)+\int_{0}^tf(\phi_sx)\,ds.\]
Notice that the function $G:\Surf\setminus\Sigma\to\R$ is well
defined. Indeed, if $\phi_{t_1}x,\phi_{t_2}x\in I$ with $t_1<t_2$
then $t_2-t_1=\tau^{(m)}(\phi_{t_1}x)$ and
$T^m\phi_{t_1}x=\phi_{t_2}x$. Therefore,
\begin{align*}
F(t_2,x)-F(t_1,x)&=F(t_2-t_1,\phi_{t_1}x)=F(\tau^{(m)}(\phi_{t_1}x),\phi_{t_1}x)\\
&=\varphi_f^{(m)}(\phi_{t_1}x)
=g(\phi_{t_1}x)-g(T^m\phi_{t_1}x)=g(\phi_{t_1}x)-g(\phi_{t_2}x).
\end{align*}
Thus $g(\phi_{t_1}x)+F(t_1,x)=g(\phi_{t_2}x)+F(t_2,x)$.

Note that by the definition of $G$ for every $x\in \Surf\setminus \Sigma$ and
$t\in\R$ we have
$G(x)-G(\phi_t x)=F(t,x)$.

In order to prove that $G:\Surf\setminus\Sigma\to\R$ is
continuous and can be extended to a continuous $G:\Surf \to\R$,
let us consider the oscillation function $\omega:\Surf\to\R_+$
defined at each $x\in \Surf$ by
\[\omega(x)=\lim_{\vep\to 0}\sup\{|G(y)-G(y')|:y,y'\in B(x,\vep)\setminus\Sigma\}.\]
Since $G(\phi_sx)=G(x)-F(s,x)$, $F$ is continuous and $\phi_s$ is
a diffeomorphism on $\Surf$, $\omega(\phi_sx)=\omega(x)$ for every
$x\in \Surf$ and $s\in\R$. Let $x\in \Surf\setminus\Sigma$. Since the
orbit of $x$ is dense and $\omega$ is upper semi-continuous, it
follows that $\omega(y)\geq\omega(x)$ for every $y\in \Surf$. By the
definition of $G$, each interior point $y$ of $I$ is a continuity
point of $G$. Therefore, $\omega(x)\leq \omega(y)=0$, so $G$ is
continuous at each $x\in \Surf\setminus\Sigma$.

To show that  $G$ can be continuously extended to $\Surf$, let us
prove that $\omega(z)=0$ for all $z\in\Sigma$.  Since $f(z)=0$ for
all $z\in\Sigma$, (\ref{redcoh}) will be trivially valid for all
$z\in\Sigma$. Fix $z_0\in\Sigma$ and let $\zeta=x+iy$ be the
singular adapted coordinate  around $z_0$. Let $\delta>0$ and
$V:[-\delta,\delta]\times[-\delta,\delta]\to\R_+$ be as in the
proof of Theorem~\ref{theorem:formphi} and set
$K:=\sup\{\|(f/V)'(z)\|:z\in[-\delta,\delta]\times[-\delta,\delta]\}$.
Since $G$ is continuous on $\Surf\setminus\Sigma$, for every
$\vep'>0$ there exists $0<\vep<\delta$ such that
$|G(s,\pm\delta)-G(s',\pm\delta)|<\vep'$ and
$|G(\pm\delta,s)-G(\pm\delta,s')|<\vep'$ for all
$s,s'\in[-\vep^2/\delta,\vep^2/\delta]$. We will prove that
\begin{equation}\label{oscsigma}
|G(z_1)-G(z_2)|\leq 3\vep'+18K\vep\text{ for all
}z_1,z_2\in([-\vep,\vep]\times[-\vep,\vep])\setminus\{(0,0)\},
\end{equation}
which yields $\omega(z_0)=0$.

By the proof of Theorem~\ref{theorem:formphi}, if
$(x_1,y_1),(x_2,y_2)\in([-\vep,\vep]\times[-\vep,\vep])\setminus\{(0,0)\}$
and $(x_2,y_2)=\phi_t(x_1,y_1)$ for some $t\in\R$ then
$x_1y_1=x_2y_2=s$ and
\begin{equation}\label{rozg}
G(x_1,y_1)-G(x_2,y_2)=\int_0^tf(\phi_v(x_1,y_2))\,dv=
\int_{y_1}^{y_2}(f/V)(s/u,u)\frac{du}{u}.
\end{equation}
It follows that for every $|s|\leq\vep$ we have
\[G(s,\vep)=G(s\vep/\delta,\delta)+
\int_\vep^\delta(f/V)(s\vep/u,u)\frac{du}{u}.\]
Hence if $s,s'\in[-\vep,\vep]$ then
\begin{align}\label{rozg2}
\begin{aligned}
|G(s,\vep)-G(s',\vep)|\leq\;&
|G(s\vep/\delta,\delta)-G(s'\vep/\delta,\delta)|\\
&+\int_\vep^\delta|(f/V)(s\vep/u,u)-(f/V)(s'\vep/u,u)|\frac{du}{u}\\
\leq\;&\vep'+\int_\vep^\delta K\frac{|s-s'|\vep}{u^2}du\leq
\vep'+K|s-s|\leq \vep'+2K\vep.
\end{aligned}
\end{align}
Let $D^+_{\pm}=\{(x,y):0<|x|\leq\pm y\leq\vep\}$ and
$D^-_{\pm}=\{(x,y):0<|y|\leq\pm x\leq\vep\}$. If $(x,y)\in D^+_+$
then, by (\ref{rozg}) and $(f/V)(0,0)=0$,
\[
|G(x,y)-G(xy/\vep,\vep)|\leq
\int_y^\vep|(f/V)(xy/u,u)|\frac{du}{u}\leq
K\int_y^\vep\left(\frac{|xy|}{u^2}+1\right)du\leq 2K\vep.
\]
In view of (\ref{rozg2}), for all $(x,y),(x',y')\in D^+_+$ we have
\begin{align*}
|G(x,y)-G(x',y')|\leq\;&
|G(x,y)-G(xy/\vep,\vep)|+|G(xy/\vep,\vep)-G(x'y'/\vep,\vep)|\\&+|G(x',y')-G(x'y'/\vep,\vep)|
\leq \vep'+6K\vep.
\end{align*}
The same applies to $D^+_-$, $D^-_+$ and $D^-_-$. This proves
(\ref{oscsigma}) and the proof is complete.
\end{proofof}

\end{document}